\documentclass[11pt]{article}
\usepackage{amsmath}
\usepackage{amsthm}
\usepackage{amsfonts}
\usepackage{latexsym}
\usepackage{graphicx}
\usepackage[pdftex,bookmarks=true,bookmarksnumbered=true]{hyperref}
\usepackage{color}
\renewcommand{\epsilon}{\varepsilon}
\renewcommand{\rho}{\varrho}
\renewcommand{\phi}{\varphi}
\newcommand{\DS}{\displaystyle}
\newcommand{\N}{{\mathbb N}}
\newcommand{\Z}{{\mathbb Z}}

\newcommand{\R}{{\mathbb R}}

\newcommand*\red{\color{red}}

\newcommand{\cA}{{\cal A}}
\newcommand{\cB}{{\cal B}}

\newcommand{\cD}{{\cal D}}
\newcommand{\cF}{{\cal F}}
\newcommand{\cG}{{\cal G}}
\newcommand{\cH}{{\cal H}}
\newcommand{\cI}{{\cal I}}
\newcommand{\cL}{{\cal L}}
\newcommand{\cM}{{\cal M}}
\newcommand{\cN}{{\cal N}}

\newcommand{\cS}{{\cal S}}

\newcommand{\cX}{{\cal X}}
\newcommand{\cY}{{\cal Y}}
\newcommand{\cZ}{{\cal Z}}
\newcommand{\ee}{{\mathbf e}}
\begin{document}
\newtheorem{definition}{Definition}[section]
\newtheorem{theorem}[definition]{Theorem}
\newtheorem{proposition}[definition]{Proposition}
\newtheorem{corollary}[definition]{Corollary}
\newtheorem{lemma}[definition]{Lemma}
\newtheorem{remark}[definition]{Remark}
\newtheorem{conjecture}[definition]{Conjecture}
\newtheorem{problem}[definition]{Problem}
\newtheorem{example}[definition]{Example}
\title{Equilibrium Validation for Triblock Copolymers via
       Inverse Norm Bounds for Fourth-Order Elliptic Operators}
\author{{\em Peter Rizzi, Evelyn Sander, and Thomas Wanner} \\[2ex]
        Department of Mathematical Sciences \\
        George Mason University \\
        Fairfax, VA 22030, USA }
%
\date{July 18, 2022}
\maketitle
%
%
\begin{abstract} 
  Block copolymers play an important role in materials sciences and have
  found widespread use in many applications. From a mathematical perspective,
  they are governed by a nonlinear fourth-order partial differential equation
  which is a suitable gradient of the Ohta-Kawasaki energy. While the equilibrium
  states associated with this equation are of central importance for the 
  description of the dynamics of block copolymers, their mathematical study
  remains challenging. In the current paper, we develop computer-assisted 
  proof methods which can be used to study equilibrium solutions in block copolymers 
  consisting of more than two monomer chains, with a focus on triblock
  copolymers. This is achieved by establishing a computer-assisted proof
  technique for bounding the norm of the inverses of certain fourth-order
  elliptic operators, in combination with an application of a constructive
  version of the implicit function theorem. While these results are only
  applied to the triblock copolymer case, the obtained
  norm estimates can also be directly used in other contexts such as the
  rigorous verification of bifurcation points, or pseudo-arclength
  continuation in fourth-order parabolic problems.
  

\end{abstract}
\newpage
\tableofcontents
%
%
%
\section{Introduction}
\label{secintro}
Block copolymers are materials formed from a number of different polymer molecules
which are connected together in polymer chains. By combining monomers with different
physical properties, one can create materials with completely new properties. For
example, thermoplastic elastomers are a type of diblock copolymer, which combine
rubbery monomers, such as polybutadiene or polyisoprene, with glassy hard monomers,
such as polystyrene. Based on these two competing properties, one obtains a compound
material that can be molded at high temperatures, but behaves as a rubber at low
temperatures. Such block copolymers are used in a number of commercial applications,
such as in sealants, gasket materials, hotmelt adhesives, rubber bands, toy products,
shoe soles, and even in road paving and roofing applications.

From a physical perspective, the study of diblock copolymers, which consist of
exactly two different monomers, was initiated by Ohta and Kawasaki~\cite{ohta:kawasaki:86a}.
They proposed a free energy functional for the description of such systems, which
extends the standard van der Waals free energy functional~\cite{vanderwaals:93a}
by a nonlocal term. This addition models the competition between both
long-range and short-range forces in the material, and it causes {\em microphase
separation}, which in turn results in pattern formation on a mesoscopic scale.
The observed pattern morphology is extremely rich, and one can observe complicated
microstructure such as gyroids, perforated layers, and more. See for
example~\cite{batesf:99a} and the references therein.

Mathematical studies of diblock copolymers have focused on dynamical aspects of
these materials. As already developed in~\cite{ohta:kawasaki:86a}, one can 
associate gradient dynamics with the Ohta-Kawasaki energy functional, which leads
to a fourth-order nonlinear parabolic partial differential equation. Being dissipative,
this equation has a global attractor, and its structure is responsible for the
long-term dynamics of the model~\cite{nishiura:02a, teramoto:nishiura:10a}. Of
particular interest is the set of equilibrium states of the diblock copolymer model.
While the numerous local and global energy minimizers of the Ohta-Kawasaki energy
describe potential long-term limits, even the saddles of the energy play an important
role in the selection of specific stable states~\cite{choksi:peletier:09a,
cyranka:wanner:18a, johnson:etal:13a, lessard:sander:wanner:17a, wanner:16a}.
There have been a number of studies which proved the existence
of equilibrium solutions of certain types, see for example~\cite{nishiura:02a,
ren:wei:03b, ren:wei:06b, ren:wei:06a, ren:wei:07b} and the references therein.
More recently, computer-assisted proofs have been used to obtain mathematically
rigorous results concerning the equilibrium structure, see for
example~\cite{cai:watanabe:19a, sander:wanner:21a, vandenberg:williams:17a,
vandenberg:williams:19a, wanner:17a}.

With the present paper, we start a study of equilibrium solutions for
block copolymers consisting of more than two monomers. For the sake of
specificity, we focus on the case of triblock copolymers, which were 
first discussed in~\cite{nakazawa:ohta:93a}, and whose rich variety of
steady state microstructures has been illustrated in~\cite{bohbot:etal:00a,
zheng:wang:95a}. Such systems have only recently become the subject of mathematical
studies, see for example~\cite{ren:wang:17a, wang:etal:17a}. For the present
paper, we study the triblock copolymer system given by
\begin{equation} \label{intro:tbcpsys1}
  \begin{array}{rcl}
    \DS \frac{\partial u_1}{\partial t} & = & \DS
      -\Delta \left( \epsilon^2 \Delta u_1 + f_1(u_1,u_2) \right)
      - \sigma \left( u_1 - \mu_1 \right) \; , \\[2.5ex]
    \DS \frac{\partial u_2}{\partial t} & = & \DS
      -\Delta \left( \epsilon^2 \Delta u_2 + f_2(u_1,u_2) \right)
      - \sigma \left( u_2 - \mu_2 \right) \; , \\[2.5ex]
    & & \DS u_1 + u_2 + u_3 \; = \; 1 \; .
  \end{array}
\end{equation}
In this system, the functions~$u_1, u_2, u_3 \colon \Omega \to \R$
describe the three monomer components. More precisely, values between~$0$
and~$1$ of~$u_i(t,x)$ indicate that at time~$t$ and location~$x \in \Omega$
the $i$-th monomer has concentration~$u_i(t,x)$. Since the above model is
of phase-field type, the function values of the~$u_i$ may not actually lie 
between 0 and 1, but they generally stay close
to this interval, and we interpret negative values or values larger than
one as being equal to concentrations of~$0$ or~$1$ in the physical system, 
respectively. This is similar to the
phase variables considered for example in the Allen-Cahn or Cahn-Hilliard
models, up to an affine transformation. The above system of equations has
to be satisfied in a bounded domain~$\Omega \subset \R^d$ for $d = 1,2,3$,
where for the purposes of this paper we restrict ourselves to the unit
cube~$\Omega = (0,1)^d$. In addition, we consider~(\ref{intro:tbcpsys1})
subject to homogeneous Neumann boundary conditions for all~$u_i$ and~$\Delta u_i$,
the small parameter~$\epsilon > 0$ models interaction length, and~$\sigma \ge 0$
represents the polymer length scale, just as in the diblock copolymer
case~\cite{sander:wanner:21a}. If we further define $\mu_3 = 1 - \mu_1 - \mu_2$,
then the three constants~$\mu_i$ have to be nonnegative, and they represent
the total mass of the three involved monomers in the sense that
\begin{equation} \label{intro:tbcpsys2}
  \frac{1}{|\Omega|} \int_\Omega u_i(t,x) \, dx
  \; = \; \mu_i
  \qquad\mbox{ for all }\qquad
  t \ge 0
  \quad\mbox{ and }\quad
  i = 1,2,3 \; .
\end{equation}
Finally, the functions~$f_1$ and~$f_2$ are suitable nonlinearities which are
derived from the gradient of a triple-well potential and which will be 
described in more detail in our model derivation in the next section.

Our approach for establishing the existence of equilibrium solutions of
the triblock copolymer system~(\ref{intro:tbcpsys1}) subject to the mass
constraints in~(\ref{intro:tbcpsys2}) is based on the constructive 
implicit function theorem introduced in~\cite{sander:wanner:16a}, 
combined with the rigorous Sobolev estimates of~\cite{wanner:18a}
and extends the approach used in~\cite{sander:wanner:21a} for the diblock
copolymer case. More precisely, we use spectral approximations based
on cosine series to find an approximate solution, and then have to
establish the following three estimates:
\begin{itemize}
\item First of all, one has to determine the residual of the
approximate solution, which in view of our use of a spectral
approximation combined with polynomial nonlinearities, amounts to
little more than evaluating a finite sum in interval arithmetic.
\item Next, one needs local Lipschitz bounds on the Fr\'echet 
derivative of~(\ref{intro:tbcpsys1}) at the approximate solution,
which can easily be obtained using the above-mentioned Sobolev
embedding results.
\item The last and most difficult step is to obtain a rigorous
bound on the operator norm of the inverse of the Fr\'echet 
derivative of~(\ref{intro:tbcpsys1}) at the approximate
solution. While in principle this will be accomplished as 
in~\cite{sander:wanner:21a}, the specifics must be adapted
to the current situation. This step is definitely the most 
elaborate part of the proof.
\end{itemize}
Once the above three tasks have been completed, one obtains 
computer-assisted proofs for small solution branches of
equilibrium solutions of~(\ref{intro:tbcpsys1}), similar 
to~\cite[Theorem~5]{sander:wanner:16a}
and~\cite[Section~5]{sander:wanner:21a}.

The above approach has a couple of shortcomings. In particular, the direct use of the
constructive implicit function theorem only provides small branch
segments, as the theorem in its original form is not aligned to
the actual tangent direction of the branch; see the discussion 
in~\cite{sander:wanner:16a}. Additionally, in order to follow
a branch through a saddle-node bifurcation
or to directly verify of the existence of certain
bifurcation points requires applying the
constructive implicit function theorem to a suitable extended
system. See for example~\cite{kamimoto:kim:sander:wanner:22a, 
lessard:sander:wanner:17a, rizzi:sander:wanner:p22a, sander:wanner:16a}.
In all of these cases, one has to study 
different extensions of the system~(\ref{intro:tbcpsys1}). 
In its current form,  for each of these applications the 
elaborate third step above has to be redone.

One of the main contributions of the current paper is the development of 
a flexible general framework such that this reassembly for each application 
is not necessary. Namely,  we derive
the norm estimate on the inverse operator right away for a
sufficiently large class of linear systems, such that we can easily reuse 
already calculated results for a variety of different extended systems. More precisely,
the central part of this paper is devoted to obtaining such
estimates for linear operators~$L$ acting on~$m \in \N_0$
scalar parameters $\eta_1, \ldots, \eta_m$ and on~$n \in \N$
functions~$v_k \colon \Omega \to \R$ in such a way that the
first~$m$ components of~$L$ are given by the scalars
\begin{equation} \label{intro:linop1}
  \sum_{i=1}^m \alpha_{ki} \eta_i +
    \sum_{j=1}^n l_{kj}(v_j)
  \qquad\mbox{ for }\qquad
  k = 1,\ldots,m \; ,
\end{equation}
and the next~$n$ components of~$L$ are given by the functions
\begin{equation} \label{intro:linop2}
  -\beta_{k} \Delta^2 v_k 
    - \sum_{i=1}^m b_{ki} \eta_i
    - \Delta \sum_{j=1}^n c_{kj} v_j
    - \sum_{j=1}^{n} \gamma_{kj} v_{j}
  \qquad\mbox{ for }\qquad
  k = 1,\ldots,n \; .
\end{equation}
In these formulas, the variables~$\alpha_{ki}$, $\beta_{k} > 0$, and $ \gamma_{kj} $
refer to real constants, 
while~$b_{ki}$ and~$c_{kj}$
denote suitably smooth real-valued functions defined on~$\Omega$,
and $ l_{kj} $ denotes a bounded linear functional on the ambient space for $ v_j $
with Riesz representative in a suitable finite dimensional subspace.

At first glance, the generality of the linear operators defined
in~(\ref{intro:linop1}) and~(\ref{intro:linop2}) might seem
exaggerated, given that the main application of this paper is the
establishment of certain triblock copolymer microstructures. In
fact, however, the above generality allows for a number of direct
applications:
\begin{itemize}
\item An immediate application is the study of pitchfork 
bifurcation points in the diblock copolymer model which are
induced by symmetry-breaking based on a cyclic group
action~\cite{rizzi:sander:wanner:p22a}. While these results answer
an open question posed in~\cite{lessard:sander:wanner:17a}, the
latter paper was based on the radii polynomial approach, and
one would have to adapt the estimates for every group order
to obtain a computer-assisted proof. The estimates of the
present paper apply directly and without change.
\item We can extend the initial study of triblock copolymers
in this paper to a more systematic study of their bifurcation
diagrams using rigorous pseudo-arclength continuation, and
thereby shed some light on the creation of bubble
assemblies~\cite{wang:etal:17a}.
\item The norm bound estimate can be used to obtain rigorous
results on double and quadruple bubbles in multi-component
metallic alloys, as modeled by Cahn-Morral
systems~\cite{maier:stoth:wanner:00a}. So far, only numerical
results have been obtained in~\cite{desi:etal:11a}.
\item More generally, our
construction opens the door to a more detailed study of the
bifurcation structure of the celebrated Cahn-Hilliard model 
on higher-dimensional domains~\cite{maier:etal:08a,
maier:etal:07a}, including pseudo-arclength continuation,
bifurcation point verification, and continuation of bifurcation
points in a two-parameter setting.
\end{itemize}
The above list is not meant to be exhaustive, but rather to justify
the extra effort which is necessary to study the linear operator
defined in~(\ref{intro:linop1}) and~(\ref{intro:linop2}). For the sake 
of keeping the current paper from becoming too long, we will address these 
applications in future work. 

The remainder of the paper is organized as follows. 
Section~\ref{sec:tbcpequil} is devoted to our preliminary study
of triblock copolymers. In addition to deriving the model and
describing its basic stability as a function of the mass vector,
we also describe how a constructive version of the implicit function
theorem can be used to obtain computer-assisted proofs for small
branches of equilibrium solutions. The section closes with the
presentation of specific computer-assisted proofs of a variety
of observed microstructures. Section~\ref{sec:defandsetup}
introduces the functional-analytic framework for our
computer-assisted equilibrium validation. In addition to
recalling  definitions and results from~\cite{sander:wanner:21a,
wanner:18a}, we also derive the Lipschitz estimates which are
necessary for the application of the constructive implicit function
theorem. The remaining ingredient for the validation of the triblock
copolymer stationary states is the derivation of the inverse norm bound,
which is the subject of Section~\ref{sec:invnormbnd}. Finally,
Section~\ref{sec:future} contains conclusions and future applications.
\section{Validated triblock copolymer equilibria}
\label{sec:tbcpequil}
Before diving into the more technical parts of the paper, we demonstrate
how these results can be applied in the context of triblock copolymers. 
More precisely, in the present section we rigorously validate equilibria
for the triblock copolymer model, by establishing intriguing patterns
for a variety of different mass vectors. To our knowledge, this is the
first computer-assisted result in this context.

To accomplish this, Subsection~\ref{subsec:model} briefly describes the
derivation of the model from its energy functional, which is followed in
Subsection~\ref{subsec:homogequil} by a discussion of the stability of the
homogeneous steady state and the creation of nontrivial steady states via
bifurcations from the homogeneous one. We then present the framework for
our computer-assisted equilibrium validation in Subsection~\ref{sec:capdetail},
which is based on a constructive version of the implicit function theorem.
Finally, Subsection~\ref{sec:rigvermic} contains sample computer-assisted
existence proofs for equilibrium solutions, which are all based on the
functional-analytic framework described in Section~\ref{sec:defandsetup}
and the main inverse norm bound derived in Section~\ref{sec:invnormbnd}.
\subsection{Derivation of the triblock copolymer model}
\label{subsec:model}
The dynamics of triblock copolymers are dictated by their associated
free energy as introduced in~\cite{nakazawa:ohta:93a}, which is given
by
\begin{equation} \label{tbcp:energy}
  E_{\epsilon,\sigma}[u] \; = \;
  \int_\Omega \left( \frac{\epsilon^2}{2} \sum_{i=1}^3
    |\nabla u_i|^2 + F(u) + \frac{\sigma}{2} \sum_{i=1}^3
    \left| (-\Delta)^{-1/2} (u_i - \mu_i) \right|^2
    \right) \, dx \; ,
\end{equation}
where $u = (u_1,u_2,u_3) \colon \Omega \to \R^3$ denotes a material
state which satisfies $u_1 + u_2 + u_3 = 1$ throughout~$\Omega$,
and~$F$ is the triple-well potential given by
\begin{equation} \label{tbcp:potential}
  F(u) = F(u_1,u_2,u_3) = \sum_{i=1}^3 g(u_i)
  \qquad\mbox{ with }\qquad
  g(s) = \frac{27 s^2 (1-s)^2}{4} \; .
\end{equation}
This energy functional has global minima at the three points~$(1,0,0)$,
$(0,1,0)$, and~$(0,0,1)$, which correspond to the three pure
monomers. From this energy functional, one can derive several
different gradient-like evolution equations if the function~$u$
also shows $t$-dependence. While the recent study~\cite{wang:etal:17a}
considered a standard $L^2(\Omega)$-gradient, which results in a nonlocal
second-order evolution equation, we follow the standard procedure which
in the two-component case leads to the Ohta-Kawasaki model. That is, we
use the $H^{-1}(\Omega)$-gradient instead. This results in a local
fourth-order partial differential equation, which exhibits a
structure similar to the classical Cahn-Hilliard, Cahn-Morral,
or Ohta-Kawasaki models.

In order to arrive at an evolution equation for the first two
components~$u_1$ and~$u_2$ only, which respects the conservation
of mass identity~$u_1 + u_2 + u_3 = 1$ throughout~$\Omega$, we follow
the procedure outlined in~\cite{elliott:luckhaus:91a, eyre:96a,
maier:stoth:wanner:00a} for the Cahn-Morral case. For this, define
the vector-valued function~$u := (u_1,u_2,u_3) \colon \Omega \to \R^3$
and let $\mu = (\mu_1,\mu_2,\mu_3)$ be the vector of total masses of the three 
involved monomers, as defined in \eqref{intro:tbcpsys2}. 
In fact, we assume that the latter
mass vector lies in the {\em Gibbs triangle\/}~$\cG$ defined as
\begin{displaymath}
  \cG := \left \{ \mu \in \R^3 \; \colon \;
    \sum_{i=1}^3 \mu_i = 1 \;\;\mbox{ and }\;\;
    \mu_i \ge 0 \;\;\mbox{ for }\;\;
    i = 1,2,3 \right\} \; .
\end{displaymath}
By computing the $H^{-1}(\Omega)$-gradient of~(\ref{tbcp:energy})
as in~\cite{elliott:luckhaus:91a, eyre:96a, maier:stoth:wanner:00a},
one can then associate the gradient dynamics given by
\begin{equation} \label{tbcp:evoleqn}
  \begin{array}{rcrl}
    \DS \frac{\partial u}{\partial t} & = & \DS
      -\Delta \left( \epsilon^2 \Delta u + \bar{f}(u) \right) -
      \sigma (u - \mu)
      & \mbox{ in } \;\; \Omega \; , \\[3ex]
    & & \DS \frac{\partial u}{\partial \nu} \; = \;
       \frac{\partial \Delta u}{\partial \nu} \; = \; 0
      & \mbox{ on } \;\; \partial\Omega \; ,
  \end{array}
\end{equation}
where~$\nu$ denotes the unit outward normal vector to the
boundary~$\partial\Omega$. We would like to point out that
while at first glance this equation appears to be exactly the
diblock copolymer model, it is in fact a system of equations
for the vector-valued function~$u$, i.e., for all
$t\ge 0$ and arbitrary $x \in \Omega$ we have $u(t,x) \in \R^3$. 
The nonlinearity~$\bar{f} \colon \R^3
\to \R^3$ is given by
\begin{displaymath}
  \bar{f}(u) = -P \nabla F(u) \; ,
  \qquad\mbox{ where }\qquad
  Pv = v - \frac{v \cdot \ee}{|\ee|^2} \, \ee
  \quad\mbox{ and }\quad
  \ee = (1,1,1) \; .
\end{displaymath}
This form of the nonlinearity ensures that the right-hand side
of~(\ref{tbcp:evoleqn}) is pointwise orthogonal to the vector~$\ee$,
and therefore the evolution of this partial differential equation
leaves the affine space~$u_1 + u_2 + u_3 = 1$ pointwise invariant.
One can therefore consider the first two equations of the evolution
equation~(\ref{tbcp:evoleqn}) only, where we replace~$u_3$
by~$1 - u_1 - u_2$. Thus, we finally obtain the equation stated
in~(\ref{intro:tbcpsys1}), where the nonlinearities~$f_1$ and~$f_2$
are given by
\begin{displaymath}
  \begin{array}{rcccl}
    \DS f_1(u_1,u_2) & = & \DS \bar{f}_1(u_1,u_2,1-u_1-u_2) & = &
      \DS \frac{-2 g'(u_1) + g'(u_2) + g'(1-u_1-u_2)}{3} \; , \\[3ex]
    \DS f_2(u_1,u_2) & = & \DS \bar{f}_2(u_1,u_2,1-u_1-u_2) & = &
      \DS \frac{g'(u_1) - 2 g'(u_2) + g'(1-u_1-u_2)}{3} \; ,
  \end{array}
\end{displaymath}
where~$g$ was defined in~(\ref{tbcp:potential}).

In the present paper, we will study the set of equilibrium states
of the system~(\ref{intro:tbcpsys1}), i.e., we set the right-hand sides
in the partial differential equations equal to zero, which in turn gives
the fourth-order elliptic nonlinear system
\begin{equation} \label{tbcp:equilsys}
  \begin{array}{rclcrcl}
    \DS -\Delta \left( \epsilon^2 \Delta u_1 + f_1(u_1,u_2) \right)
      - \sigma \left( u_1 - \mu_1 \right) & = & 0 \; , & \quad &
      \DS \frac{1}{|\Omega|} \int_\Omega u_1 \, dx & = & \mu_1 \; , \\[3ex]
    \DS -\Delta \left( \epsilon^2 \Delta u_2 + f_2(u_1,u_2) \right)
      - \sigma \left( u_2 - \mu_2 \right) & = & 0 \; , & \quad &
      \DS \frac{1}{|\Omega|} \int_\Omega u_2 \, dx & = & \mu_2 \; ,
  \end{array}
\end{equation}
subject to homogeneous Neumann boundary conditions for~$u_i$
and~$\Delta u_i$, for $i=1,2$.
\subsection{Stability of the homogeneous steady state}
\label{subsec:homogequil}
One can easily see that the equilibrium problem~(\ref{tbcp:equilsys})
for the triblock copolymer equation has the spatially constant
solution~$(\mu_1,\mu_2)$, whenever we have $\mu = (\mu_1,\mu_2,\mu_3)
\in \cG$. Thus, one can hope to find nontrivial equilibria through
path-following from this homogeneous steady state as the
parameter~$\epsilon$ is varied. Since one can easily see that for
large values of~$\epsilon$ the homogeneous state is stable (see also
the discussion below), we need to focus on mass vectors for which this
steady state becomes unstable as~$\epsilon$ decreases.

More generally, suppose now that~$u_1, u_2 \colon \Omega \to \R$ denotes
a solution of the system~(\ref{tbcp:equilsys}), i.e., the functions~$u_1$,
$u_2$, and $u_3 = 1 - u_1 - u_2$ are an equilibrium for the triblock
copolymer equation~(\ref{intro:tbcpsys1}). In order to understand the
stability of this steady state one has to consider the linearization
given by
\begin{equation} \label{tbcp:equilin}
  \begin{array}{rcl}
    \DS \frac{\partial v_1}{\partial t} & = & \DS
      -\Delta \left( \epsilon^2 \Delta v_1 -
        \frac{2}{3} g''(u_1) v_1 + \frac{1}{3} g''(u_2) v_2 -
        \frac{1}{3} g''(u_3) (v_1 + v_2) \right)
      - \sigma v_1 \; , \\[3ex]
    \DS \frac{\partial v_2}{\partial t} & = & \DS
      -\Delta \left( \epsilon^2 \Delta v_2 +
        \frac{1}{3} g''(u_1) v_1 - \frac{2}{3} g''(u_2) v_2 -
        \frac{1}{3} g''(u_3) (v_1 + v_2) \right)
      - \sigma v_2 \; ,  \end{array}
\end{equation}
or, more precisely, the spectrum of the linear elliptic operator 
induced by its right-hand side. For the homogeneous case this 
linearization simplifies to the linear partial differential equation
\begin{equation} \label{tbcp:homoglin}
  \frac{\partial}{\partial t}
    \begin{pmatrix} v_1 \\ v_2 \end{pmatrix}
  \; = \;
  -\Delta \left( \epsilon^2 \Delta
    \begin{pmatrix} v_1 \\ v_2 \end{pmatrix} +
    M \begin{pmatrix} v_1 \\ v_2 \end{pmatrix} \right) -
    \sigma \begin{pmatrix} v_1 \\ v_2 \end{pmatrix} \; ,
\end{equation}
where the matrix~$M \in \R^{2 \times 2}$ is defined as
\begin{equation} \label{tbcp:homoglinM}
  M \; = \;
  \frac{1}{3}
    \begin{pmatrix}
      -2g''(\mu_1) & g''(\mu_2) \\[1ex]
      g''(\mu_1) & -2g''(\mu_2)
    \end{pmatrix} -
  \frac{1}{3} \, g''(1-\mu_1-\mu_2)
    \begin{pmatrix} 1 & 1 \\ 1 & 1 \end{pmatrix}
  \; .
\end{equation}
This linearization is considered subject to homogeneous Neumann
boundary conditions as before, and in addition, with the homogeneous
integral constraints $\int_\Omega v_1 \, dx = \int_\Omega v_2 \, dx = 0$.
Its stability is the subject of the next lemma.
\begin{lemma}[Stability of the homogeneous state]
\label{lem:homogstability}
Let~$\mu \in \cG$ denote an arbitrary mass vector in the Gibbs
triangle, and consider the linearization of the triblock copolymer
model~(\ref{intro:tbcpsys1}) at this homogeneous state, as given
in~(\ref{tbcp:homoglin}) and~(\ref{tbcp:homoglinM}) above. Then
the matrix~$M$ is diagonalizable, i.e., we can find two eigenvalues
$\nu_1,\nu_2 \in \R$ and eigenvectors $p_1,p_2 \in \R^2$ such that
\begin{equation} \label{lem:homogstability1}
  M p_1 = \nu_1 p_1
  \quad\mbox{ and }\quad
  M p_2 = \nu_2 p_2 \; ,
  \quad\mbox{ as well as }\quad
  \nu_1 \ge \nu_2 \; .
\end{equation}
Furthermore, if the eigenvalues of the negative Laplacian subject
to homogeneous Neumann boundary conditions and zero mass constraint
on~$\Omega$ are given by $0 < \kappa_1 \le \kappa_2 \le \ldots \to
\infty$, and the associated eigenfunctions~$\phi_k$ satisfy
\begin{equation} \label{lem:homogstability2}
  -\Delta \phi_k = \kappa_k \phi_k \quad\mbox{ on }\;\Omega
  \qquad\mbox{ and }\qquad
  \frac{\partial \phi_k}{\partial \nu} = 0
    \quad\mbox{ on }\;\partial\Omega \; ,
\end{equation}
then the following hold:
\begin{itemize}
\item[(a)] For every $j = 1,2$ and $k \in \N$ the vector~$p_j \phi_k$
of functions is an eigenfunction of the right-hand side of~(\ref{tbcp:homoglin})
with associated eigenvalue
\begin{equation} \label{lem:homogstability3}
  \lambda_{j,k} = \kappa_k \left( \nu_j - \epsilon^2 \kappa_k \right)
    - \sigma \; .
\end{equation}
Furthermore, the spectrum of the operator given by the right-hand
side of~(\ref{tbcp:homoglin}) consists precisely of these
eigenvalues~$\lambda_{j,k}$ for $j = 1,2$ and $k \in \N$.
\item[(b)] If the inequality $\nu_2 \le \nu_1 \le 0$ holds, then the
homogeneous equilibrium~$\mu$ of~(\ref{intro:tbcpsys1}) is stable for
all values $\epsilon > 0$. On the other hand, if we have $\nu_1 > 0$,
then the homogeneous state is unstable for all sufficiently small
$\epsilon > 0$.
\end{itemize}
\end{lemma}
\begin{proof}
It is elementary to show that  if we write 
$\mu_3 = 1 - \mu_1 - \mu_2$, then the characteristic polynomial
of the matrix~$M$ has the discriminant
\begin{displaymath}
  \frac{2}{9} \left(
  \left( g''(\mu_1) - g''(\mu_2) \right)^2 +
  \left( g''(\mu_2) - g''(\mu_3) \right)^2 +
  \left( g''(\mu_3) - g''(\mu_1) \right)^2 \right)
  \; \ge \; 0 \; ,
\end{displaymath}
which is clearly nonnegative. Thus, the matrix~$M$ always has two real
eigenvalues and associated real eigenvectors, i.e., we can assume
that~(\ref{lem:homogstability1}) is satisfied. But then one easily
obtains that the pair of functions~$\psi = p_j \phi_k$ satisfies
the identity $-\Delta(\epsilon^2 \Delta \psi + M\psi) = \lambda_{j,k}
\psi$, by applying~(\ref{lem:homogstability2}) component-wise, in
combination with~(\ref{lem:homogstability1}). In addition, note that
according to our construction the eigenfunctions~$\{ \phi_k \}_{k \in \N}$
form a complete orthogonal set in the Hilbert space~$X = \{ w \in
L^2(\Omega) \colon \int_\Omega w \, dx = 0 \}$. Thus, the function
pairs~$p_j \phi_k$ for $j = 1,2$ and $k \in \N$ form a complete
orthogonal set in~$X \times X$, which immediately establishes~{\em (a)\/}.
Finally, the statements in~{\em (b)\/} follow easily from the formula
in~(\ref{lem:homogstability3}) and the fact that $\kappa_k > 0$ for
all $k \in \N$. This completes the proof of the lemma.
\end{proof}
\begin{figure}
  \begin{center}
    \hspace*{-2.5cm}
    \includegraphics[width=0.5\textwidth]{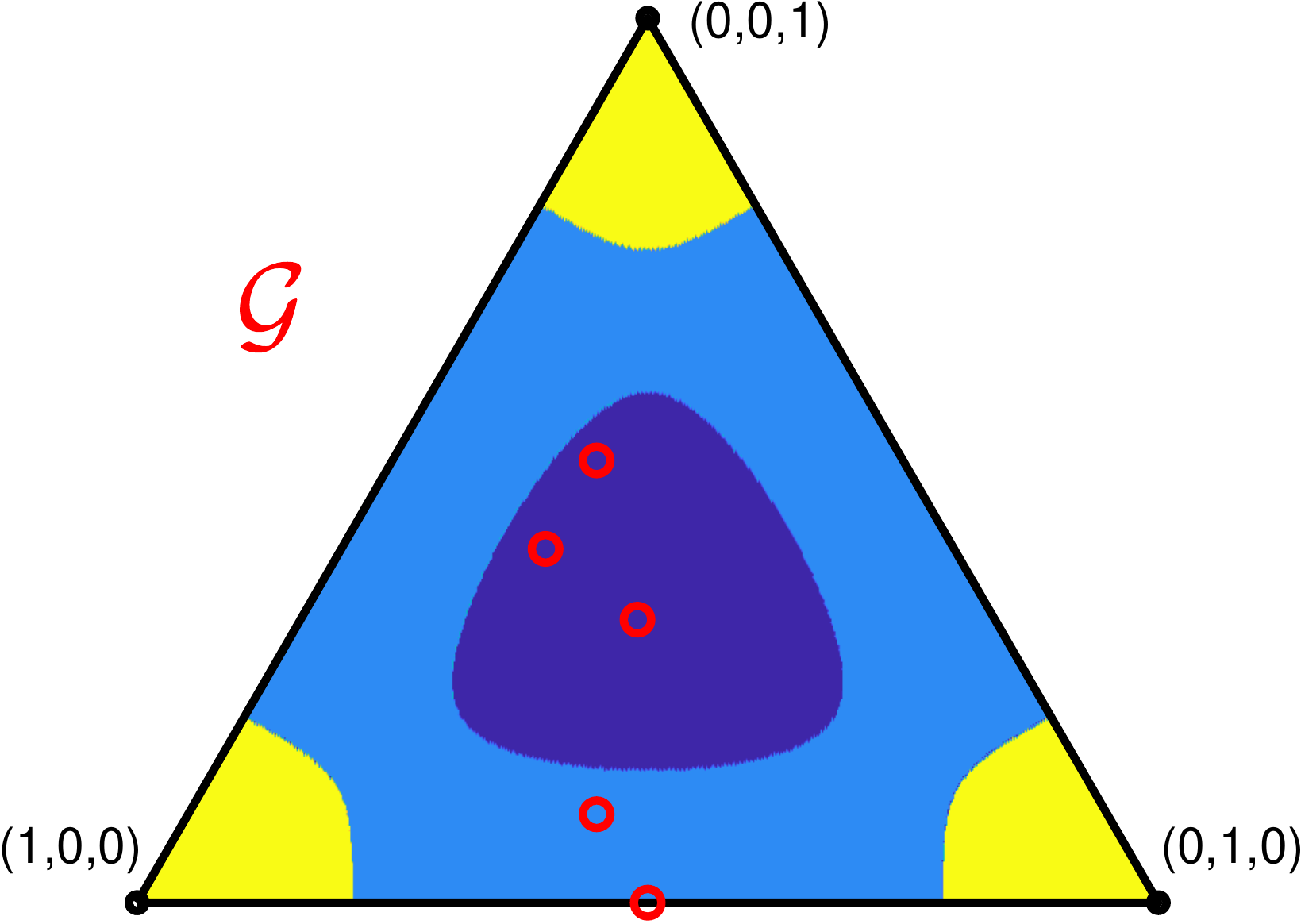}
    \hspace*{2.5cm}
    \raisebox{2.5cm}{%
      \makebox(4.5,5.0){
      \begin{tabular}{|c|c|c|} \hline
        $\mu_1$ & $\mu_2$ & $\mu_3$ \\ \hline \hline
        0.30 & 0.20 & 0.50 \\
        0.40 & 0.20 & 0.40 \\
        0.35 & 0.33 & 0.32 \\
        0.50 & 0.40 & 0.10 \\
        0.50 & 0.50 & 0.00 \\ \hline
      \end{tabular}}}
  \end{center}
  \caption{Stability regions for the homogeneous steady state of the
           triblock copolymer model~(\ref{intro:tbcpsys1}) in the Gibbs
           triangle~$\cG$. For total mass vectors in the yellow regions
           the homogeneous state is stable, if the vector lies in either
           the light or dark blue areas, then the state is unstable for
           sufficiently small $\epsilon > 0$. The dark blue region
           corresponds to $\nu_1 \ge \nu_2 > 0$, while the light blue
           region is for $\nu_1 > 0 \ge \nu_2$ in~(\ref{lem:homogstability1}).
           The five red dots indicate mass vectors at which we validated
           nontrivial solutions, and the associated mass values are listed
           in the table. Figures~\ref{fig:bifdiag} and~\ref{fig:tbcp2dimages:a}--\ref{fig:tbcp2dimages:c}
            have~$\mu$ values in the dark blue region, whereas
           Figures~\ref{fig:bifdiag1d} and~\ref{fig:tbcp1dimages} have values of~$\mu$
           in the light blue region.}
  \label{fig:stabregion}
\end{figure}%

The stability of the homogeneous steady state~$\mu = (\mu_1,\mu_2,\mu_3)
\in \cG$ is illustrated in Figure~\ref{fig:stabregion}. In this figure,
yellow regions correspond to $\nu_2 \le \nu_1 \le 0$, i.e., in those
regions the homogeneous state~$\mu$ is stable for all $\epsilon > 0$.
On the other hand, the light blue region corresponds to the inequality
$\nu_1 > 0 \ge \nu_2$, while the dark blue region is for $\nu_1 \ge
\nu_2 > 0$. In both of these regions, the homogeneous state~$\mu$ is
unstable for all sufficiently small $\epsilon > 0$, and one can hope
to observe sudden phase separation in solutions of the triblock copolymer
model originating nearby.
\begin{figure}
  \begin{center}
  \includegraphics[width=0.45\textwidth]{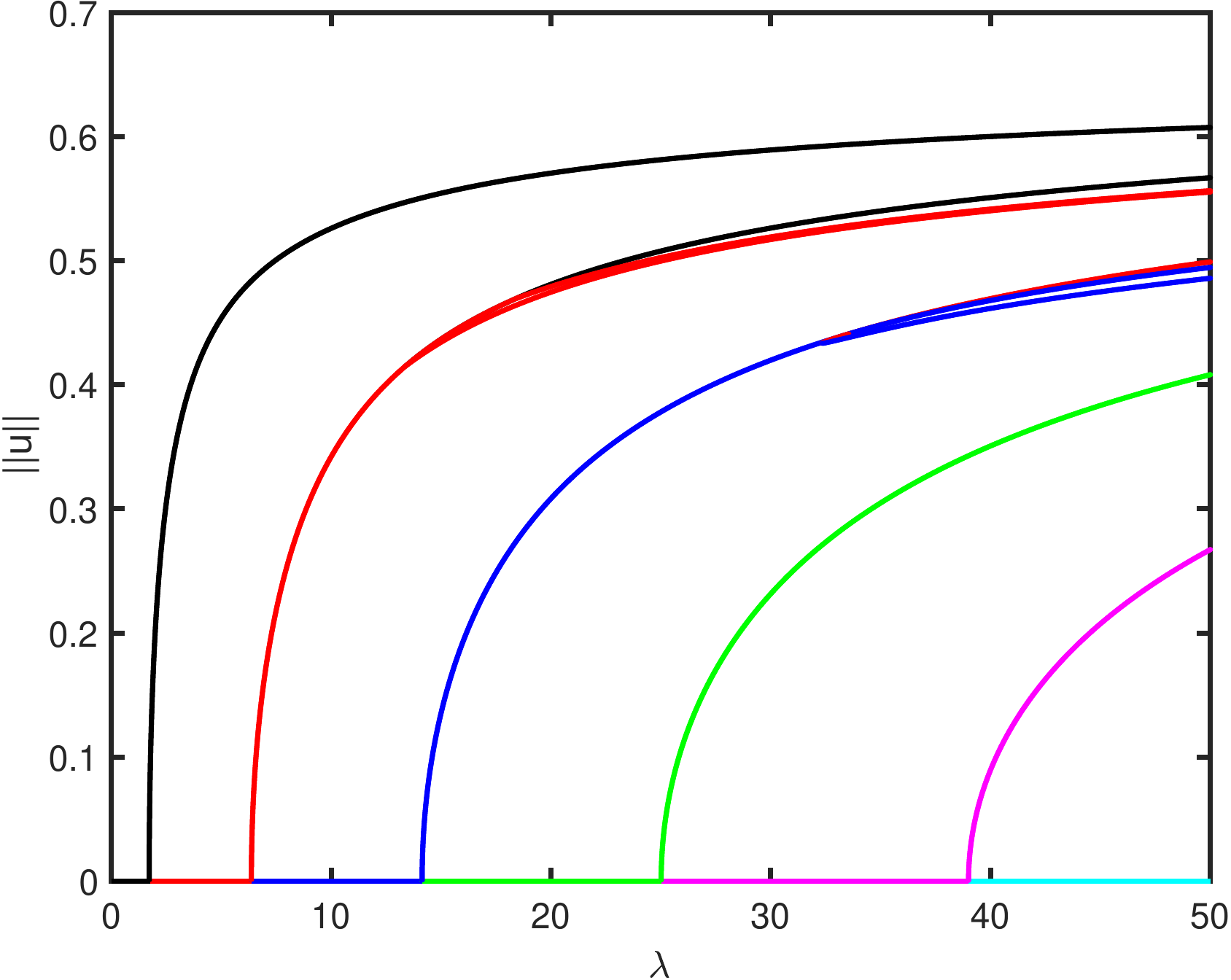}
  \includegraphics[width=0.45\textwidth]{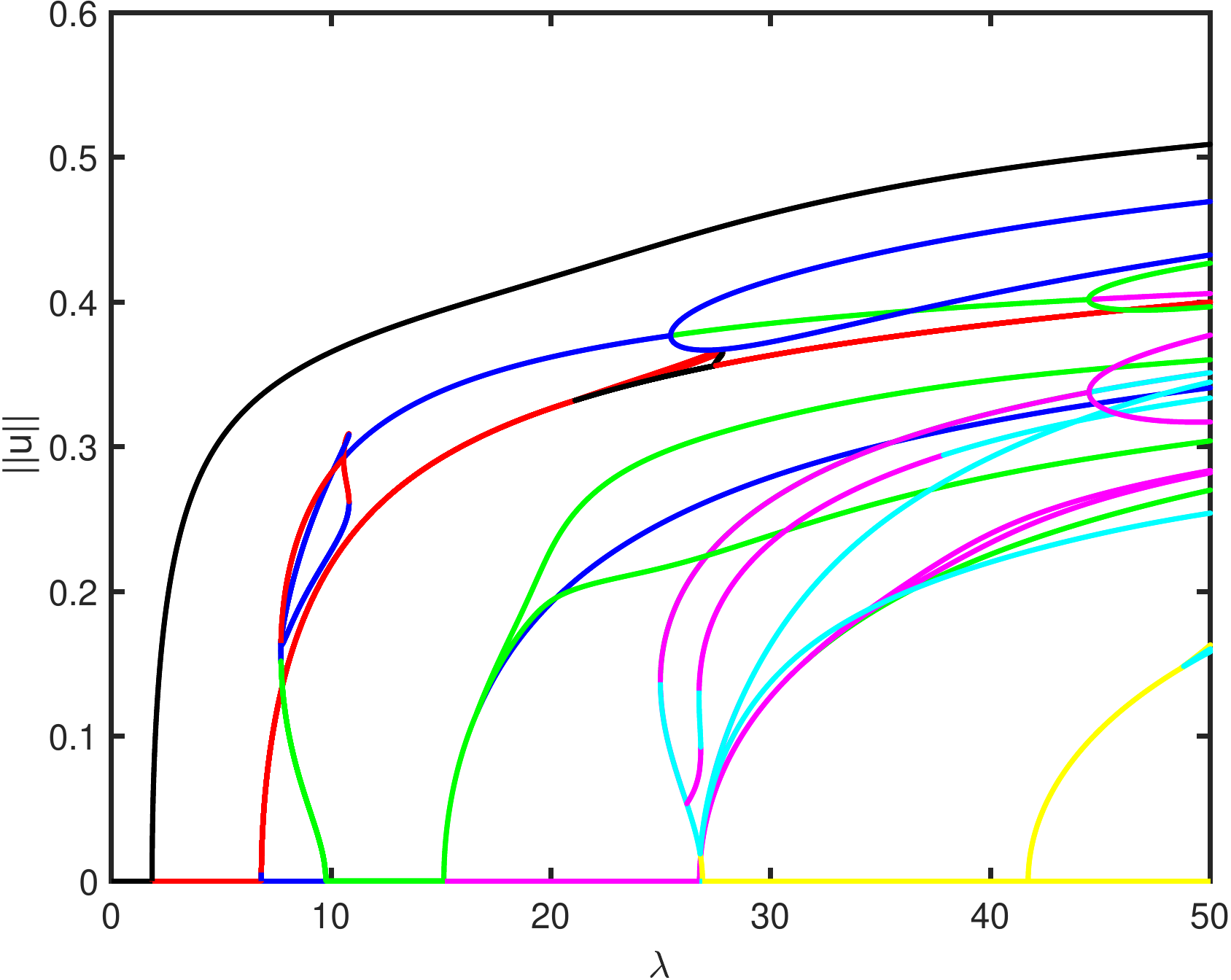}
  \caption{\label{fig:bifdiag1d}
    Sample numerically computed bifurcation diagrams for the triblock
    copolymer equilibrium solutions on the
    one-dimensional domain $\Omega = (0,1)$. The parameters $\sigma=6.0$ and
    on left $\mu = (0.5,0.4,0.1)$ and right $\mu = (0.4,0.2,0.4)$ are held fixed, 
    while~$\lambda = 1/\epsilon^2$ is
    permitted to vary. The color coding represents the index of the equilibrium
    branch, with black, red, blue, green, and magenta corresponding to index~$0$,
    $1$, $2$, $3$, and~$4$, respectively. Since $\mu_3 \approx 0$ in the left image, 
    the diagram is
    qualitatively very similar to the one of the diblock copolymer equation
    shown in~\cite{johnson:etal:13a, lessard:sander:wanner:17a}.
    For the right image, since $\mu_3 \gg 0$, the diagram looks very different.
    }
  \end{center}
\end{figure}
\begin{figure}
  \begin{center}
  \includegraphics[width=0.45\textwidth]{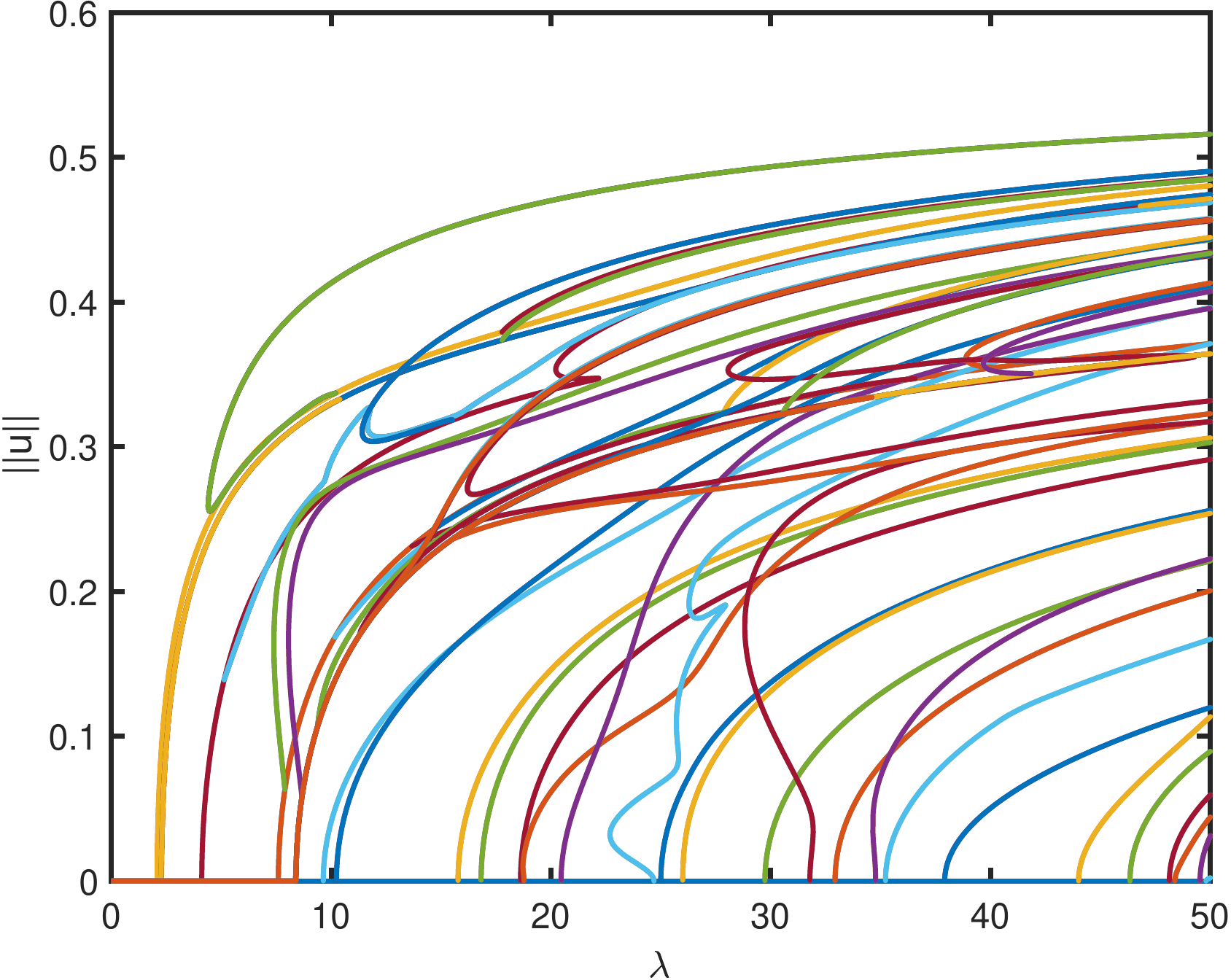}
  \hspace*{0.3cm}
  \includegraphics[width=0.45\textwidth]{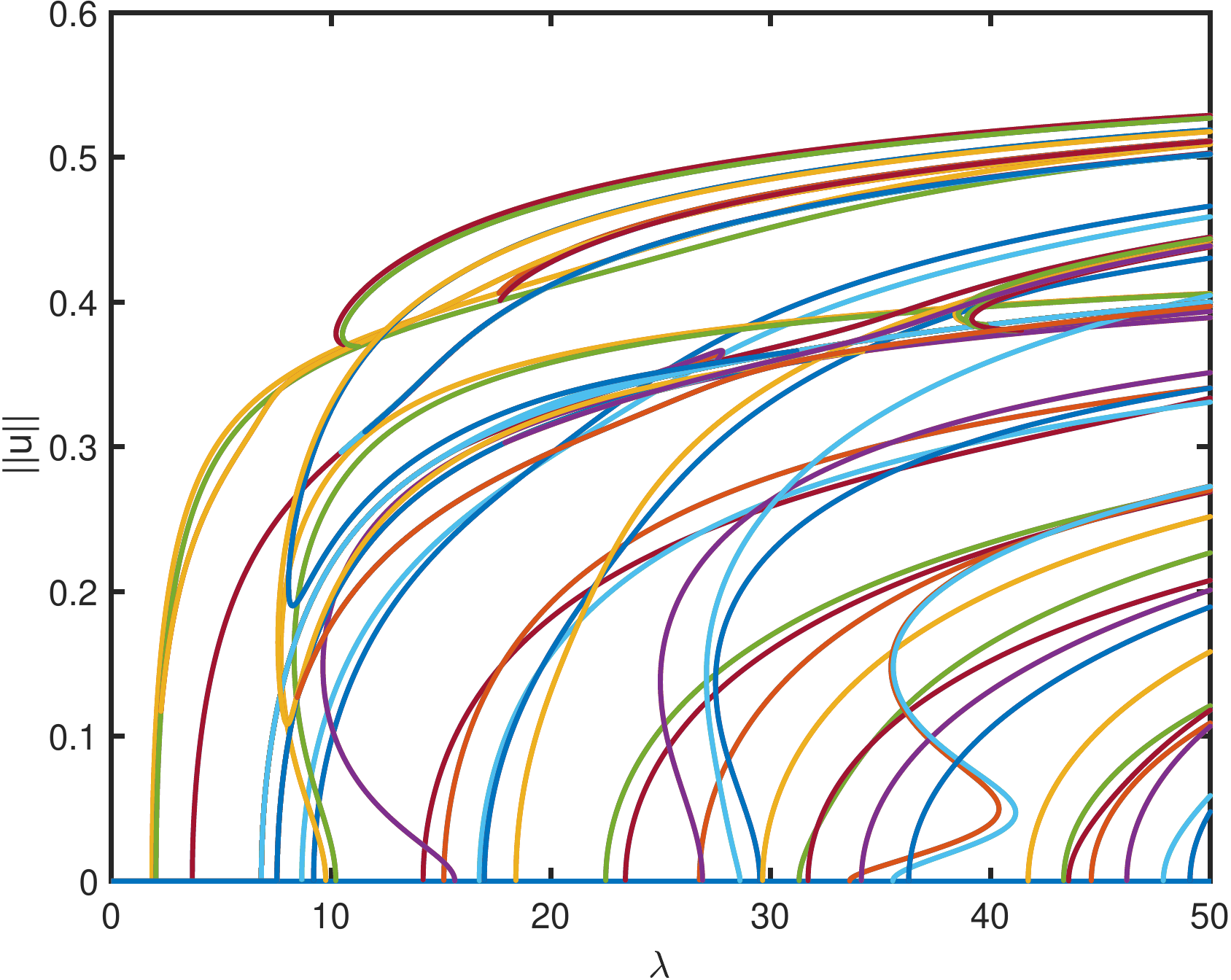} 
  \caption{\label{fig:bifdiag}
           Sample numerically computed triblock copolymer bifurcation diagrams
           for the two-dimensional domain $\Omega = (0,1)^2$. In both figures,
           we have chosen $\sigma=6.0$ and $\lambda = 1/\epsilon^2$ varies.
           The left diagram is for mass vector $\mu = (0.3,0.2,0.5)$, while the
           right panel is for the vector $\mu = (0.4,0.2,0.4)$. These diagrams
           are not a complete set of all equilibrium solutions, but they give
           a sense of the vast array of possible solutions one can find.}
  \end{center}
\end{figure}

To explain this last comment in more detail, we introduce an abbreviation
which will be used throughout the remainder of the paper. As mentioned
in the previous paragraph, the stability of the homogeneous state in the
blue regions of Figure~\ref{fig:stabregion} changes for small enough
values of~$\epsilon$, and in fact it is the limit $\epsilon \to 0$ which
leads to interesting nonhomogeneous stationary states. Uncovering these
states will be accomplished by continuation techniques, and it is therefore
more convenient to instead introduce the new parameter
\begin{displaymath}
  \lambda = \frac{1}{\epsilon^2}
\end{displaymath}
which is then studied in the limit $\lambda \to \infty$. One
can easily see that if we divide both sides of the linearized
problem~(\ref{tbcp:equilin}) by~$\epsilon^2$ and rescale time, the
spectrum of the right-hand side for small $\lambda \approx 0$ is a
small perturbation of the spectrum of the stable bi-Laplacian
operator~$-\Delta^2$. Now consider a homogeneous mass vector~$\mu$
in one of the blue regions in Figure~\ref{fig:stabregion}. Then the
stability of the associated homogeneous state has to change as~$\lambda$
increases, and standard results from bifurcation theory imply the 
appearance of nontrivial equilibrium solutions. This can be seen 
in Figures~\ref{fig:bifdiag1d} and~\ref{fig:bifdiag}, which contain
sample numerically computed bifurcation diagrams for a few different
mass vectors in the blue regions. The diagrams in Figure~\ref{fig:bifdiag1d}
are for the triblock copolymer model on the one-dimensional domain
$\Omega = (0,1)$, while the ones in Figure~\ref{fig:bifdiag} are for
$\Omega = (0,1)^2$. Notice that all of these diagrams indicate the
appearance of a multitude of nontrivial stationary states. Moreover,
while their number seems manageable in the one-dimensional situation,
this is no longer the case in space dimension two. In the remainder of
this paper, we show how these numerically computed equilibrium solutions
can be established rigorously.
\subsection{Computer-assisted equilibrium validation}
\label{sec:capdetail}
Our rigorous equilibrium validation is a significant extension of the 
constructive implicit function theorem approach which
was first introduced in~\cite{sander:wanner:16a, wanner:17a}, and which
was further refined in~\cite{sander:wanner:21a}. In the present subsection,
we demonstrate how it can be adapted to the situation of the triblock
copolymer model. Our goal is to prove the existence of stationary states
for the triblock copolymer model on the domain~$\Omega = (0,1)^d$, where
for the purposes of this paper we focus on $d = 1,2$. Such equilibrium
solutions satisfy the nonlinear elliptic system~(\ref{tbcp:equilsys}),
which for our approach has to be slightly reformulated. Due to the
involved mass constraints, we introduce the transformations
\begin{displaymath}
  u_1 = \mu_1 + w_1
  \qquad\mbox{ and }\qquad
  u_2 = \mu_2 + w_2 \; .
\end{displaymath}
Furthermore, as mentioned already at the end of the last subsection,
rather than visualizing bifurcation diagrams in the
$\epsilon$-$(w_1,w_2)$-coordinate system and the small limit $\epsilon \to 0$,
which would imply that the equilibrium branches of interest become arbitrarily
close together, we instead use the large continuation parameter
$\lambda = 1/\epsilon^2$. Thus instead of~(\ref{tbcp:equilsys}) we
consider the transformed system
\begin{equation} \label{tbcp:equilsysW}
  \begin{array}{rclcrcl}
    \DS -\Delta \left( \Delta w_1 + \lambda f_1(\mu_1 + w_1,
      \mu_2 + w_2) \right) - \lambda \sigma w_1 & = & 0 \; , & \quad &
      \DS \int_\Omega w_1 \, dx & = & 0 \; , \\[3ex]
    \DS -\Delta \left( \Delta w_2 + \lambda f_2(\mu_1 + w_1,
      \mu_2 + w_2) \right) - \lambda \sigma w_2 & = & 0 \; , & \quad &
      \DS \int_\Omega w_2 \, dx & = & 0 \; .
  \end{array}
\end{equation}
The underlying basic function spaces are the following Sobolev spaces, where the subscript $n$
indicates Neumann boundary conditions. 
\begin{equation} \label{tbcp:h2dot}
  \bar{H}_n^2(\Omega) \; = \; \left\{ w \in H^2(\Omega) \; \colon
    \int_\Omega w \, dx = 0 \, , \;
    \frac{\partial w}{\partial \nu} = 0
    \,\mbox{ on } \,\partial\Omega \right\}
  \;\mbox{ and }\;
  \bar{H}_n^{-2}(\Omega) = \bar{H}_n^2(\Omega)^* \; .
\end{equation}
Then the equilibrium system~(\ref{tbcp:equilsysW}) can be written as the
nonlinear zero finding problem
\begin{equation} \label{tbcp:zeroeqn}
  \cF(\lambda,w) = 0
  \qquad\mbox{ for }\qquad
  \cF : \R \times \cX \to \cY
\end{equation}
with
\begin{equation} \label{tbcp:defcxcy}
  \cX = \bar{H}_n^2(\Omega) \times \bar{H}_n^2(\Omega)
  \quad\mbox{ and }\quad
  \cY = \bar{H}_n^{-2}(\Omega) \times \bar{H}_n^{-2}(\Omega) \; ,
\end{equation}
as well as
\begin{equation} \label{tbcp:defcf}
  \begin{array}{rcl}
    \DS \cF(\lambda,(w_1,w_2)) & = & -\Delta (\Delta w +
      \lambda f(\mu + w)) - \lambda \sigma w \\[1ex]
    & = & \DS
      \left( -\Delta \left( \Delta w_1 + \lambda f_1(\mu_1 + w_1,
      \mu_2 + w_2) \right) - \lambda \sigma w_1 , \right. \\[1ex]
    & & \; \left. \DS
      -\Delta \left( \Delta w_2 + \lambda f_2(\mu_1 + w_1,
      \mu_2 + w_2) \right) - \lambda \sigma w_2 \right) \; .
  \end{array}
\end{equation}
This system is solved using the constructive implicit function theorem
presented in~\cite{sander:wanner:16a}, which is based on similar results
in~\cite{plum:09a, wanner:17a}. We state this theorem below.
For its application, one needs to establish
the following four assumptions:
\begin{itemize}
\item[(H1)] Assume that we have found a numerical approximation
$(\lambda^*,w^*) \in \R \times \cX$ of a solution of the
system~(\ref{tbcp:zeroeqn}). Then one needs to find an explicit
constant~$\rho > 0$ such that
\begin{displaymath}
  \left\| \cF(\lambda^*,w^*) \right\|_\cY \le \rho \; .
\end{displaymath}
\item[(H2)] Assume that the Fr\'echet derivative~$D_w\cF(\lambda^*,w^*)
\in \cL(\cX,\cY)$ is invertible, and that its inverse~$D_w\cF(\lambda^*,
w^*)^{-1} : \cY \to \cX$ is bounded and satisfies the estimate
\begin{displaymath}
  \left\| D_w\cF(\lambda^*,w^*)^{-1} \right\|_{\cL(\cY,\cX)} \le K
  \; ,
\end{displaymath}
for some explicit constant $K > 0$, where~$\| \cdot \|_{\cL(\cY,\cX)}$
denotes the operator norm in~$\cL(\cY,\cX)$.
\item[(H3)] There exist constants~$L_1, L_2, \ell_w > 0$ and $\ell_\lambda
\ge 0$ such that for all pairs $(\lambda,w) \in \R \times \cX$ with
$\| w - w^* \|_\cX \le \ell_w$ and $|\lambda - \lambda^*| \le \ell_\lambda$
we have
\begin{displaymath}
  \left\| D_w\cF(\lambda,w) -
    D_w\cF(\lambda^*,w^*) \right\|_{\cL(\cX,\cY)} \le
    L_1 \left\| w - w^* \right\|_\cX +
    L_2 \left|\lambda - \lambda^* \right| \; .
\end{displaymath}
\item[(H4)] There exist constants~$L_3, L_4 > 0$ such that for all
$\lambda \in \R$ with $|\lambda - \lambda^*| \le \ell_\lambda$ one has
\begin{displaymath}
  \left\| D_\lambda \cF(\lambda,w^*) \right\|_{\cY} \le
    L_3 + L_4 \left| \lambda - \lambda^* \right| \; ,
\end{displaymath}
where~$\ell_\lambda$ is the constant from~(H3).
\end{itemize}
The constructive implicit function theorem from~\cite{sander:wanner:16a}
then takes the following form.
\begin{theorem}[Constructive Implicit Function Theorem]
\label{nift:thm}
Let~$\cX$ and~$\cY$ denote the Hilbert spaces defined in~(\ref{tbcp:defcxcy}),
and let~$\cF : \R \times \cX \to \cY$ be defined as in~(\ref{tbcp:defcf}).
Furthermore, suppose that the pair~$(\lambda^*,w^*) \in \R \times \cX$ satisfies
hypotheses~(H1) through~(H4). Finally, suppose that
\begin{equation} \label{nift:thm1}
  4 K^2 \rho L_1 < 1
  \qquad\mbox{ and }\qquad
  2 K \rho < \ell_w \; .
\end{equation}
Then there exist pairs of constants~$(\delta_\lambda,\delta_w)$ with
$0 \le \delta_\lambda \le \ell_\lambda$ and $0 < \delta_w \le \ell_w$,
as well as
\begin{equation} \label{nift:thm2}
  2 K L_1 \delta_w + 2 K L_2 \delta_\lambda \le 1
  \qquad\mbox{ and }\qquad
  2 K \rho + 2 K L_3 \delta_\lambda + 2 K L_4 \delta_\lambda^2
    \le \delta_w  \; ,
\end{equation}
and for each such pair the following holds. For every~$\lambda \in \R$
with $|\lambda - \lambda^*| \le \delta_\lambda$ there exists a uniquely
determined element~$w(\lambda) \in \cX$ with $\| w(\lambda) - w^* \|_\cX
\le \delta_w$ such that $\cF(\lambda, w(\lambda)) = 0$.
In other words, if we define
\begin{displaymath}
  \cB_{\delta_w}^\cX = \left\{ w \in \cX \; : \;
    \left\| w - w^* \right\|_\cX \le \delta_w \right\}
  \quad\mbox{ and }\quad
  \cB_{\delta_\lambda}^\R = \left\{ \lambda \in \R \; : \;
    \left| \lambda - \lambda^* \right| \le \delta_\lambda \right\}
  \; ,
\end{displaymath}
then all solutions of the nonlinear problem $\cF(\lambda,w)=0$ in the
set $\cB_{\delta_\lambda}^\R \times \cB_{\delta_w}^\cX$ lie on the graph
of the function $\lambda \mapsto w(\lambda)$.  In addition, the function
$\lambda \mapsto w(\lambda)$ is infinitely-many times Fr\'echet
differentiable.
\end{theorem}
The above theorem is used for all of the results given in the next
Subsection~\ref{sec:rigvermic}. The following is a summary of how we 
approach each of the hypotheses~(H1)--(H4).
\begin{itemize}
\item The numerical approximation of a potential equilibrium state is
found using AUTO~\cite{doedel:81a} in the form of a finite Fourier
cosine sum, as described in more detail in Subsection~\ref{sec:fun}.
\item The residual bound~$\rho$ in~(H1) is computed using the specific
norms we use on~$\cY$. This is accomplished by evaluating a suitable sum
which depends on the Fourier coefficients in this representation using
the interval arithmetic package Intlab~\cite{rump:99a}, and makes use of
the equivalent Sobolev norms which will be described in
Subsection~\ref{sec:fun} below.
\item The inverse norm bound in~(H2) is established in
Section~\ref{sec:invnormbnd}. This estimate is derived in the 
broad context of~\eqref{intro:linop1} and~\eqref{intro:linop2}
given in Section~\ref{secintro}, which contains the triblock copolymer
model linearization as a special case. It relies heavily on the framework
developed in Subsections~\ref{sec:fun}, \ref{sec:sob}, and~\ref{sec:proj}.
\item The Lipschitz estimates given in~(H3) and~(H4) that
are required in the specific case of the triblock copolymer equation
are derived in Subsection~\ref{sec:lipschitz}.
\end{itemize}
The details of these more technical steps of the paper are contained
in Sections~\ref{sec:defandsetup} and~\ref{sec:invnormbnd}. First,
however, we present some sample results.
\subsection{Rigorously verified microstructures}
\label{sec:rigvermic}
In this section, we illustrate the methods of this paper by rigorously 
validating equilibrium solutions for the triblock copolymer equation 
for fixed values of~$\epsilon$, $\sigma$, and~$\mu$.  In particular, Figures~\ref{fig:tbcp1dimages},~\ref{fig:tbcp2dimages:a},~\ref{fig:tbcp2dimages:b},
and~\ref{fig:tbcp2dimages:c} show numerically computed approximations of solutions.
As mentioned before, in these figures, instead of using the parameter~$\epsilon$,
we give the results for fixed~$\lambda=1/\epsilon^2$. From the numerics alone, we
cannot guarantee that a qualitatively similar solution exists near the computed
solution, or if this solution is isolated. However in each case, and using the methods
established in the paper, we have rigorously established that for the
given parameters there exists a true solution to the triblock copolymer equation
within a known fixed distance of the depicted solution, and we specify the
distance in each case. Furthermore, we have also validated isolation of the solution. 
That is, the true solution is unique within a fixed radius ball around
the computed solution, which is also explicitly specified in each case. 

To implement the verification outlined at the end of the previous section
without a parameter search,
one must make a tradeoff. Specifically, when applying Theorem~\ref{thm:k},
one must find an integer $N \in \N$ and $\tau > 0$ such that a relation of the form
\begin{equation}
  \label{eq:thm-k-preview}
  \frac{1}{N^2}\sqrt{A(N)^2+B^2} \leq \tau < 1
\end{equation}
holds, where~$A(N)$ and~$B$ are constants, and~$A(N)$ depends on~$N$.
The inverse norm bound is then determined by dividing by $1-\tau$.
This division leads to a tradeoff in the number of basis functions per spatial dimension ($N$) 
and the desired sharpness of the norm bound.
In practice, we usually target $\tau \approx 0.75$ as this prevents unnecessary inflation
of the final bound while remaining computationally feasible.
In order to achieve these~$\tau$ values, 
we must find an~$N$ such that~\eqref{eq:thm-k-preview} holds.
This is difficult to do \emph{a priori} since~$A(N)$ depends on the choice
of~$N$ through a numerical computation. Thus, we make the following simplifying assumptions
for an initial estimate of~$N$:
\begin{enumerate}
  \item $A(N)$ is bounded above as $N$ increases
  \item $B$ dominates $A(N)$ so that $\sqrt{A(N)^2+B^2} \approx B$.
\end{enumerate}
The latter assumption provides a simple estimation for $N$ that would result in validation:
\begin{displaymath}
  \frac{B}{\tau} \leq N^2.
\end{displaymath}
We emphasize this approach simply avoids a computationally intense search for~$N$.
Once~$N$ is chosen, the value of~$A(N)$ can be computed directly as in Theorem~\ref{thm:k},
and~\eqref{eq:thm-k-preview} can be verified.
We take this simplified approach because the specific forms of~$A(N)$ and~$B$ in Theorem~\ref{thm:k}
satisfy the assumptions, and it provides a more computationally feasible value of~$N$ in general.
In one dimension, this is not an issue, as the 
calculation is very quick, a few seconds with Intlab 12 and Matlab 2020b on a Mac mini using an 
Intel processor with 3.2~GHz and 32~GB memory, under MacOS Monterey. 
However, for the two-dimensional case, this becomes a larger issue, since
the calculation of $K$ for  $\lambda = 10$ takes around 20 seconds, 
and for $\lambda = 20$ it takes 3-5 minutes. 
For $\lambda$ much larger, the required $N$ results in a full matrix that is too big to keep in memory. 
We note that these large $N$ values slow our implementation significantly because
we use full matrices at every step.
However, a vast majority of the entries in these matrices are small in magnitude. For example, 
in the cases we considered, the observed percentage of matrix entries below $10^{-16}$ in 
magnitude is about 97\%-98\% in one dimension, and in two dimensions, this becomes over 99\%.  
Thus one could assuredly speed up the implementation significantly by using 
sparse approximations. We leave this to a future effort.
\begin{table}
\begin{center}
  \begin{tabular}{|c|c|c||c|c|c|c|c|}
  \hline
  Label & $\mu$       & $\lambda$ & $N$   & $K$ & $(L_1,L_2,L_3)$ & $\delta_\alpha$  & $\delta_x$ \\
  \hline\hline
 4a & $(0.5,0.5,0.0)$  & 50  & 216 & 19.256  & $(2.79,1.45,0.24)$ &  9.3759e-04   & 8.8156e-03   \\ \hline
 4b & $(0.5,0.5,0.0)$  & 50  & 242 & 21.802  & $(2.71,1.33,0.24)$ &  7.7418e-04   & 8.0867e-03  \\ \hline
 4c & $(0.5,0.5,0.1)$  & 50  & 219 & 19.595  & $(2.99,1.49,0.26)$ &  8.1679e-04    & 8.3619e-03 \\ \hline
 4d & $(0.5,0.4,0.1)$  & 50  & 245 & 21.617  & $(2.76,1.40,0.25)$ & 7.3311e-04    & 8.0180e-03 \\ \hline      
  \end{tabular}
  \vspace*{0.3cm}
  \caption{\label{table:1d}
           Solution validation information for the one-dimensional solutions depicted in 
           Figure~\ref{fig:tbcp1dimages}. For all solutions, $\sigma  = 6$.
           Note that the values of $K$  and $L_i$  are computed at higher 
           precision than reported here. They are only given here for illustration of 
           how we establish the two $\delta$ values.  }
\end{center}
\end{table}
\begin{figure}
  \begin{center}
  \includegraphics[width=0.45\textwidth]{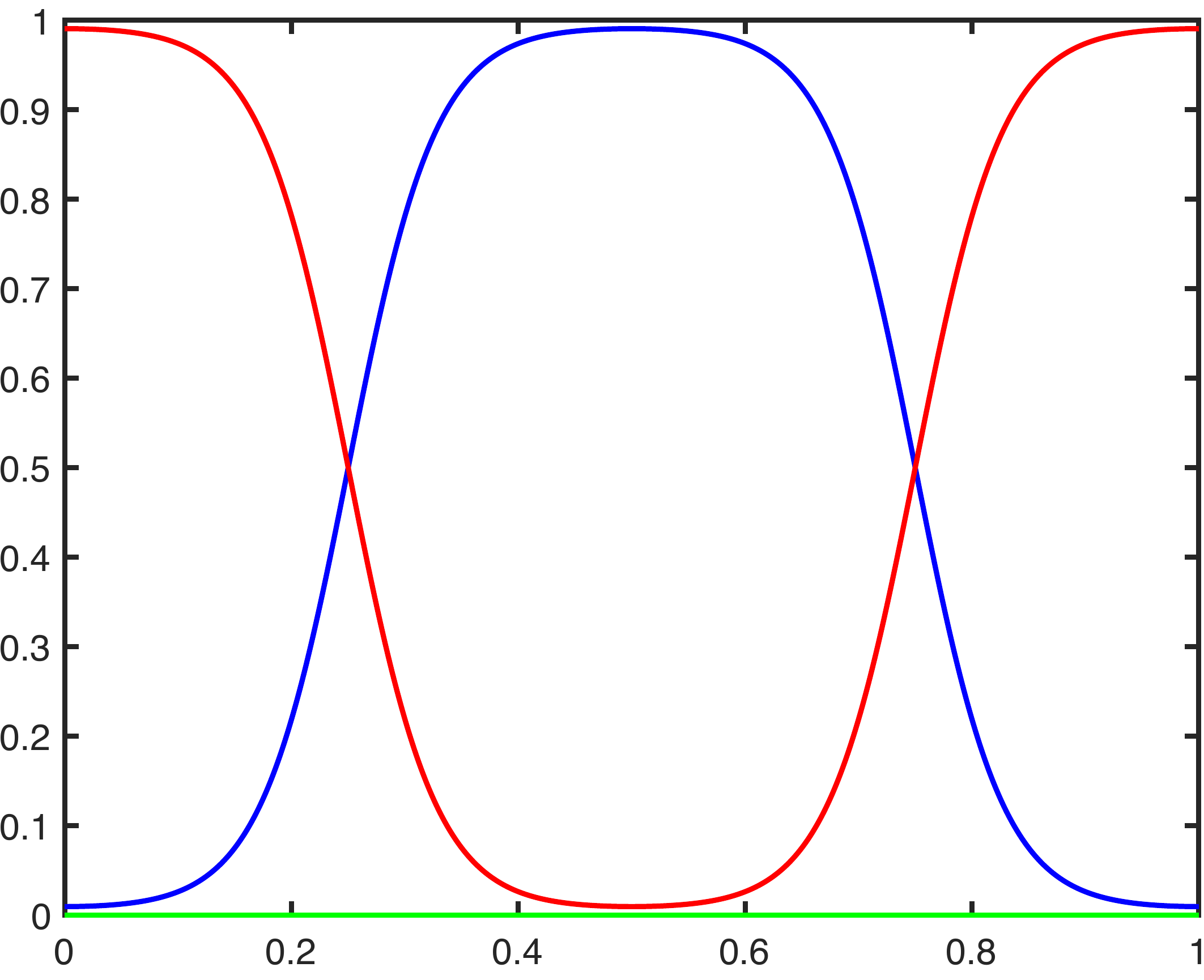}
  \hspace*{0.3cm}
  \includegraphics[width=0.45\textwidth]{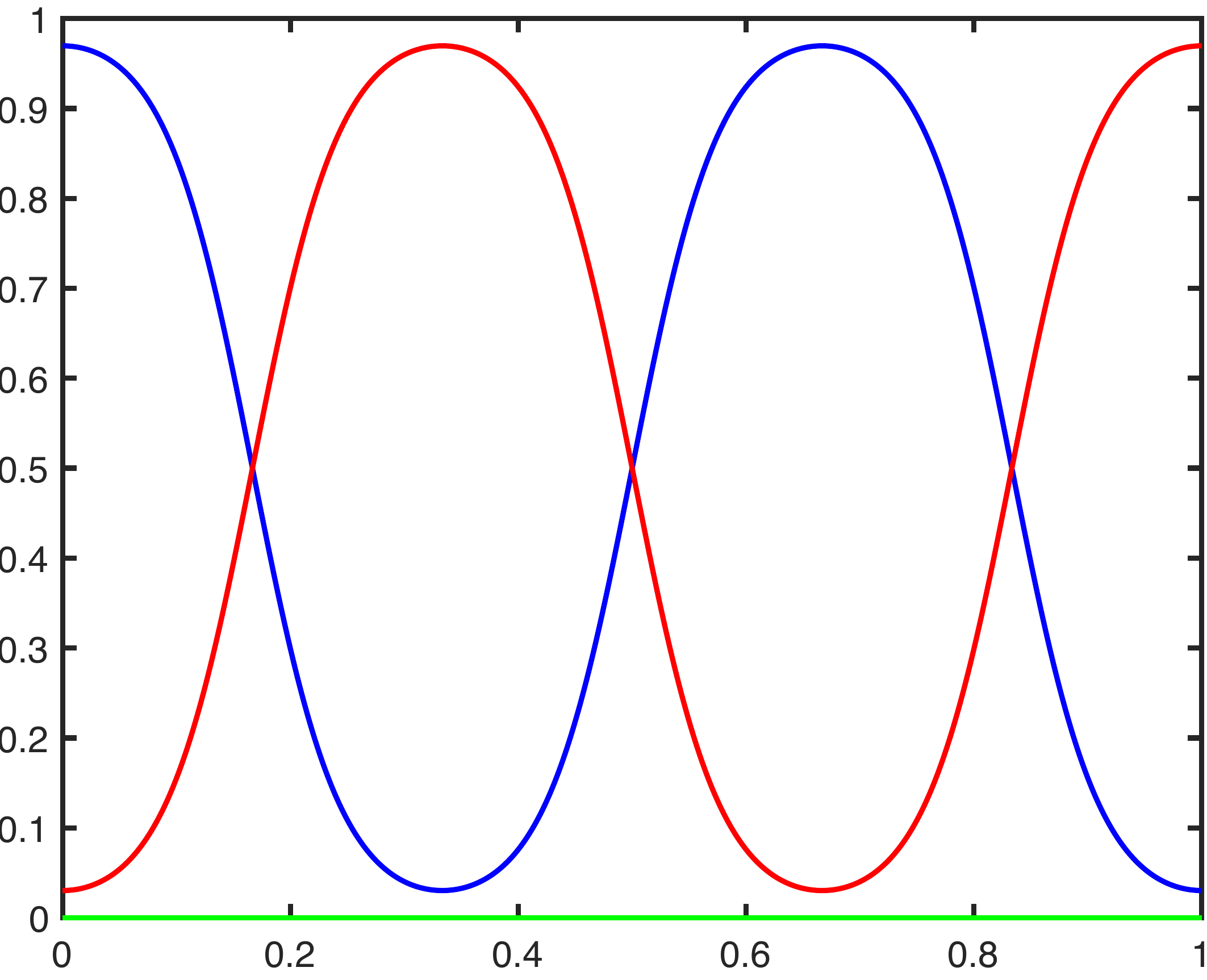} \\ 
  \includegraphics[width=0.45\textwidth]{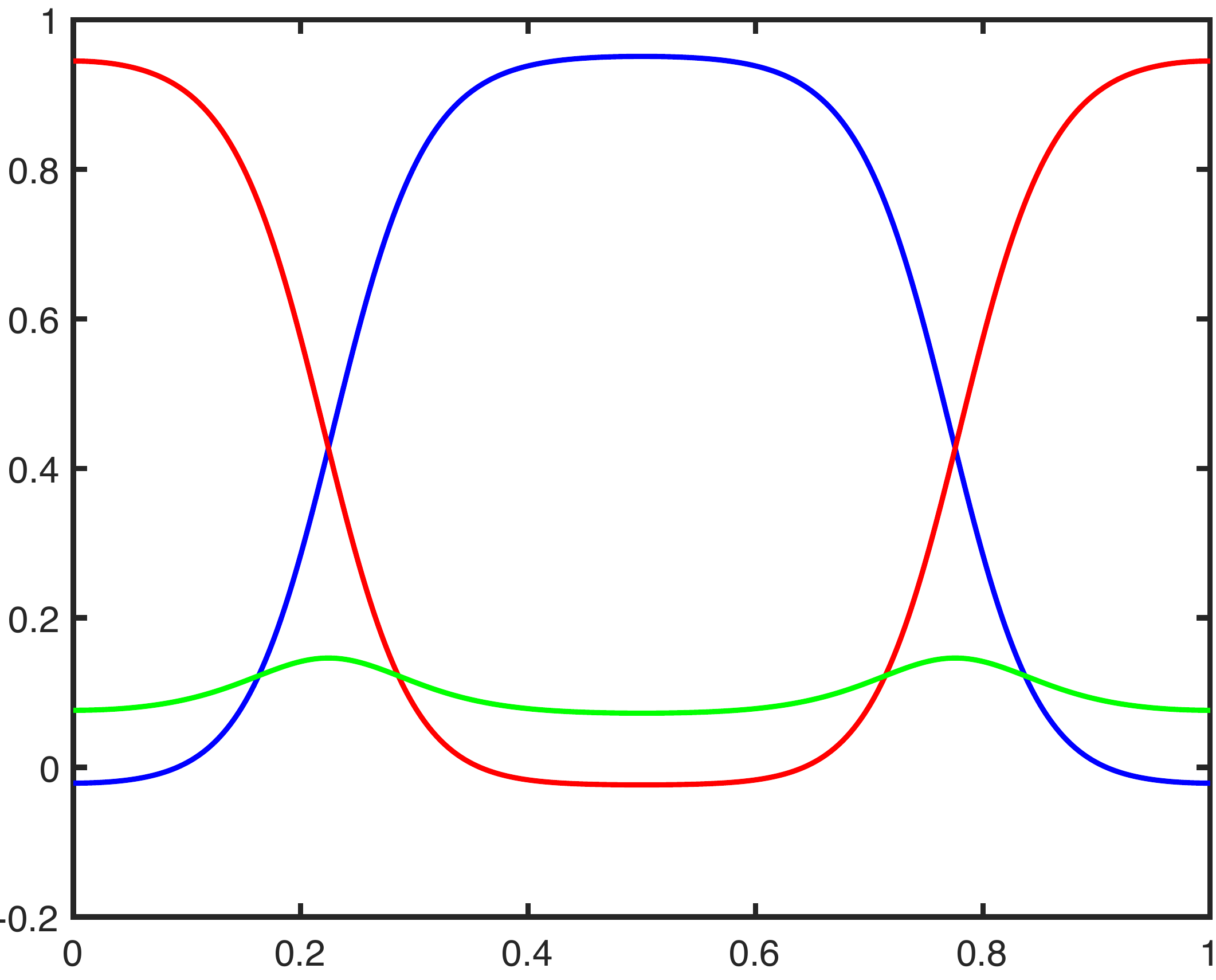}
    \hspace*{0.3cm}
  \includegraphics[width=0.45\textwidth]{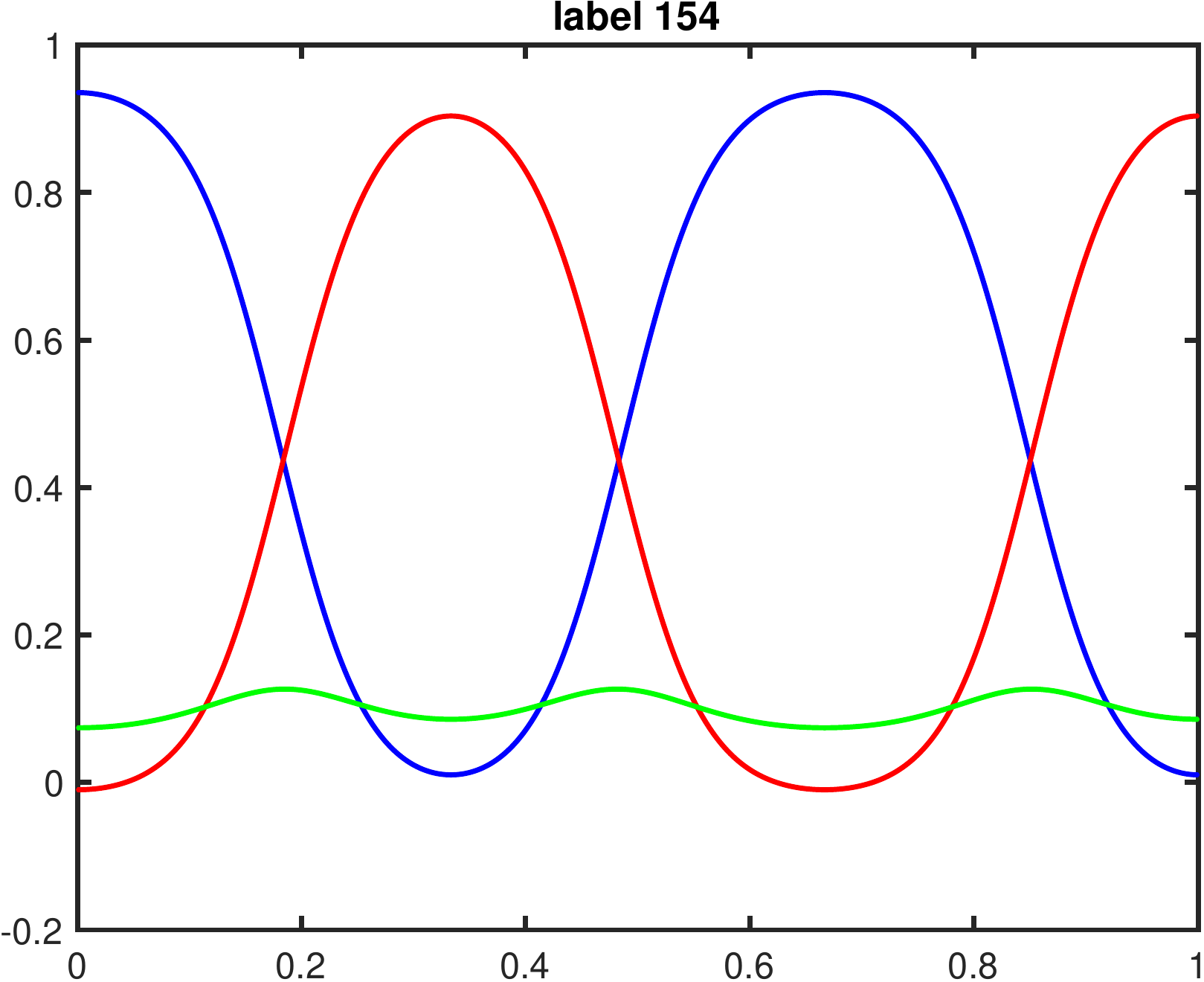}
  \caption{\label{fig:tbcp1dimages}
    Sample validated triblock copolymer equilibrium solutions on the one-dimensional domain $\Omega = (0,1)$.
    In all figures, we have chosen $\sigma=6$ and $\lambda = 50$. 
    The top row (left to right: a,b) is for $\mu = (0.5,0.5,0.0)$, while the bottom row (c,d) 
    has $\mu = (0.5,0.4,0.1)$.
    In all plots, the solutions~$u_1$, $u_2$, and~$u_3$ are shown in
    blue, red, and green, respectively. The validation parameters are listed in Table~\ref{table:1d}.}
  \end{center}
\end{figure}

We now describe the selection of example equilibria for which we have rigorously
verified stationary solutions. Figure~\ref{fig:tbcp1dimages} shows a set of equilibria
which have been computed for the~$\mu$ values $(0.5,0.5,0)$ and $(0.5,0.4,0.1)$, located 
in the light blue region for the case of  the one-dimensional domain $\Omega = (0,1)$.
In this case,  we consider the parameter $\lambda = 50$. For all of the shown equilibrium solutions, 
we consider the
parameter value $\sigma = 6$. Notice that for $\mu_3 = 0$ the
triblock copolymer equation reduces to the diblock copolymer model. Therefore, these
solutions are respectively equal to and close to cases previously studied in the 
case of two monomer blocks. In fact, even for $\mu_3 = 0.1$ the triblock copolymer
model behaves similarly to the two-component case, and Figure~\ref{fig:bifdiag1d}
illustrates this in the left image by showing the bifurcation diagram for the second~$\mu$ value above
in one space dimension. Note that the diagram bears a striking resemblance
to the bifurcation diagram for the diblock copolymer equation, see~\cite{johnson:etal:13a,
lessard:sander:wanner:17a}. To emphasize that this similarity is based on 
$\mu$ rather than the one spatial dimension, the right image of the figure shows a bifurcation diagram 
for one space dimension such that $\mu = (0.4,0.2,0.4)$, i.e.
 all three components of $\mu$ are far from 
zero. The diagram  is a significant departure from 
the diblock copolymer case.
\begin{table}
\begin{center}
  \begin{tabular}{|c|c|c||c|c|c|c|c|}
  \hline
  Label & $\mu$       & $\lambda$ & $N$   & $K$ & $(L_1,L_2,L_3)$ & $\delta_\alpha$  & $\delta_x$ \\
  \hline\hline
 5a & $(0.3,0.2,0.5)$  	 & 10    & 46  & 968.48 & $(1.06,1.54,0.308)$&    8.5889e-07  & 4.9838e-04  \\ \hline
 5b & $(0.3,0.2,0.5)$  	 & 20    & 86  & 7898.2   & $(1.98,2.01,0.296)$&    5.6899e-09  & 3.1894e-05  \\ \hline
 5c & $(0.3,0.2,0.5)$  	 & 20    & 98  & 434.29 & $(1.95,1.43,0.294)$&    2.3028e-06   & 5.8787e-04  \\ \hline
 6a & $(0.35,0.33,0.32)$ & 10    & 48  & 401.17 & $(0.974,1.55,0.277)$&   6.7908e-06  & 1.3841e-03   \\ \hline      
 6b & $(0.35,0.33,0.32)$ & 20    & 99  & 572.03 & $(2.17,1.85,0.296)$  &  1.2801e-06   & 4.1691e-04   \\ \hline      
 6c & $(0.35,0.33,0.32)$ & 20    & 97  &  674.44& $(2.04,1.85,0.285)$ &  9.9216e-07   & 3.7167e-04  \\ \hline      
 7a & $(0.4,0.2,0.4)$  	 & 10    & 40  &  238.08& $(0.731,0.82,0.248)$& 3.2873e-05   & 3.3207e-03   \\ \hline        
 7b & $(0.4,0.2,0.4)$  	 & 20    & 94  &  361.27& $(2.06,1.80,0.298)$ &  3.1087e-06   & 6.6886e-04   \\ \hline        
 7c & $(0.4,0.2,0.4)$  	 & 20    & 80  &  117.59& $(1.84,1.50,0.280)$ &   3.4630e-05  & 2.2780e-03   \\ \hline        
  \end{tabular}
  \vspace*{0.3cm}
  \caption{\label{table:2d}
           Solution validation information for the two-dimensional solutions depicted in 
           Fig.~\ref{fig:tbcp2dimages:a} and~\ref{fig:tbcp2dimages:b}. 
           For all solutions, $\sigma  = 6$. As with the one-dimensional case, the 
           values of $K$ and $L_i$ are not stated at the full precision that we computed.  }
\end{center}
\end{table}
\begin{figure}
  \begin{center}
  \includegraphics[width=0.32\textwidth]{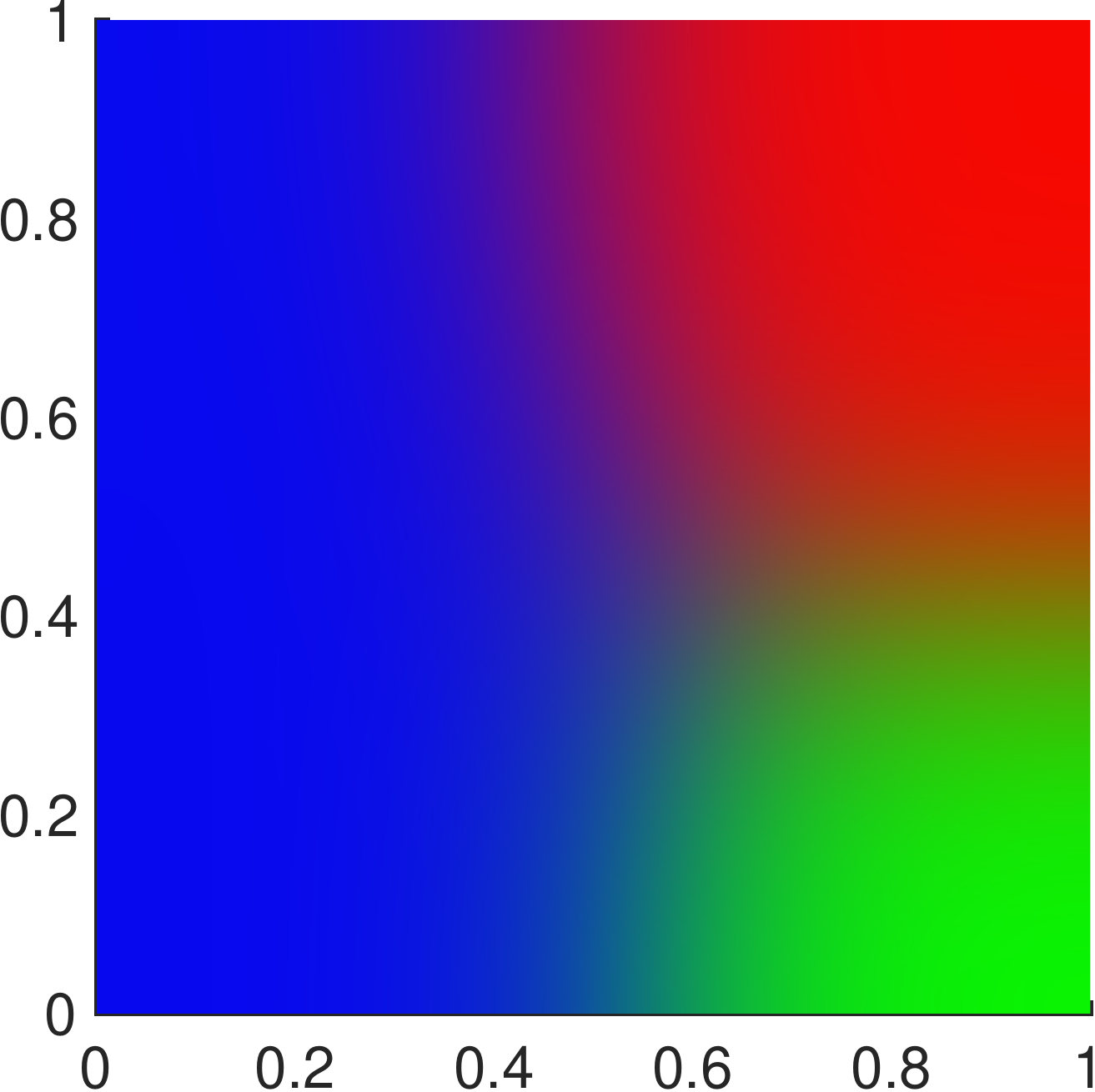}
  \includegraphics[width=0.32\textwidth]{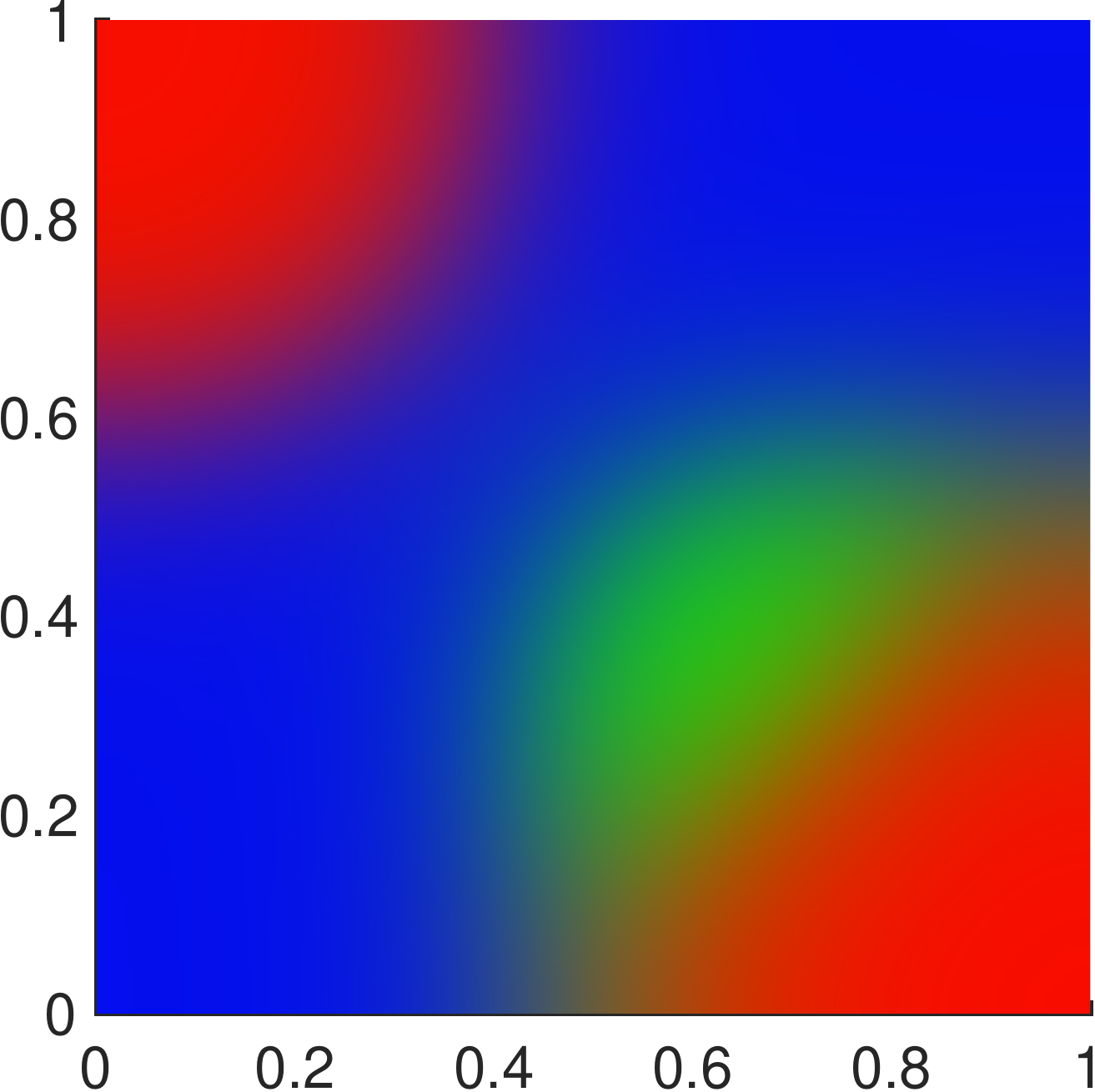} 
  \includegraphics[width=0.32\textwidth]{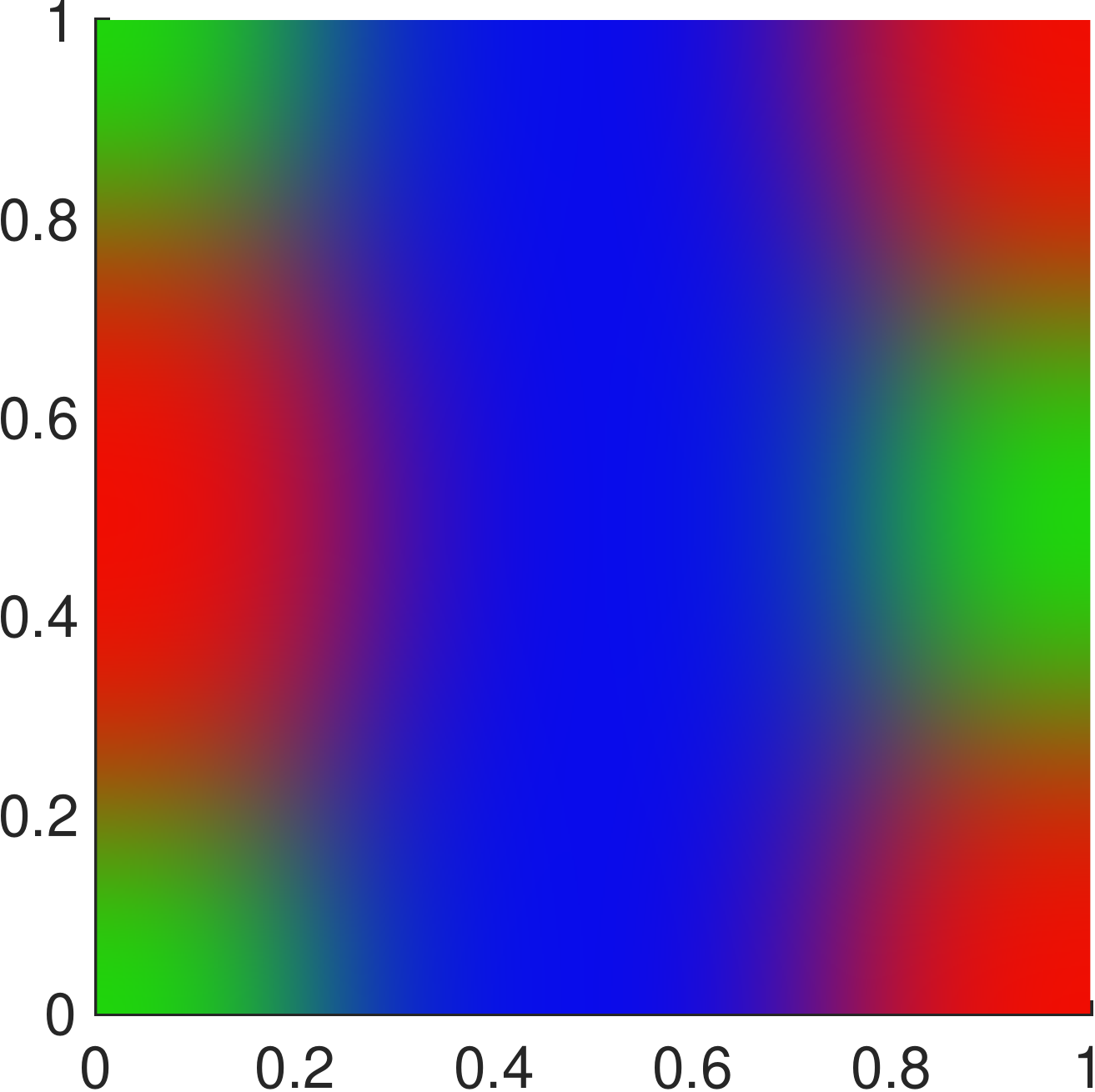}
  \caption{\label{fig:tbcp2dimages:a}
           Sample validated triblock copolymer equilibrium solutions on the
           two-dimensional domain $\Omega = (0,1)^2$ and for the parameters $\sigma=6$,
           $\mu = (0.3,0.2,0.5)$, as well as (left to right) (a)~$\lambda = 10$ and (b,c)~$\lambda=20$. 
           The validation parameters are listed in Table~\ref{table:2d}.
           }
  \end{center}
\end{figure}
\begin{figure}
  \begin{center}
  \includegraphics[width=0.32\textwidth]{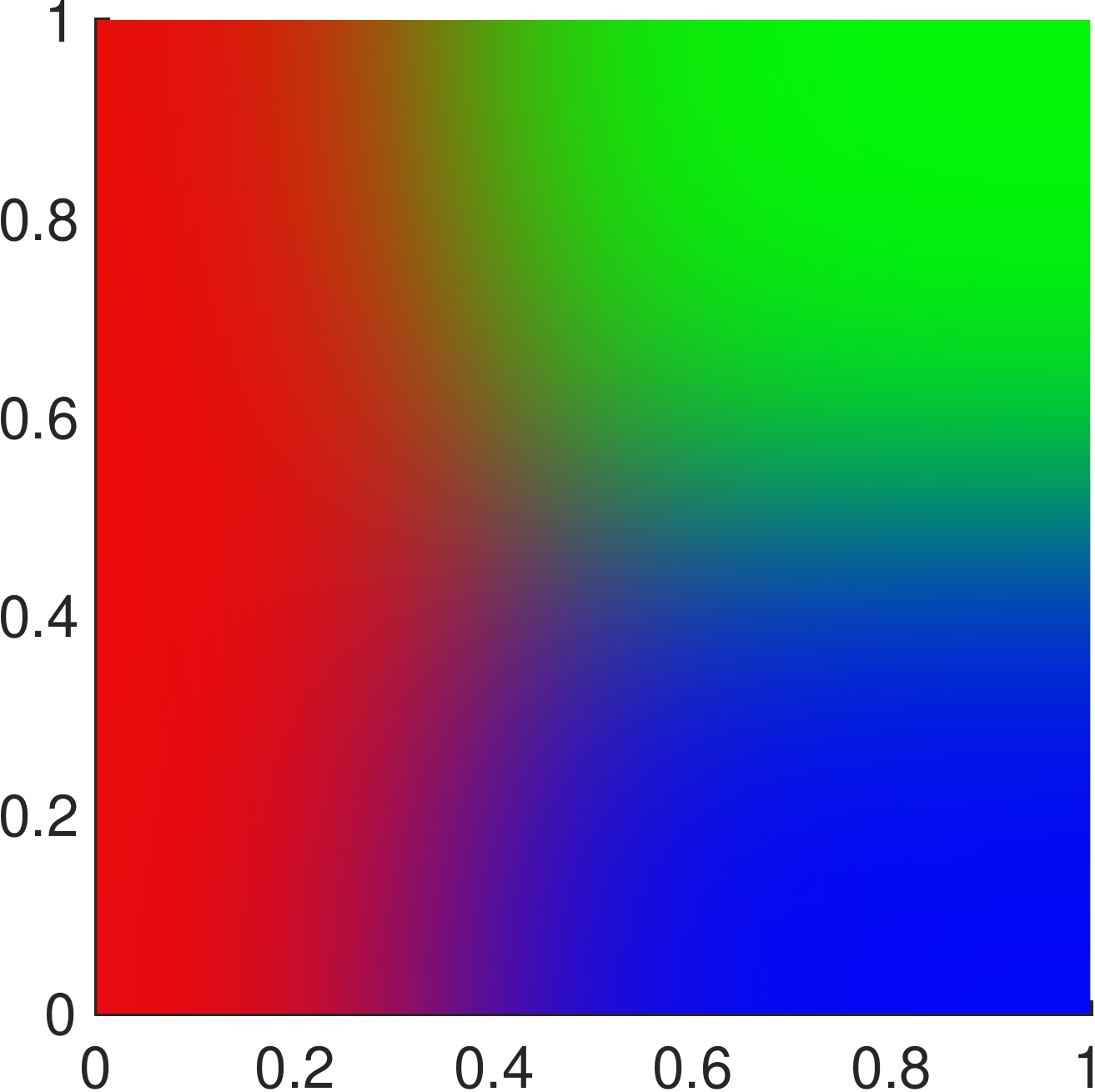}
  \includegraphics[width=0.32\textwidth]{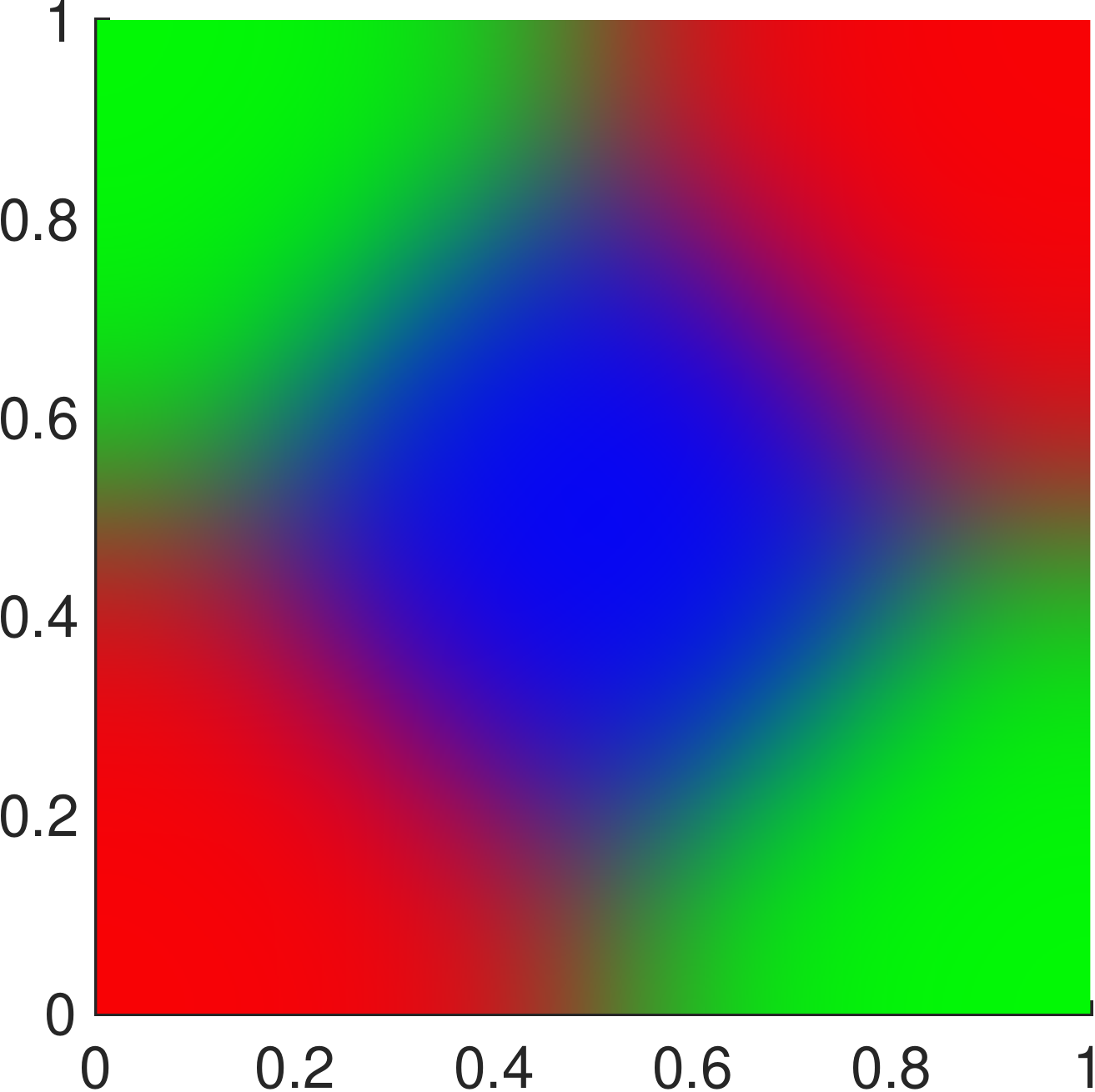} 
  \includegraphics[width=0.32\textwidth]{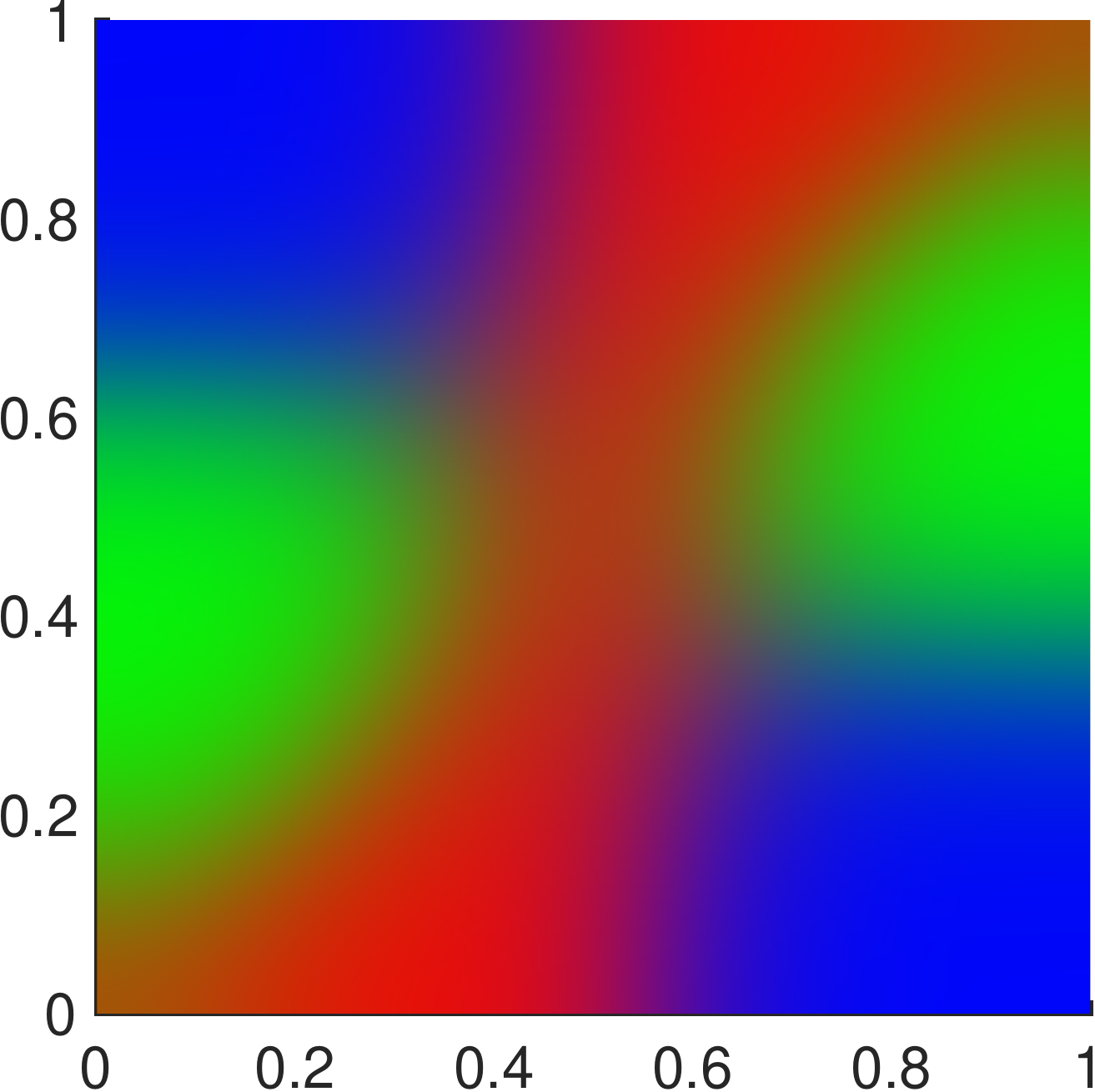}
  \caption{\label{fig:tbcp2dimages:b}
           Sample validated triblock copolymer equilibrium solutions on the
           two-dimensional domain $\Omega = (0,1)^2$, where $\sigma=6$,
           $\mu = (0.35,0.33,0.32)$, and (a)~$\lambda = 10$ and (b,c)~$\lambda=20$.  
           The validation parameters are listed in Table~\ref{table:2d}.
           }
  \end{center}
\end{figure}
\begin{figure}
  \begin{center}
  \includegraphics[width=0.32\textwidth]{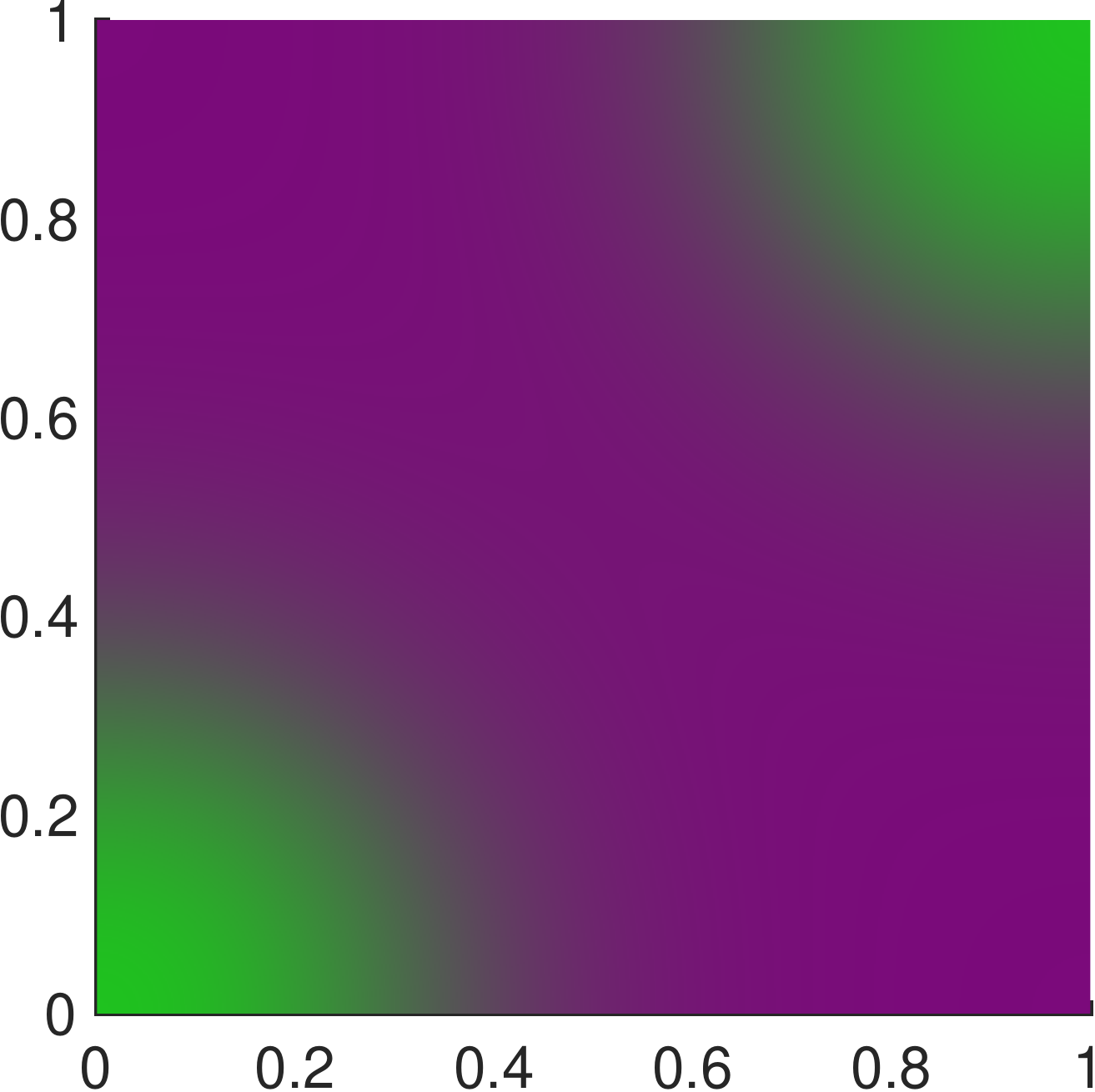}
  \includegraphics[width=0.32\textwidth]{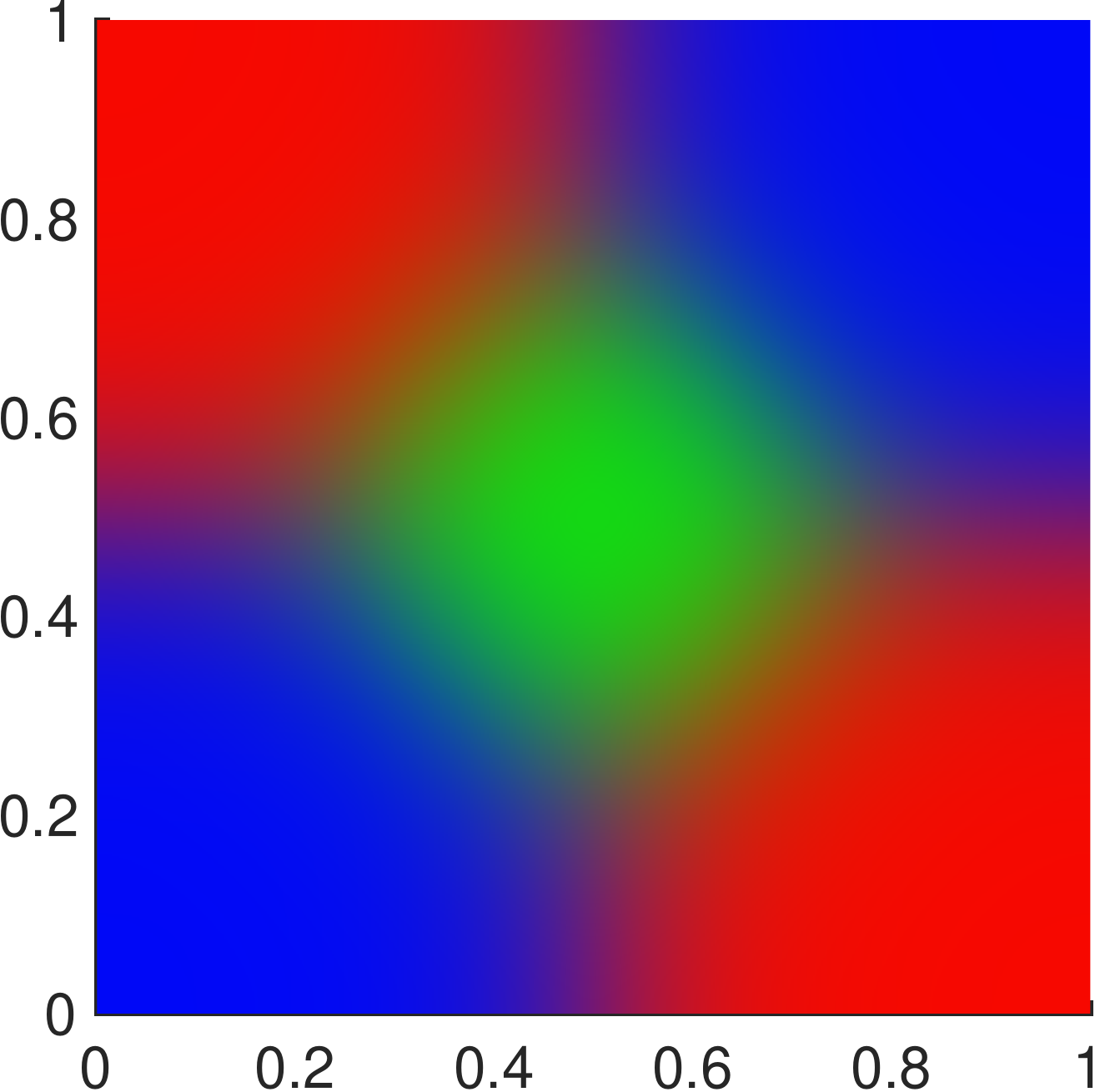} 
  \includegraphics[width=0.32\textwidth]{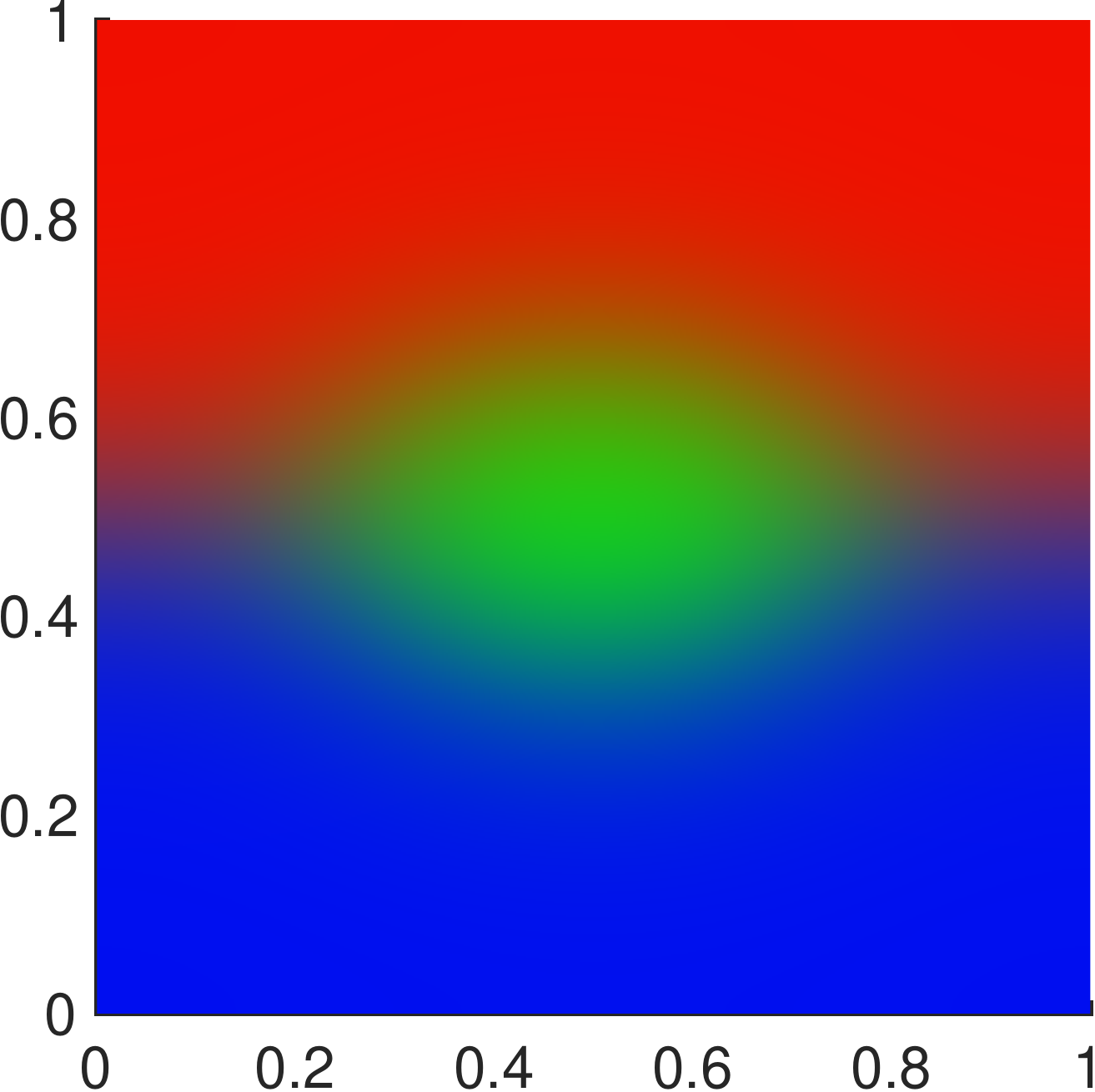}
  \caption{\label{fig:tbcp2dimages:c}
           Sample validated triblock copolymer equilibrium solutions on the
           two-dimensional domain $\Omega = (0,1)^2$, where $\sigma=6$,
           $\mu = (0.4,0.2,0.4)$, and (a)~$\lambda = 10$ and (b,c)~$\lambda=20$.  
           The validation parameters are listed in Table~\ref{table:2d}.}
  \end{center}
\end{figure}

We now turn to the case of the two-dimensional square domain $\Omega = (0,1)^2$.
Figures~\ref{fig:tbcp2dimages:a}--\ref{fig:tbcp2dimages:c}
depict equilibria for the triblock copolymer system for the three values of~$\mu$ in
the dark blue region of Figure~\ref{fig:stabregion}. Since all three components of~$\mu$
are significantly nonzero, these cases are quite different from the diblock copolymer
case. In order to depict these solutions, we have used the pointwise values of the three
components~$(u_1,u_2,u_3)$, each of which are basically values between~$0$ and~$1$, as the
RGB values of the resulting image. Therefore, a region which is primarily red corresponds
to a region primarily consisting of the first monomer, a green region mostly of the second
monomer, and a blue region consisting of the third monomer. In between the almost pure red,
green, and blue regions, there are narrow transition layers which usually appear as gray or
brown. These contain a mixture of multiple monomers. In fact 
mixed regions do not only occur in transition layers. For example, in 
Figure~\ref{fig:tbcp2dimages:c}a, there are two large regions of green and purple, 
where the purple region corresponds to a mixed monomer layer. 

Our primary focus in this paper is on establishing the framework necessary for rigorous
computation of equilibria, and thus we do not try to give an exhaustive set of connected branches
of stationary states. First of all, for each value of~$\mu$, there are an enormous number
of equilibrium solutions for each~$\lambda$ and~$\sigma$ value, as shown in Figure~\ref{fig:bifdiag}.
In this and all other bifurcation diagrams, we have fixed $\sigma  = 6$, but for nearby~$\sigma$
values, the complexity of the figure is consistently large. Therefore it is not realistically
tractable to create an exhaustive set of all equilibria. Second, in order to validate branches
of solutions, we would need to  combine the methods of this paper with the validated
pseudo-arclength continuation methods established in ~\cite{kamimoto:kim:sander:wanner:22a}.
This will also involve new development, since in the context of that paper, they were only
applied in a finite-dimensional case. The latter paper creates the first steps, via a more
flexible method of estimation. Rather than trying to do everything at once, a
systematic study combining these techniques will be the topic of a forthcoming paper.
\section{Functional-analytic framework and basic estimates}
\label{sec:defandsetup}
In this section, we describe the functional-analytic framework for
establishing stationary states of the triblock copolymer model using
the constructive implicit function theorem from Subsection~\ref{sec:capdetail}.
We will present the underlying spaces and norms, and recall necessary
auxiliary results and estimates. The results of the present section
reduce the equilibrium verification problem to the derivation of the
Fr\'echet derivative inverse operator norm bound, which is the central
result of this paper and will be established in the next section.

More precisely, we begin by discussing the necessary function spaces in
Subsection~\ref{sec:fun}, which form the foundation for our spectral
approach based on Fourier cosine series. The following
Subsection~\ref{sec:sob} recalls a number of rigorous Sobolev
embedding results that originated in~\cite{wanner:18b}, and results
which allow us to replace the standard Sobolev norms with more
computationally appropriate ones. Subsection~\ref{sec:proj} is
devoted to the required finite-dimensional approximation spaces
and associated projection operators, which are used in our
computer-assisted proofs. Finally, in Subsection~\ref{sec:lipschitz}
we derive the necessary Lipschitz estimates for the Fr\'echet derivatives
of the underlying nonlinear operator. Once that is accomplished, the
only missing piece of the puzzle is then the norm bound for the
inverse, and this is left for Section~\ref{sec:invnormbnd} of the paper.
\subsection{Fourier cosine series expansions and Sobolev spaces}
\label{sec:fun}
As mentioned in the last section, the functional-analytic backdrop
for our equilibrium validation are the spaces~$\bar{H}_n^2(\Omega)$
and~$\bar{H}_n^{-2}(\Omega)$ introduced in~(\ref{tbcp:h2dot}). These
spaces are considered on the unit cube $\Omega = (0,1)^d$ in dimension
$d = 1,2,3$, and they incorporate both the zero mass constraint and
the homogeneous Neumann boundary conditions. Important for our spectral
approach is the fact that Fourier cosine series forms a complete 
orthogonal set in both spaces. To describe this in more detail,
define the constants $c_0 = 1$ and $c_\ell = \sqrt{2}$ for~$\ell
\in \N$. Furthermore, we will make use of multi-indices $k \in \N_0^d$
of the form $k = (k_1,\dots,k_d)$ and let
\begin{displaymath}
  c_k = c_{k_1} \cdot \ldots \cdot c_{k_d} \; .
\end{displaymath}
If one then defines
\begin{equation} \label{eqn:phik}
  \phi_k(x) = c_k \prod_{i=1}^d \cos ( k_i \pi x_i )
  \qquad\mbox{ for all }\qquad
  x = (x_1,\ldots,x_d) \in \Omega \; , 
\end{equation}
then the family $\{ \phi_k \}_{k \in \N_0^d}$ is a complete orthonormal
basis for the space~$L^2(\Omega)$. Thus, any measurable and square-integrable
function $u : \Omega \to \R$ can be written in terms of its Fourier cosine
series
\begin{equation} \label{eqn:ucossum}
  u(x) = \sum_{k \in \N_0^d} \alpha_k \phi_k(x) \; ,
\end{equation}
where the real numbers $\alpha_k \in \R$ are the Fourier cosine coefficients
of~$u$, and we have
\begin{displaymath}
  \| u \|_{L^2} = \left( \sum_{k \in \N_0^d} \alpha_k^2 \right)^{1/2} \; ,
\end{displaymath}
where~$\| \cdot \|_{L^2}$ denotes the standard $L^2(\Omega)$-norm on the
above domain~$\Omega$. To simplify notation, we further introduce the
abbreviations
\begin{displaymath}
  |k| = \left( k_1^2 + \dots + k_d^2 \right)^{1/2}
  \qquad\mbox{ and }\qquad
  |k|_\infty = \max( k_1,  \dots , k_d )
\end{displaymath}
to distinguish between the Euclidean and maximum norms of multi-indices.

Recall that each function~$\phi_k$ is an eigenfunction of the negative
Laplacian subject to homogeneous Neumann boundary conditions. The corresponding
eigenvalue is given by~$\kappa_k$, and is defined via the equations
\begin{equation} \label{def:kappak}
  -\Delta \phi_k = \kappa_k \phi_k
  \qquad\mbox{ with }\qquad
  \kappa_k = \pi^2 \left(k_1^2 + k_2^2 + \dots + k_d^2\right) =
    \pi^2 |k|^2 \; .
\end{equation}
Notice also that every function~$\phi_k$ satisfies the identity
$\partial (\Delta \phi_k) /\partial \nu = -\kappa_k  \partial
\phi_k /\partial \nu = 0$, i.e., any finite Fourier cosine series
as above automatically satisfies both imposed boundary conditions
of the triblock copolymer equation~(\ref{intro:tbcpsys1}).

It will be useful to think of our basic function spaces in terms
of the Fourier cosine series representation in~(\ref{eqn:ucossum}).
Thus, for $\ell \in \N$ we consider the space 
\begin{displaymath}
  \cH^\ell =
    \left\{ u = \sum_{k \in \N_0^d} \alpha_k \phi_k : \| u \|_{\cH^\ell}
    < \infty \right\}
  \qquad\mbox{ with }\qquad
  \|u\|_{\cH^\ell}^2 = \sum_{k \in \N_0^d} \left( 1 +
    \kappa_k^\ell \right) \alpha_k^2 \; ,
\end{displaymath}
where the latter identity is equivalent to
\begin{displaymath}
  \|u\|^2_{\cH^\ell} =
  \|u \|_{L^2}^2 + \left\| (-\Delta)^{\ell/2} u \right\| ^2_{L^2} \;,
\end{displaymath}
and the fractional Laplacian for odd~$\ell$ is defined using the
spectral definition. One can show that the above spaces are subspaces
of the standard Sobolev spaces~$H^k(\Omega) = W^{k,2}(\Omega)$, which
were discussed in~\cite{adams:fournier:03a}. In addition, notice that
for sufficiently large~$\ell$, their definition automatically incorporates
the boundary conditions of~(\ref{intro:tbcpsys1}). For example, we have
$\cH^0 = L^2(\Omega)$ and $\cH^1 = H^1(\Omega)$, as well as
\begin{displaymath}
  \cH^2 = \left\{ u \in H^2(\Omega): \frac{\partial u}{\partial \nu} = 0
    \right\}
  \quad\mbox{ and }\quad
  \cH^4 = \left\{ u \in H^4(\Omega): \frac{\partial u}{\partial \nu} =
    \frac{\partial \Delta u}{\partial \nu} = 0 \right\} \; ,
\end{displaymath}
where the boundary conditions in the last two equations are considered
in the sense of the trace operator. See~\cite{sander:wanner:21a} for
more details on these identities.

While the spaces~$\cH^\ell$ incorporate the boundary conditions
of~(\ref{intro:tbcpsys1}), in the last section we reformulated the triblock
copolymer model so that both solution components satisfy the integral
constraint $\int_\Omega u \; dx = 0$, since the case of nonzero mass
average has been absorbed into the placement of the parameters~$\mu_1$
and~$\mu_2$. This mass constraint can be incorporated by considering
suitable subspaces of~$\cH^\ell$. To this end, consider an arbitrary
integer~$\ell \in \Z$ and define the space
\begin{equation} \label{defhlbar}
  \overline{\cH}^\ell =
  \left\{ u = \sum_{k \in \N_0^d, \; |k|>0} \alpha_k \phi_k
    \; : \; \| u \|_{\overline{\cH}^\ell} < \infty \right\}
  \quad\mbox{ with }\quad
  \| u \|_{\overline{\cH}^\ell}^2 =
    \sum_{k \in \N_0^d, \; |k|>0} \kappa_k^\ell \alpha_k^2 \;.
\end{equation}
We would like to point out that in these reduced spaces, we
use a simpler norm than the one used in~$\cH^\ell$. For $\ell = 0$
this definition reduces to the subspace of~$L^2(\Omega)$ of all
functions with average zero, equipped with its standard norm, while
for $\ell > 0$ we have $\overline{\cH}^\ell \subset \cH^\ell$ and
the new norm is equivalent to the original norm on~$\cH^\ell$.
Moreover, note that in the case of a negative integer $\ell < 0$
the series in~(\ref{eqn:ucossum}) is interpreted formally, i.e.,
the element $u \in \overline{\cH}^\ell$ is identified with the
sequence of its Fourier cosine coefficients. One can verify
that in this case~$u$ acts as a bounded linear functional
on~$\overline{\cH}^{-\ell}$. In fact, for all~$\ell < 0$ the 
space~$\overline{\cH}^\ell$ can be considered as a subspace of
the negative exponent Sobolev space~$H^\ell(\Omega) =
W^{\ell,2}(\Omega)$, see again~\cite{adams:fournier:03a}.
Finally, for every $\ell \in \Z$ the space~$\overline{\cH}^\ell$
is a Hilbert space with inner product
\begin{equation} \label{def:hbarellnorm}
  (u,v)_{\overline{\cH}^\ell} =
  \sum_{k \in \N_0^d, \; |k|>0} \kappa_k^\ell \alpha_k \beta_k \;, 
\end{equation}
where
\begin{displaymath}
  u = \sum_{k \in \N_0^d, \; |k|>0} \alpha_k \phi_k
    \in \overline{\cH}^\ell
  \qquad\mbox{ and }\qquad
  v = \sum_{k \in \N_0^d, \; |k|>0} \beta_k \phi_k
    \in \overline{\cH}^\ell \;. 
\end{displaymath}
Being separable Hilbert spaces, the spaces~$\overline{\cH}^\ell$ do
have complete orthonormal sets. The most important one for us is the
one given via rescalings of the functions~$\phi_k$, which is identified
in the following lemma. Its straightforward proof is left to the reader.
\begin{lemma}[Complete orthonormal set in~$\overline{\cH}^\ell$]
\label{lem:orthoghell}
For every $\ell \in \Z$ a complete orthonormal set in the Hilbert
space~$\overline{\cH}^\ell$ is given by the family
$\left\{ \kappa_k^{-\ell/2} \phi_k(x) \right\}_{k \in \N_0^d,
\; |k|>0}$.
\end{lemma}
The above spaces are the foundation for our functional-analytic
setting. Notice that using these spaces, we can equivalently
reformulate the equilibrium system~(\ref{tbcp:equilsysW}), as
written in~(\ref{tbcp:zeroeqn}) and~(\ref{tbcp:defcf}), using
the space notation in~(\ref{tbcp:defcxcy}). It is clear from
our above discussion that we have
\begin{equation} \label{tbcp:defcxcy2a}
  \cX = \overline{\cH}^2 \times \overline{\cH}^2
  \qquad\mbox{ and }\qquad
  \cY = \overline{\cH}^{-2} \times \overline{\cH}^{-2} \; ,
\end{equation}
where on these product spaces we use the norms
\begin{equation} \label{tbcp:defcxcy2b}
  \left\| (w_1,w_2) \right\|_{\cX}^2 =
    \| w_1 \|_{\overline{\cH}^2}^2 +
    \| w_2 \|_{\overline{\cH}^2}^2
  \quad\mbox{ and }\quad
  \left\| (w_1,w_2) \right\|_{\cY}^2 =
    \| w_1 \|_{\overline{\cH}^{-2}}^2 +
    \| w_2 \|_{\overline{\cH}^{-2}}^2 \; . 
\end{equation}
Notice that the nonlinear problem $\cF(\lambda,w) = 0$ is now formulated
weakly, and in particular, the second boundary condition $\partial (\Delta w_i)
/ \partial \nu = 0$ for $i = 1,2$ is no longer explicitly stated in this weak
formulation. Note, however, that the first boundary conditions $ \partial w_i /
\partial \nu = 0$ have been incorporated into the space~$\cX$. Furthermore,
the fact that the functions~$f_1$ and~$f_2$ in~(\ref{tbcp:defcf}) are both
of class~$C^2$ is sufficient to guarantee that the function~$\cF : \R \times
\cX \to \cY$ is well-defined and Fr\'echet differentiable, since we only
consider domains up to dimension three.
\subsection{Constructive Sobolev embeddings and norm bounds}
\label{sec:sob}
We now turn our attention to a number of auxiliary results which relate
the norms of the spaces from the last subsections to each other, as well
as to other norms. Needless to say, all of these results need to be
explicit with concrete bounds, since they will be used in a constructive
computer-assisted proof setting. We begin by recalling two classical
results concerning Sobolev spaces --- namely the Sobolev embedding theorem
and the Banach algebra estimate in the Sobolev space of order two. These
results relate the norms on the function spaces~$\overline{\cH}^2$
and~$\cH^2$ to each other, as well as to the classical infinity norm.
As a side result, we obtain that all functions in~$\cH^2$ are in fact
continuous functions on~$\overline{\Omega}$, and that~$\cH^2$ is closed
under multiplication. These results are essential for the results of the
next section.
%
%
\begin{lemma}[Sobolev embeddings and Banach algebra estimates]
\label{lem:sobolev}
Consider the Hilbert spaces~$\cH^2$ and~$\overline{\cH}^2$ from
the last subsection, which are defined over the unit cube
$\Omega = (0,1)^d$ for dimensions $d = 1,2,3$. Then the
following statements hold:
\begin{itemize}
\item[(a)] {\bf Sobolev embedding:} For all $u \in \cH^2$ and arbitrary
$\bar{u} \in \overline{\cH}^2$ the estimates
\begin{displaymath}
  \| u \|_\infty \; \le \; C_m \; \| u \|_{\cH^2}
  \qquad\mbox{ and }\qquad
  \| \bar{u} \|_\infty \; \le \; \overline{C}_m \;
    \| \bar{u} \|_{\overline{\cH}^2}
\end{displaymath}
are satisfied, where the constants~$C_m$ and~$\overline{C}_m$ can be
found in Table~\ref{table1}, and~$\| \cdot \|_\infty$ denotes the
supremum norm in~$L^\infty(\Omega)$. In particular, these estimates
show that every function in~$\cH^2$ is almost everywhere equal to a
continuous function on~$\overline{\Omega}$.
\item[(b)] {\bf Banach algebra estimate:} For all $u,v \in \cH^2$
we have
\begin{displaymath}
  \| u v \|_{\cH^2} \; \le \;
  C_b \; \| u \|_{\cH^2} \| v \|_{\cH^2} \; ,
\end{displaymath}
where the constant~$C_b$ can be found in Table~\ref{table1}.
In other words, the Sobolev space~$\cH^2$ is closed under
multiplication.
\item[(c)] {\bf Explicit norm equivalence:} For all $\bar{u}
\in \overline{\cH}^2$ we have
\begin{displaymath}
  \| \bar{u} \|_{\overline{\cH}^2} \; \le \;
  \| \bar{u} \|_{\cH^2} \; \le \;
  C_e \| \bar{u} \|_{\overline{\cH}^2}
  \qquad\mbox{ with }\qquad
  C_e = \frac{\sqrt{1 + \pi^4}}{\pi^2} \; . 
\end{displaymath}
\end{itemize}
\end{lemma}
The proofs for the first inequality in~{\em (a)\/} and the inequality
in~{\em (b)\/} can be found in~\cite{wanner:18b}. The remaining statements
were established in~\cite{sander:wanner:21a}. Many of these estimates 
were themselves obtained via computer-assisted proofs, see again the
mentioned references.
\begin{table}
\begin{center}
    \begin{tabular}{|c||c|c|c|}
    \hline
    Dimension $d$ & $1$ & $2$ & $3$ \\ \hline\hline
    Sobolev Embedding Constant $C_m$ & $1.010947$ &
      $1.030255$ & $1.081202$  \\ \hline
    Sobolev Embedding Constant $\overline{C}_m$ &
      $0.149072$ & $0.248740$ & $0.411972$ \\ \hline
    Banach Algebra Constant $C_b$ &
      $1.471443$ & $1.488231$ & $1.554916$ \\ \hline  
    \end{tabular}
    \vspace*{0.3cm}
    \caption{\label{table1} 
    The table contains the explicit values for the constants
    introduced in Lemma~\ref{lem:sobolev}, depending on the domain
    dimension~$d$. They were derived using rigorous computational
    techniques in~\cite{sander:wanner:21a,wanner:18b}.}
\end{center}
\end{table}

Our next and final result of this subsection discusses the
relation between the spaces~$\overline{\cH}^\ell$ for varying
values of the differentiation order~$\ell$, i.e., we discuss
the so-called scale of these spaces. More precisely, it shows that,
on the one hand, due to our norm choices the Laplacian acts as an
isometry between spaces of appropriate differentiation orders. On
the other hand, it provides explicit embedding constants from spaces
with larger differentiation order to ones with smaller order. The proof
of the following lemma can be found in~\cite{sander:wanner:21a}.
\begin{lemma}[Sobolev scale properties]
\label{lem:sobolevscale}
Consider the Hilbert spaces~$\cH^\ell$ for $\ell \in \Z$ from 
the last subsection, which are defined over the unit cube
$\Omega = (0,1)^d$ for $d = 1,2,3$. Then the following
statements hold:
\begin{itemize}
\item[(a)] {\bf Laplacian isometry:} For every integer $\ell \in \Z$ the
Laplacian operator~$\Delta$ is an isometry from~$\overline{\cH}^\ell$
to~$\overline{\cH}^{\ell-2}$, i.e., for all $u \in \overline{\cH}^\ell$
the identities
\begin{displaymath}
  \| \Delta^{-1} u \|_{\overline{\cH}^{\ell+2}} \; = \;
  \| u \|_{\overline{\cH}^\ell} \; = \;
  \| \Delta u \|_{\overline{\cH}^{\ell-2}}
\end{displaymath}
are satisfied.
\item[(b)] {\bf Scale embeddings:} For all $u \in \overline{\cH}^m$ and
all $\ell \le m$ we have the estimate
\begin{displaymath}
  \| u \|_{\overline{\cH}^\ell} \; \le \;
  \frac{1}{\pi^{m-\ell}} \, \| u \|_{\overline{\cH}^m} \; . 
\end{displaymath}
Furthermore, note that in the special case $\ell = 0 \le m$ we have
$\| u \|_{\overline{\cH}^0} = \| u \|_{L^2}$.
\end{itemize}
\end{lemma}
\subsection{Spectral projection operators}
\label{sec:proj}
We now turn our attention to the finite-dimensional approximation
spaces that will be used in our computer-assisted existence proofs
for equilibrium solutions of~(\ref{intro:tbcpsys1}). These turn out
to simply be generated by truncated cosine series, and this is
briefly recalled in the present subsection via suitable
projection operators.

For this, let~$N \in \N$ denote a positive integer, and
consider $u \in \cH^\ell$ for $\ell \in \N_0$, or alternatively
$u \in \overline{\cH}^\ell$ for $\ell \in \Z$, of the form
$u = \sum_{k \in \N_0^d} \alpha_k \phi_k$, where in the latter
case $\alpha_0 = 0$. Then as in~\cite{sander:wanner:21a} we define
the projection
\begin{equation} \label{eqn:defpn}
  P_N u = \sum_{k \in \N_0^d, \; |k|_\infty< N} \alpha_k \phi_k \; . 
\end{equation}
In this definition we use the $\infty$-norm of the multi-index~$k$,
since this simplifies the implementation aspects of our method. The
so-defined operator~$P_N$ is a bounded linear operator on~$\cH^\ell$
with induced operator norm equal to~$1$, and it leaves the
space~$\overline{\cH}^\ell$ invariant if $\ell \in \Z$. Moreover,
one can easily show that for any $N \in \N$ we have
\begin{equation}
  \label{eqn:dimpn}
  \dim P_N {\cH}^\ell = N^d
  \qquad\mbox{ and }\qquad
  \dim P_N \overline{\cH}^\ell = N^d - 1 \;. 
\end{equation}
Notice also that for all $\ell \in \N_0$ the identity 
$(I - P_1) {\cH}^\ell = \overline{\cH}^\ell$ holds. Since this is
an especially useful operator, we introduce the abbreviation
\begin{equation}
  \overline{P} = I - P_1 \; . 
\end{equation}
It was shown in~\cite{sander:wanner:21a} that this operator~$\overline{P}$
satisfies the identity
\begin{equation} \label{eqn:pbarl2prod}
  \left( \overline{P} u, v \right)_{L^2} \; = \; (u,v)_{L^2}
  \qquad\mbox{ for all }\qquad
  u \in \cH^0
  \quad\mbox{ and }\quad
  v \in \overline{\cH}^0 \; .
\end{equation}
To close this subsection, we present a norm bound for the infinite Fourier
cosine series part that is discarded by the projection~$P_N$ in terms
of a higher-regularity norm. More precisely, we have the following
result, whose proof can again be found in~\cite{sander:wanner:21a}.
\begin{lemma}[Projection tail estimates]
\label{lem:projtailest}
Consider two integers $\ell \le m$ and let $u \in \overline{\cH}^m$
be arbitrary. Then the projection tail~$(I - P_N)u$ satisfies the
estimate
\begin{displaymath}
  \| (I - P_N) u \|_{\overline{\cH}^\ell} \; \le \;
  \frac{1}{\pi^{m - \ell} N^{m - \ell}} \,
    \| (I - P_N) u \|_{\overline{\cH}^m} \; \le \;
  \frac{1}{\pi^{m - \ell} N^{m - \ell}} \,
    \| u \|_{\overline{\cH}^m} \;. 
\end{displaymath}
\end{lemma}

\subsection{Lipschitz bounds for the Fr\'echet derivatives}
\label{sec:lipschitz}
To close this section, we now turn our attention to the Lipschitz
bounds which are required in hypotheses~(H3) and~(H4) of the
constructive implicit function theorem. The basic idea for their
derivation is the same as in~\cite{sander:wanner:21a}, and it
makes use of the explicit form of the Fr\'echet derivatives
of~$\cF$ with respect to~$w$ and~$\lambda$, combined with a
suitable version of the mean value theorem and our estimates
from Subsection~\ref{sec:sob}. In this way, we obtain the
following result.
\begin{lemma}[Lipschitz bounds for the Fr\'echet derivatives of~$\cF$]
\label{lem:lipschitzfrechet}
Consider the nonlinear triblock copolymer operator~$\cF : \R \times \cX \to \cY$
defined in~(\ref{tbcp:defcf}), between the spaces introduced in~(\ref{tbcp:defcxcy}).
Then both Hypotheses~(H3) and~(H4) are satisfied with the explicit constants
\begin{equation}\label{eqn:lipconstants}
  \begin{array}{cclcccl}
    L_1 &=& \DS \frac{2^{3/2} \overline{C}_m (|\lambda^*|+\ell_\lambda)
      f^{(2)}_{\max}}{\pi^2} \; , & \qquad\quad &
    L_2 &=& \DS \frac{f^{(1)}_{*}}{\pi^2} + \frac{\sigma}{\pi^4} \; , \\[2ex]
    L_3 &=& \DS \frac{2 f^{(1)}_{\max}}{\pi^{2}} + \frac{\sigma}{\pi^{4}}
      \; , \;\;\mbox{ and } & \qquad & 
    L_4 &=& 0 \; , 
  \end{array}
\end{equation}
where the values~$f^{(2)}_{\max}$, $f^{(1)}_{\max}$, and~$f^{(1)}_{*}$ are
defined in~(\ref{eqn:lipconstants2}) and(\ref{eqn:lipconstants3}) below,
and the value of~$\overline{C}_m$ can be found in Table~\ref{table1}.
\end{lemma}
\begin{proof}
In the following, 
we recall that the~$\mu_i$ are constants representing the total mass of the $i$-th monomer
and that~$w_i$ are defined as $u_i - \mu_i$ (i.e., the zero-mass component of~$u_i$)
to reformulate~\eqref{tbcp:equilsys} as~\eqref{tbcp:equilsysW}.
For brevity, we use the abbreviation $\mu + w = (\mu_1 + w_1,
\mu_2 + w_2)$, and we denote the Jacobian matrix of~$f = (f_1,f_2)$
at a point~$z \in \R^2$ by~$Df(z) = (\nabla f_1(z), \nabla f_2(z))^t$.
Recall that the Fr{\'e}chet derivative of the nonlinear operator~$\cF$
is then explicitly given by
\begin{eqnarray*}
  D_w \cF(\lambda,w)[\tilde{w}] & = &
    -\Delta (\Delta \tilde{w} + \lambda Df (\mu + w) \tilde{w})
    - \lambda \sigma \tilde{w} \\[1ex]
  & = &  \left( -\Delta \left( \Delta \tilde{w}_1 + \lambda 
    \nabla f_1(\mu + w) \cdot \tilde{w} \right) -
    \lambda \sigma \tilde{w}_1 \; , \right. \\[0.7ex]
  & & \quad \left. -\Delta \left( \Delta \tilde{w}_2 + \lambda
    \nabla f_2(\mu + w) \cdot \tilde{w} \right) -
    \lambda \sigma \tilde{w}_2 \right) \; .  
\end{eqnarray*}
In the following proof, we will make frequent use of the results from
the last three subsections. For a pair of functions $w = (w_1,w_2)$
in either~$\cX$ or~$\cY$ defined in~\eqref{tbcp:defcxcy2a}, we use
the respective norms given in~\eqref{tbcp:defcxcy2b}. In addition,
and parallel to these definitions, we define the norms~$\|\cdot\|_\cZ$
by $\|w\|_{\cZ}^2 = \|w_1\|_{{\cH}^0}^2  + \|w_2\|_{{\cH}^0}^2$,
and the norm~$\|\cdot\|_\cI$ by $\| w \|_{\cI}^2 = \|w_1\|_\infty^2
+ \|w_2\|_{\infty}^2$. Finally, if~$w^*$ denotes the solution 
approximation from the constructive implicit function theorem, then
we define the set $R = \{z \in \R^2 \; : \; \| z \| \le \|w^*\|_\cI
+ \overline{C}_m \ell_w \}$, the constants
\begin{equation}\label{eqn:lipconstants2}
  f^{(1)}_{\max} \; = \;
    \max_{\substack{i,j=1,2 \\ z \in R}}
    \left| \frac{\partial f_i}{ \partial z_j} (z+\mu) \right|
  \quad\mbox{ and }\quad
  f^{(2)}_{\max} = \max_{\substack{i,j,k=1,2 \\ z \in R}}
    \left| \frac{\partial^2 f_i}{\partial z_k \partial z_j} (z+\mu)
    \right| \; ,
\end{equation}
as well as finally the constant
\begin{equation} \label{eqn:lipconstants3}
  f^{(1)}_{*} \; = \;
  \max_{i,j=1,2} \left \Vert
    \frac{\partial f_i}{ \partial z_j} (w^*+\mu)
    \right\Vert_\infty \; .
\end{equation}
Consider now any two scalar norms~$\|\cdot\|_{s_1},
\|\cdot\|_{s_2}$ and two vector norms~$\|\cdot\|_{S_1},
\|\cdot\|_{S_2}$ which are related by the identity
$\| (w_1,w_2) \|_{S_i}^2 = \|w_1\|_{s_i}^2 + \| w_2\|_{s_i}^2$.
Assume further that for every scalar function~$u$ such that
the~$s_i$ norms are well-defined one has the estimate  
$\|u\|_{s_1} \le C \|u\|_{s_2}$ for some constant~$C > 0$.
Then the corresponding estimate is satisfied with unchanged~$C$
also in the vector-valued case, i.e., one has $\|w\|_{S_1}
\le C \|w\|_{S_2}$. Therefore, we can use the norm bounds
relating the spaces~$\overline{\cH}^{-2}$, $\overline{\cH}^2$,
$\overline{\cH}^0$, and~$C(\overline{\Omega})$ given in
Lemmas~\ref{lem:sobolev} and~\ref{lem:sobolevscale} to
establish norm bounds relating the spaces~$\cX$, $\cY$,
$\cZ = {\cH}^0 \times {\cH}^0$, and~$\cI = C(\overline{\Omega})
\times C(\overline{\Omega})$, respectively. In particular,
Lemma~\ref{lem:sobolev}{\em (a)\/} and
Lemma~\ref{lem:sobolevscale}{\em (a),(b)\/} imply the four
statements
\begin{eqnarray}
  \|w\|_{\cI} \le \overline{C}_m \|w\|_{\cX}
  & \quad\mbox{ and }\quad &
  \| \Delta w \|_{\cY} = \| w \|_{\cZ} \; ,
    \quad\mbox{ as well as }
    \label{lem:lipschitzfrechet12} \\[1ex]
  \| w \|_{\cZ} \le \pi^{-2} \| w \|_{\cX}
  & \quad\mbox{ and }\quad &
  \| w \|_{\cY} \le \pi^{-4} \| w \|_{\cX} \; .
    \label{lem:lipschitzfrechet34}
\end{eqnarray}
In preparation for the verification of the actual Lipschitz
estimate of the theorem, we consider a smooth function
$h: \R^m \to \R^m$, and let~$Dh(z)$ denote the Jacobian
matrix of~$h$ at~$z \in \R^m$. Moreover, consider two
points $z,\hat{z} \in \R^m$, let $y \in \R^m$, and let~$\cD$
denote the line segment between~$z$ and~$\hat{z}$. Then the
mean value theorem applied to the $k$-th component of~$h$
yields
\begin{displaymath}
  \left| h_k(z)  - h_k(\hat{z}) \right| \le
  \max_{c \in \cD} \left\Vert \nabla h_k(c) \right\Vert_2
    \| z-\hat{z} \|_2 \le
  \sqrt{m} \max_{\substack{j=1,\ldots,m \\ c \in \cD}}
    \left\vert \frac{\partial h_k}{\partial z_j}(c)
    \right\vert \| z-\hat{z} \|_2 \; ,
\end{displaymath}
and thus
\begin{equation} \label{lem:lipschitzfrechet5}
  \| h(z) - h(\hat{z}) \|_2 \le
  m \max_{\substack{j,k=1,\ldots,m \\ c \in \cD}}
    \left\vert \frac{\partial h_k}{\partial z_j}(c)
    \right\vert  \| z-\hat{z} \|_2 \; .
\end{equation}
In addition, we have
\begin{displaymath}
  \left\| (Dh(z) - Dh(\hat{z})) y \right\|_2^2 =
  \sum_{k = 1}^m \left( (\nabla h_k(z)-\nabla h_k(\hat{z}))
  y \right)^2 \; .
\end{displaymath}
Notice that
\begin{displaymath}
  \left| (\nabla h_k(z)-\nabla h_k(\hat{z})) y \right| \le
  \left\| \nabla h_k(z)-\nabla h_k(\hat{z}) \right\|_2
  \| y \|_2 \; ,
\end{displaymath}
as well as
\begin{displaymath}
  \left\lvert \frac{\partial h_k}{\partial z_j}(z) -
    \frac{\partial h_k}{\partial z_j}(\hat{z}) \right\rvert \le
    \sqrt{m} \max_{\substack{i=1,\ldots,m \\ c \in \cD}}
      \left\lvert \frac{\partial^2 h_k}{\partial z_i
      \partial z_j}(c) \right\rvert \| z-\hat{z} \|_2 \; ,
\end{displaymath}
and therefore
\begin{displaymath}
  \|\nabla h_k(z)-\nabla h_k(\hat{z})\|_2 \le
  m \max_{\substack{i,j=1,\ldots,m \\ c \in \cD}}
    \left\lvert \frac{\partial^2 h_k}{\partial z_i
    \partial z_j}(c) \right\rvert \| z-\hat{z} \|_2
    \; .
\end{displaymath}
This finally implies
\begin{equation}  \label{lem:lipschitzfrechet6}
  \begin{array}{ccc}
    \DS \left\| (Dh(z) - Dh(\hat{z})) y \right\|_2
      & \le & \DS m \left( \sum_{k = 1}^m
      \max_{\substack{i,j=1,\ldots,m \\ c \in \cD}}
      \left\lvert \frac{\partial^2 h_k}{\partial
      z_i \partial z_j}(c) \right\rvert ^2 \right)^{1/2}
      \| z-\hat{z} \|_2 \| y \|_2 \\[5ex]
    & \le & \DS m^{3/2} \max_{\substack{i,j,k=1,\ldots,m \\
      c \in \cD}} \left\lvert \frac{\partial^2 h_k}{\partial
      z_i \partial z_j}(c) \right\rvert
      \| z-\hat{z} \|_2 \|y\|_2 \; . 
  \end{array}
\end{equation}
Note that the above computations are similar in spirit to the
ones found in~\cite{kamimoto:kim:sander:wanner:22a}. 

After these preparations, we finally turn our attention to
the Lipschitz estimates of the theorem. From the explicit form
of the Fr\'echet derivative~$D_w \cF$ one obtains
\begin{displaymath}
  \begin{array}{rcl}
    & & \hspace*{-2cm}
     \DS \left\| D_w\cF(\lambda,w) \tilde{w} -
     D_w \cF(\lambda^*,w^*)\tilde{w} \right\|_{\cY} \\[1.5ex]
   & = & \DS
     \left\| -\Delta \left( \lambda Df(\mu + w) \tilde{w} -
     \lambda^* Df(\mu + w^*) \tilde{w}  \right) -
     \left( \lambda - \lambda^* \right) \sigma \tilde{w}
     \right\|_{\cY} \\[1.5ex]
    & \le & \DS
      \left| \lambda-\lambda^* \right| \left( \left\|\Delta
      \left(  Df(\mu + w) \tilde{w} \right) \right\|_{\cY} +
      \sigma \| \tilde{w} \|_{\cY} \right) \\[1ex]
    & & \DS \quad \; + \;
      \left| \lambda^* \right| \| \Delta \left( \left(
      Df(\mu + w) - Df(\mu + w^*) \right) \tilde{w}
      \right) \|_{\cY} \; .
  \end{array}
\end{displaymath}
Then the second statement in~\eqref{lem:lipschitzfrechet12},
together with the observation that the components
of~$Df(\mu + w) \tilde{w}$ do not necessarily have total
mass~$0$, yields
\begin{eqnarray*}
  \left\Vert \Delta \left( Df (\mu + w)  \tilde{w} \right)
    \right\Vert_{\cY} & \le &
    \left\Vert Df(\mu + w) \tilde{w} \right\Vert_{\cZ} \\[1ex]
  & \le & \left\Vert (Df(\mu + w) - Df(\mu + w^*)) \tilde{w}
    \right\Vert_{\cZ} + \left\Vert Df(\mu + w^*) \tilde{w}
    \right\Vert_{\cZ} \; . 
\end{eqnarray*}
Let $\xi(x) \in \R^2$ be a point on the line segment between
the vectors~$w(x)$ and~$w^*(x)$. Then one can bound~$\|\xi\|_\cI$
via
\begin{displaymath}
  \|\xi\|_\cI  \; \le \;
  \|w^*\|_\cI + \| w - w^*\|_\cI \; \le \;
  \|w^*\|_\cI + \overline{C}_m \| w - w^*\|_\cX \; \le \;
  \|w^*\|_\cI + \overline{C}_m \ell_w \; . 
\end{displaymath}
We would like to point out that this last inequality implies
that for the region~$R$ used in the definitions of
both~$f^{(1)}_{\max}$ and~$f^{(2)}_{\max}$ one therefore
obtains~$\xi(x) \in R$. Together with~\eqref{lem:lipschitzfrechet12},
\eqref{lem:lipschitzfrechet34}, and~\eqref{lem:lipschitzfrechet6}
this furnishes
\begin{eqnarray*}
  \left\Vert (Df(w + \mu ) - Df(w^*+\mu)) \tilde{w}
    \right\Vert_{\cZ} & \le & 2^{3/2} f^{(2)}_{\max}
    \|w-w^*\|_\cI \|\tilde{w}\|_\cZ \\[2ex]
  & \le & \frac{2^{3/2} \overline{C}_m
    f^{(2)}_{\max}}{\pi^2} \; \|w-w^*\|_\cX
    \|\tilde{w}\|_\cX \; . 
\end{eqnarray*}
Additionally, we have
\begin{displaymath}
  \left\Vert Df(\mu + w^*) \tilde{w} \right\Vert_{\cZ} \; \le \;
  \max_{\substack{i,j=1,2 \\ x \in \Omega}} \left\Vert
    \frac{\partial f_i}{\partial z_j}(\mu + w^*(x))
    \right\Vert_\infty \|\tilde{w}\|_\cZ \; \le \;
 f^{(1)}_{*} \|\tilde{w}\|_\cZ \; \le \;
 \frac{f^{(1)}_{*}}{\pi^2} \; \|\tilde{w}\|_\cX \; .
\end{displaymath}
Combining the above statements along with the statements
in~\eqref{lem:lipschitzfrechet34}, we further see that
\begin{displaymath}
  \begin{array}{rcl}
    & & \hspace*{-2cm} \DS
      \| D_w \cF(\lambda,w) \tilde{w} -
      D_w \cF(\lambda^*,w^*) \tilde{w} \|_{\cY} \\[1.5ex]
    & \le & \DS
      \left( \frac{f^{(1)}_{*}}{\pi^2} + \frac{\sigma}{\pi^4}
      \right) |\lambda - \lambda^*| \left\| \tilde{w}
      \right\|_{\cX} \; + \;
      \left( \frac{2^{3/2}\overline{C}_m (|\lambda^*|+\ell_\lambda)
      f^{(2)}_{\max}}{\pi^2} \right) \left\| w - w^* \right\|_{\cX}
      \left\| \tilde{w} \right\|_{\cX} \; ,
  \end{array}
\end{displaymath}
which immediately establishes the values of~$L_1$ and~$L_2$
in hypothesis~(H3). As for the condition in~(H4), we recall that
\begin{displaymath}
  D_{\lambda} \cF(\lambda,w) \; = \;
  -\Delta(f(\mu+w))-\sigma w \; , 
\end{displaymath}
and using the estimate in~\eqref{lem:lipschitzfrechet5} one
further obtains
\begin{eqnarray*}
  \| D_{\lambda}\cF(\lambda,w) - D_{\lambda}\cF(\lambda^*,w^*)
    \|_{\cY} & \le &
    \left\| \Delta(f(\mu+w)-f(\mu+w^*)) \right\|_{\cY} +
    \sigma \| w-w^* \|_{\cY} \\[1.5ex]
  & \le & \left\| f(\mu+w)-f(\mu+w^*) \right\|_{\cZ} +
    \frac{\sigma}{\pi^4} \|w-w^*\|_{\cX} \\[1.5ex]
  & \le & \left( \frac{2 f^{(1)}_{\max}}{\pi^2} +
    \frac{\sigma}{\pi^4} \right) \|w-w^*\|_{\cX} \; .
\end{eqnarray*}
This finally establishes the values for~$L_3$ and~$L_4$,
and completes the proof of Lemma~\ref{lem:lipschitzfrechet}.
\end{proof}
%
%
%
%
\newcommand{\subspaceSymbol}{U}%
\newcommand{\rangeSubspaceSymbol}{V}%
\newcommand{\indexsetSymbol}{\mathcal{J}}%
\newcommand{\subspaceIndex}{i}%
\newcommand{\indexIndex}{k}%
\newcommand{\nscalars}{m}%
\newcommand{\nfunctions}{n}%
\newcommand{\htwo}{\cH^{2}}%
\newcommand{\htwobar}{\overline{\cH}^{2}}%
\newcommand{\hmtwobar}{\overline{\cH}^{-2}}%
\newcommand{\hzerobar}{\overline{\cH}^{0}}%
\newcommand{\basisSymbol}{\cB}%
\newcommand{\embeddingConstantToHtwobar}{{\red{C_{?}}}}%
\newcommand{\Ldomain}{\mathbf X}%
\newcommand{\LdomainElement}{\mathbf x}
\newcommand{\LdomainElementScalar}{\eta}
\newcommand{\LdomainElementFun}{v}
\newcommand{\Lrange}{\mathbf Y}%
\newcommand{\LrangeElement}{\mathbf y}
\newcommand{\suitableFunctionSpace}{\htwobar}%
\section{Inverse norm bound for fourth-order elliptic operators}
\label{sec:invnormbnd}
This section is devoted to establishing an inverse bound 
for the operator~$L$ defined in~\eqref{intro:linop1} and~\eqref{intro:linop2}.
This bound can be used in various applications to obtain
hypothesis~(H2), which is required for Theorem~\ref{nift:thm}, the
constructive implicit function theorem. More precisely, our goal in
the following is to derive a constant~$K$ such that 
\begin{displaymath}
  \left\| L^{-1} \right\|_{\cL(\Lrange,\Ldomain)} \le K \; ,
\end{displaymath}
i.e., we need to find a bound on the operator norm of the inverse
of the linear operator $ L $. We divide 
the derivation of this estimate into four parts. In
Section~\ref{sec:investoutline} we give an outline of our approach,
introduce necessary definitions and auxiliary results, and present
the main result of this section. This result will be verified in the
following three sections. First, we discuss the finite-dimensional
projection of~$L$ in Section~\ref{sec:findim}. Using this
finite-dimensional operator, we then construct an approximate
inverse in Section~\ref{sec:approxinverse}, before everything is
assembled to provide the desired estimate in the final
Section~\ref{sec:comprof}. In contrast to the discussion of
Section~\ref{sec:capdetail}, we use a formulation where the
main space~$\Ldomain$ is a product space of~$ \nscalars $
scalar constraints and~$ \nfunctions $ subspaces of~$ \htwobar $,
namely $ \Ldomain = \R^\nscalars \times \prod_{i=1}^\nfunctions U_i $.
As mentioned in the introduction, this is in preparation of future
applications of this theory, which go well beyond the triblock
copolymer model.
\subsection{General outline and auxiliary results}
\label{sec:investoutline}
For every $ \subspaceIndex = 1, \ldots, \nfunctions $, let 
$ \subspaceSymbol_\subspaceIndex \subset \htwobar $ 
be a closed subspace and let
$ \indexsetSymbol_\subspaceIndex $ 
denote an infinite index set consisting of multi-indices such
that $ \left\{ \phi_{\indexIndex} : \indexIndex \in \indexsetSymbol_\subspaceIndex \right\} $ 
forms a complete orthogonal set of $ U_\subspaceIndex $, where the considered
basis functions~$\phi_k$ were introduced in~\eqref{eqn:phik}.
We emphasize that $ \subspaceSymbol_\subspaceIndex $ is not necessarily all of $ \htwobar $,
but it is critical to have a complete orthogonal set for each 
$ \subspaceSymbol_\subspaceIndex $ which consists of a subset of the basis functions in \eqref{eqn:phik}.
We may now form a complete orthogonal set for $ \Ldomain = \R^\nscalars \times \prod_{\subspaceIndex=1}^{\nfunctions} \subspaceSymbol_\subspaceIndex $
by using the standard basis $ \left\{ e_j \right\}_{j=1}^{\nscalars} $ for $ \R^\nscalars $ 
and $ \left\{ \phi_{\indexIndex} : \indexIndex \in \indexsetSymbol_\subspaceIndex \right\} $
for every $ \subspaceIndex = 1,\ldots,\nfunctions $
as
\begin{equation}
    \basisSymbol_\indexsetSymbol = \left\{ e_j \times 0_{\left(\htwobar\right)^\nfunctions} \right\}_{j=1 \ldots \nscalars} \bigcup \; \left\{ 0_{\R^\nscalars \times \left(\htwobar\right)^{i-1}} \times \phi_{\indexIndex} \times 0_{\left(\htwobar\right)^{\nfunctions-i}} \right\}_{\subspaceIndex=1 \ldots \nfunctions,\,\indexIndex \in \indexsetSymbol_\subspaceIndex}
\end{equation}
where $ \indexsetSymbol = \left\{ \indexsetSymbol_1, \indexsetSymbol_2, \ldots, \indexsetSymbol_\nfunctions \right\} $.
For convenience of notation in the subsequent discussion, 
for every element $ \LdomainElement \in \Ldomain $
we abbreviate the operator defined in \eqref{intro:linop1} and \eqref{intro:linop2} by
$ L : \Ldomain \to \Lrange $ where 
\begin{equation}  \label{eqn:ldef}
  \Ldomain = \R^\nscalars \times \prod_{\subspaceIndex=1}^{\nfunctions} \subspaceSymbol_\subspaceIndex 
  \qquad\text{and}\qquad
  \Lrange = \R^\nscalars \times \prod_{\subspaceIndex=1}^{\nfunctions} \rangeSubspaceSymbol_\subspaceIndex
\end{equation}
with the assumption that $ \rangeSubspaceSymbol_\subspaceIndex \subset \hmtwobar$ also has the complete orthogonal set
$ \left\{ \phi_k : k \in \indexsetSymbol_\subspaceIndex \right\} $.
This is the most general form of the operator $ L $, and
standard results imply that~$L$ is a bounded linear
operator $L \in \cL(\Ldomain,\Lrange)$.

As mentioned earlier, the constructive implicit function theorem
crucially relies on being able to find a bound~$K$ such that
$\|L^{-1}\|_{\cL(\Lrange,\Ldomain)} \le K$. Our goal is to
accomplish this by using a finite-dimensional
approximation for~$L$, since that can be analyzed via rigorous
computational means. Our finite-dimensional approximation for~$L$
is given as follows. For fixed $N \in \N$ define the finite-dimensional
spaces
\begin{displaymath}
  \Ldomain_N = P_N \Ldomain
  \qquad\mbox{ and }\qquad
  \Lrange_N = P_N \Lrange \; ,
\end{displaymath}
where the projection operator given in~(\ref{eqn:defpn}) 
is applied componentwise on the functional components of $ \Ldomain,\Lrange $,
i.e., on each $ \subspaceSymbol_\subspaceIndex $ individually,
and acts as the identity on the scalar components.
We then define $L_N: \Ldomain_N \to \Lrange_N$ by 
\begin{equation} \label{eqn:defln}
  L_N = \left. P_N L \right|_{\Ldomain_N} \; .
\end{equation}
Let~$K_N$ be a bound on the inverse of the finite-dimensional
operator~$L_N$, i.e., suppose that we have established the
estimate
\begin{equation} \label{eqn:defkn}
  \left\| L_N^{-1} \right\|_{\cL(\Lrange_N,\Ldomain_N)} \le K_N \; ,
\end{equation}
where the spaces~$\Ldomain_N$ and~$\Lrange_N$ are equipped with the norms
of~$\Ldomain$ and~$\Lrange$, respectively. We will discuss further details
on appropriate coordinate systems and the actual computation
of both~$L_N$ and~$K_N$ in Section~\ref{sec:findim}. Nevertheless, after
these preparations we are able to state our main
result for this section.
\begin{theorem}[Inverse estimate for fourth-order operators]
\label{thm:k}
Consider the spaces~$\Ldomain$ and~$\Lrange$ defined in~\eqref{eqn:ldef},
as well as the bounded linear operator~$L \in \cL(\Ldomain,\Lrange)$
acting on~$m \in \N_0$ scalar parameters $\LdomainElementScalar_1, \ldots, \LdomainElementScalar_m$ and
on~$n \in \N$ functions~$v_k \colon \Omega \to \R$ in such a way that
the first~$m$ components of~$L$ are given by the scalars
\begin{equation} \label{thm:k1}
  \sum_{i=1}^m \alpha_{ki} \LdomainElementScalar_i +
    \sum_{j=1}^n l_{kj}(\LdomainElementFun_j)
  \qquad\mbox{ for }\qquad
  k = 1,\ldots,m \; ,
\end{equation}
and the next~$n$ components of~$L$ are given by the functions
\begin{equation} \label{thm:k2}
  -\beta_{k} \Delta^2 \LdomainElementFun_k 
    - \sum_{i=1}^m b_{ki} \LdomainElementScalar_i
    - \Delta \sum_{j=1}^n c_{kj} \LdomainElementFun_j
    - \sum_{j=1}^{n} \gamma_{kj} \LdomainElementFun_{j}
  \qquad\mbox{ for }\qquad
  k = 1,\ldots,n \; .
\end{equation}
In these formulas, the variables~$\alpha_{ki} \in \R$,
$\beta_{k} > 0$, and $ \gamma_{kj} \in \R $ are real constants, 
while~$b_{ki} \in \hzerobar$, and~$c_{kj} \in \htwo$. Moreover,
the~$l_{kj}$ denote bounded linear functionals with Riesz
representative in the spaces~$P_N U_j$, i.e., there exist functions
$a_{kj} \in P_N U_j$ such that one has the identities
$\ell_{kj}(v_j) = (a_{kj},\LdomainElementFun_j)_{\htwobar}$.

Let~$K_N$ be a constant satisfying~\eqref{eqn:defkn}, and define
$ C_T = (\min_{j=1,\ldots,n} \beta_{j})^{-1} > 0$. Furthermore,
define the constants~$A$ and~$B$ by
\begin{align*}
  A \; &:= \; \frac{K_N \sqrt{n}}{\pi^2 N^2} \left( \sum_{k=1}^\nfunctions 
  \max_{1 \le j \le \nfunctions}
      \|c_{kj}\|_{\infty}^2
  \right)^{1/2} \; ,
  \\[1ex]
  B \; &:= \; \frac{C_T \sqrt{2 \max\{m,n\}}}{\pi^2 N^2} \left( \sum_{k=1}^\nfunctions 
  \max_{\substack{1 \le i \le \nscalars \\ 1 \le j \le \nfunctions}} 
      \left\{\|b_{ki}\|_{\hzerobar}, \left( C_b C_e \|c_{kj}\|_{\htwo} + 
      \frac{|\gamma_{kj}|}{\pi^2} \right) \right\}^2
  \right)^{1/2} \; ,
\end{align*}
and assume there exists a constant~$\tau > 0$ and an integer $N \in \N$
such that
\begin{equation} \label{thm:k3}
  \sqrt{A^2+B^2}
  \; \le \; \tau \; < \; 1 \; .
\end{equation}
Then the
operator~$L$ in \eqref{eqn:ldef} satisfies
\begin{equation} \label{thm:k4}
  \left\| L^{-1} \right\|_{\cL(\Lrange,\Ldomain)} \le
  \frac{\max ( K_N, C_T) }{1-\tau} \; . 
\end{equation}
\end{theorem}
At first glance it might seem strange that the constants~$\alpha_{kj}$
and the functions~$a_{kj}$ do not enter either the condition
in~\eqref{thm:k3} or the estimate in~\eqref{thm:k4}. This, however,
is not true, as they determine the constant~$K_N$ from~\eqref{eqn:defkn}.

Before we begin to prove this main theorem, we state a
necessary result which is based on a Neumann series argument
to derive bounds on the operator norm of an inverse of an
operator. This is a standard functional-analytic technique,
which we state here for the reader's convenience. A proof
can be found in~\cite[Lemma~4]{sander:wanner:16a}.
\begin{proposition}[Neumann series inverse estimate]
\label{prop:neumann}
Let $\cA \in \cL(\Ldomain,\Lrange)$ be an arbitrary bounded linear operator
between two Banach spaces, and let $\cS \in \cL(\Lrange,\Ldomain)$ be one-to-one.
Assume that there exist positive constants~$\rho_1$ and~$\rho_2$
such that
\begin{displaymath}
  \| I - \cS \cA \|_{\cL(\Ldomain,\Ldomain)} \le \rho_1 < 1
  \qquad\mbox{ and }\qquad
  \|\cS\|_{\cL(\Lrange,\Ldomain)} \le \rho_2 \;. 
\end{displaymath}
Then $\cA$ is one-to-one and onto, and 
\begin{displaymath}
\| \cA^{-1}\|_{\cL(\Lrange,\Ldomain)} \le \frac{\rho_2}{1-\rho_1} \;.
\end{displaymath}
In subsequent discussions, we will refer to~$\cS$ as an
{\em approximate inverse}.
\end{proposition}
We are now ready to proceed with the proof of the main result
of the section, Theorem~\ref{thm:k}. Our goal is
to prove that~$L$ is one-to-one, onto, and has an inverse whose
operator norm is bounded by the value 
\begin{displaymath}
  K = \frac{\max(K_N,C_T)}{1-\tau} \; .
\end{displaymath}
The complete proof of the above is spread across the remaining
subsections, with the following structure of the key definitions and
auxiliary results:
\begin{itemize}
\item Section \ref{sec:findim}, Lemma \ref{lem:kn} provides a computable
      upper bound for $ \left\| L_N^{-1} \right\| $. 
\item Section \ref{sec:approxinverse}, Definition \ref{def:ST} gives a
      construction of the approximate inverse $ \cS $. 
\item Section \ref{sec:approxinverse}, Lemma \ref{lem:rho2} shows that
      we can take $ \rho_2 = \max(K_N, C_T) $. 
\item Section \ref{sec:comprof}, Lemma \ref{lem:rho1} provides a
      formula for $ \rho_1 $. 
\end{itemize}
Once all of these results have been established, the proof of
Theorem~\ref{thm:k} is complete.
\subsection{Finite-dimensional projections of the linearization}
\label{sec:findim}
\newcommand{\xCoeff}{\xi}%
\newcommand{\yCoeff}{\zeta}%
In this section, we consider~$L_N$, the finite-dimensional projection
of the operator~$L$, which was introduced in~\eqref{eqn:defln}. The
linear map~$L_N$ is tractable using rigorous computational methods,
since calculating a finite-dimensional inverse is something that can
be done using numerical linear algebra. To derive~$L_N$ in more detail,
we recall the definitions of the following projection spaces, all of
which are Hilbert spaces:
\begin{displaymath}
  \begin{array}{rclcrclcrcl}
    \Ldomain & = & \R^\nscalars \times \prod_{\subspaceIndex=1}^{\nfunctions}
      \subspaceSymbol_\subspaceIndex \;, & \quad &
      \Ldomain_N & = & P_N \Ldomain \; , & \quad &
      \Ldomain_\infty & = & (I- P_N) \Ldomain \; , \\[1ex]
    \Lrange & = & \R^\nscalars \times \prod_{\subspaceIndex=1}^{\nfunctions}
      \rangeSubspaceSymbol_\subspaceIndex \;, & \quad &
      \Lrange_N & = & P_N \Lrange \; , & \quad &
      \Lrange_\infty & = & (I- P_N) \Lrange \; ,
  \end{array}
\end{displaymath}
where the projection operator~$P_N$ is applied componentwise on the
functional components of the spaces~$\Ldomain$ and~$\Lrange$,
i.e., on each $ \subspaceSymbol_\subspaceIndex $ individually,
and acts as the identity on the scalar components.
Recall that in~(\ref{eqn:defln}) we defined $L_N: \Ldomain_N \to \Lrange_N$ via
$L_N = \left. P_N L \right|_{\Ldomain_N}$.

In order to work with this finite-dimensional operator in a
straightforward computational manner, we need to find its matrix
representation. If we define $ (\indexsetSymbol_i)_N $ to be the
subset of all multi-indices $ k \in \indexsetSymbol_i $ such that 
$ 0 < |k|_\infty < N $ and $ \indexsetSymbol_N =
\left\{(\indexsetSymbol_1)_N, (\indexsetSymbol_2)_N, \ldots,
(\indexsetSymbol_\nfunctions)_N \right\}$, then both~$\Ldomain_N$
and~$\Lrange_N$ have the basis~$\basisSymbol_{\indexsetSymbol_N}$
and one obtains the matrix representation $B$ via the definition
\begin{equation}\label{eqn:defbkell}
    B = \begin{pmatrix}
        B_{00} & B_{01} & \ldots & B_{0\nfunctions} \\
        B_{10} & B_{11} & \ddots & \vdots \\ 
        \vdots & \ddots & \ddots & B_{(\nfunctions-1)\nfunctions}\\
        B_{\nfunctions 0} & \ldots & B_{\nfunctions(\nfunctions-1)} &
          B_{\nfunctions \nfunctions}
    \end{pmatrix} \; ,
\end{equation}
where the matrices~$ B_{{i,j}} $ are as follows. Denote the element
$ \LdomainElement \in \Ldomain $ in the form $ \LdomainElement =
(\LdomainElementScalar, \LdomainElementFun) $, where
$ \LdomainElementScalar = (\LdomainElementScalar_1,\LdomainElementScalar_2,
\ldots,\LdomainElementScalar_\nscalars) \in \R^\nscalars $ and
$ \LdomainElementFun = (\LdomainElementFun_1,\LdomainElementFun_2,
\ldots,\LdomainElementFun_\nfunctions) \in \prod_{\subspaceIndex=1}^\nfunctions
\subspaceSymbol_\subspaceIndex$. Then the basis elements of
$ \basisSymbol_{\indexsetSymbol_N} $ are given by $ (e_\ell,0) $ for
$ 1 \leq \ell \leq \nscalars $ and~$ (0,\Phi_{ik}) $ for
$k \in (\indexsetSymbol_i)_N$ and $1 \le i \le n$, where~$ \Phi_{ik} $
is defined as the $ \nfunctions $-dimensional vector with~$ \phi_{k} $
in the $ i $-th component and $ 0 $ elsewhere. In addition, we consider
the Hilbert space $ Z = \R^\nscalars \times (L^2(\Omega))^\nfunctions $
and recall that for $ t_1, t_2 \in \R^\nscalars $ and $ w_1, w_2 \in
\prod_{\subspaceIndex=1}^{\nfunctions} \subspaceSymbol_\subspaceIndex $
the inner product on~$Z$ is defined via
\begin{displaymath}
  ((t_1,w_1), (t_2,w_2))_Z =
  (t_1,t_2)_{\R^\nscalars} +
    \sum_{i=1}^\nfunctions ((w_1)_i, (w_2)_i)_{L^2(\Omega)} \; .
\end{displaymath}
Then the above matrices~$ B_{{ij}} $ are defined via the identities
\begin{align*}
  (B_{00})_{k\ell} &= (L[(e_\ell,0)],(e_k,0))_Z \; ,
    & k &= 1,\ldots,\nscalars \; , & \ell& = 1,\ldots,\nscalars \; , \\
  (B_{i0})_{k\ell} &= (L[(e_\ell,0)],(0,\Phi_{ik}))_Z \; ,
    & k & \in (\indexsetSymbol_i)_N \; , & \ell &= 1,\ldots,\nscalars \; , \\
  (B_{0j})_{k\ell} &= (L[(0,\Phi_{j\ell})],(e_k,0))_Z \; ,
    & k &= 1,\ldots,\nscalars \; , & \ell &\in (\indexsetSymbol_j)_N \; , \\
  (B_{ij})_{k\ell} &= (L[(0,\Phi_{j\ell})],(0,\Phi_{ik}))_Z \; ,
    & k &\in (\indexsetSymbol_i)_N \; , & \ell &\in (\indexsetSymbol_j)_N \; ,
\end{align*}
where $i,j = 1,\ldots,n$. To conclude the abstract definition, we
emphasize that
\begin{align*}
  B_{00} &\in \R^{\nscalars \times \nscalars} \; ,
    & B_{0j} &\in \R^{\nscalars \times \#(\indexsetSymbol_j)_N} \; , \\
  B_{i0} &\in \R^{\#(\indexsetSymbol_i)_N \times \nscalars} \; ,
    & B_{ij} &\in \R^{\#(\indexsetSymbol_i)_N \times \#(\indexsetSymbol_j)_N}
    \; ,
\end{align*}
where $ \#S $ denotes the number of elements in the set $ S $.

Now that we have properly defined the involved function spaces and the
procedure to construct the matrix representation~$ B $ of~$L_N$, we can
use~\eqref{thm:k1} and~\eqref{thm:k2} to obtain an explicit representation
of~$ L $ acting on $ \LdomainElement = (\LdomainElementScalar, \LdomainElementFun)
\in \Ldomain $. As mentioned in the formulation of Theorem~\ref{thm:k}, using the
Riesz Representation theorem one can write each functional~$ \ell_{kj}(v_j) $
in~\eqref{thm:k1} as the inner product~$ (a_{kj}, v_j)_{\htwobar} $. 
This substitution yields the explicit form
\begin{equation}
  \label{eqn:explicit-form-L}
  L \LdomainElement = 
  \begin{pmatrix}
    \left[ \sum_{i=1}^m \alpha_{ki} \LdomainElementScalar_i +
    \sum_{j=1}^n (a_{kj},\LdomainElementFun_j)_{\htwobar} \right]_{k=1}^\nscalars \\
    \left[ -\beta_{k} \Delta^2 \LdomainElementFun_k 
      - \sum_{i=1}^m b_{ki} \LdomainElementScalar_i 
      - \Delta \sum_{j=1}^n c_{kj} \LdomainElementFun_j 
      - \sum_{j=1}^{\nfunctions} \gamma_{kj} \LdomainElementFun_j
    \right]_{k=1}^\nfunctions
  \end{pmatrix} \; ,
\end{equation}
which in turn leads to the following explicit forms for the components
of~$ B $:
\begin{subequations}
    \label{eqn:defbijkell}
    \begin{align}
        (B_{00})_{k\ell} &= \alpha_{k\ell} \; , \\[1ex]
        (B_{i0})_{k\ell} &= -(b_{i\ell}, \phi_{k})_{L^2(\Omega)} \; , \\[1ex]
        (B_{0j})_{k\ell} &= (a_{kj},\phi_{\ell})_{\suitableFunctionSpace} \; , \\[1ex]
        (B_{ij})_{k\ell} &= (-\delta_{ij}\beta_{i}\Delta^2 \phi_{\ell} -
          \Delta c_{ij}\phi_{\ell} - \gamma_{ij}\phi_{\ell}, \phi_{k})_{L^2(\Omega)}
          \nonumber \\[0.5ex]
        &= (-\delta_{ij}\beta_{i}\kappa_{\ell}^2 \phi_{\ell} + \kappa_{k}
          c_{ij}\phi_{\ell} - \gamma_{ij}\phi_{\ell}, \phi_{k})_{L^2(\Omega)}
          \nonumber \\[0.5ex]
        &= -\delta_{ij}\delta_{k\ell} \beta_{i}\kappa_{\ell}^2 -
          \gamma_{ij}\delta_{k\ell} +
          (\kappa_{k} c_{ij}\phi_{\ell}, \phi_{k})_{L^2(\Omega)}
          \label{eqn:functional-b-k-ell} \; ,
    \end{align}
\end{subequations}
where we use $(-\Delta c_{ij}\phi_{\ell},\phi_{k})_{L^2(\Omega)} =
(c_{ij}\phi_{\ell}, -\Delta \phi_{k})_{L^2(\Omega)} =
(c_{ij}\phi_{\ell}, \kappa_k \phi_{k})_{L^2(\Omega)}$,
as well as~\eqref{def:kappak}.

The matrix representation $ B $ characterizes~$L_N$ on the algebraic
level in the following sense. If we consider an element
$ \LdomainElement_N \in \Ldomain_N $, then one can introduce the
representations
\begin{displaymath}
  \LdomainElement_N  = \sum_{b \in \basisSymbol_{\indexsetSymbol_N}}
    \xCoeff_b b
  \qquad\mbox{ and }\qquad
  L_N \LdomainElement_N  = \sum_{b \in \basisSymbol_{\indexsetSymbol_N}}
    \yCoeff_b b \; ,
\end{displaymath}
where the coefficients satisfy both $\xCoeff_b \in \R$ and
$\yCoeff_b \in \R$, and the basis elements~$b$ are taken from the
set
\begin{displaymath}
  \basisSymbol_{\indexsetSymbol_N} \; = \;
  \left\{ e_j \times 0_{\left(\htwobar\right)^\nfunctions}
    \right\}_{j=1 \ldots \nscalars} \bigcup \;
  \left\{ 0_{\R^\nscalars \times \left(\htwobar\right)^{i-1}}
    \times \phi_{\indexIndex} \times 0_{\left(\htwobar\right)^{\nfunctions-i}}
    \right\}_{\subspaceIndex=1 \ldots \nfunctions,\,\indexIndex
    \in (\indexsetSymbol_\subspaceIndex)_N} \; .
\end{displaymath}
If we collect the numbers~$\xCoeff_b$ and~$\yCoeff_b$ in
vectors~$\xCoeff$ and~$\yCoeff$ in the straightforward way, then
one immediately obtains the matrix-vector identity
\begin{displaymath}
  \yCoeff = B \xCoeff \; .
\end{displaymath}
This natural algebraic representation has one slight drawback that still needs
to be addressed. We would like to use the regular Euclidean norm on real vector
spaces, as well as the induced matrix norm, to study the
$\cL(\Ldomain_N,\Lrange_N)$-norm of~$L_N$. For our computer-assisted proof,
we are therefore interested in a scaled version of~$ B $ which gives a
computable~$ K_N $, and this scaled matrix is the subject of the following Lemma.
\begin{lemma}[Computable $ K_N $]
  \label{lem:kn}
  Let $ B $ be defined as in \eqref{eqn:defbkell},
  $ D_i = \text{diag }(\{ \kappa_k : k \in (\indexsetSymbol_i)_N \}) $,
  and let~$ I_m $ be the $ m \times m $ identity matrix.
  Assemble $ D $ as the block diagonal matrix
  \begin{displaymath}
    D = \left( \begin{array}{ccccc}
        I_{\nscalars} & 0 & 0 & \ldots & 0 \\
        0 & D_1 & 0 & \ldots & 0 \\
        0 & 0 & D_2 & \ddots & \vdots \\
        \vdots & \vdots & \ddots & \ddots & 0 \\
        0 & 0 & \ldots & 0 & D_\nfunctions
        \end{array} \right) \; ,
  \end{displaymath}
  and define $ \tilde{B} = D^{-1}BD^{-1} $.
  Then $ K_N $ in \eqref{eqn:defkn} can be taken as $\| \tilde{B}^{-1}\|_2 $.
  In other words, using this formula, we can use interval arithmetic to
  establish a rigorous upper bound on the norm of this finite-dimensional
  inverse.
\end{lemma}
\begin{proof}
  To begin with, we recall Lemma~\ref{lem:orthoghell} which shows that,
  for each $ i = 1, \ldots, \nfunctions $,
  the collection~$\{ \kappa_k^{-1} \phi_k(x) \}$ with~$k \in (\indexsetSymbol_i)_N$ as above is
  an orthonormal basis in $P_N U_i \subset \htwobar$, and~$\{ \kappa_k \phi_k(x) \}$
  is an orthonormal basis in $\Lrange_N \subset \Lrange$, where the
  eigenvalues~$ \kappa_k $ are defined in~\eqref{def:kappak}.
  Thus, we need to use the modified representations
  \begin{displaymath}
      \LdomainElement_N  = \sum_{\tilde{b} \in \tilde{\basisSymbol}_{\indexsetSymbol_N}}
          \tilde{\xCoeff}_{\tilde{b}} \tilde{b}
      \qquad\mbox{ and }\qquad
      L_N \LdomainElement_N  = \sum_{\hat{b} \in \hat{\basisSymbol}_{\indexsetSymbol_N}}
          \tilde{\yCoeff}_{\hat{b}} \hat{b}
  \end{displaymath}
  where we use the alternative basis sets (note the $ \kappa_k^{\pm 1} $ factors)
  \begin{align*}
    \tilde{\basisSymbol}_{\indexsetSymbol_N} &=
      \left\{ e_j \times 0_{\left(\htwobar\right)^\nfunctions}
      \right\}_{j=1 \ldots \nscalars} \bigcup \; \left\{ 0_{\R^\nscalars
      \times \left(\htwobar\right)^{i-1}} \times
      \kappa_{\indexIndex}^{-1}\phi_{\indexIndex} \times
      0_{\left(\htwobar\right)^{\nfunctions-i}}
      \right\}_{\subspaceIndex=1 \ldots \nfunctions,\,\indexIndex
      \in (\indexsetSymbol_i)_N} \; , \\
    \hat{\basisSymbol}_{\indexsetSymbol_N} &=
      \left\{ e_j \times 0_{\left(\htwobar\right)^\nfunctions}
      \right\}_{j=1 \ldots \nscalars} \bigcup \; \left\{ 0_{\R^\nscalars
      \times \left(\htwobar\right)^{i-1}} \times
      \kappa_{\indexIndex}\phi_{\indexIndex} \times
      0_{\left(\htwobar\right)^{\nfunctions-i}}
      \right\}_{\subspaceIndex=1 \ldots \nfunctions,\,\indexIndex
      \in (\indexsetSymbol_i)_N} \; .
  \end{align*}
  In order to pass back and forth 
  between these two representations we use the block diagonal matrix $ D $.
  One can see that on the level of vectors we have
  \begin{displaymath}
    \xCoeff = D^{-1} \tilde{\xCoeff}
    \quad\mbox{ and }\quad
    \yCoeff = D \tilde{\yCoeff} \; ,
    \quad\mbox{ and therefore }\quad
    \tilde{\yCoeff} = D^{-1} B D^{-1} \tilde{\xCoeff} \; .
  \end{displaymath}
  In view of Lemma~\ref{lem:orthoghell} one then obtains
  \begin{displaymath}
    \left\| L_N \right\|_{\cL(\Ldomain_N,\Lrange_N)} =
    \| \tilde{B} \|_2
    \qquad\mbox{ with }\qquad
    \tilde{B} = D^{-1} B D^{-1} \; ,
  \end{displaymath}
  where~$\| \cdot \|_2$ denotes the regular induced $2$-norm of a matrix.
  Moreover, one can verify that we also have the identity
  \begin{equation} \label{eqn:lninversenorm}
    \left\| L_N^{-1} \right\|_{\cL(\Lrange_N,\Ldomain_N)} =
    \left\| \tilde{B}^{-1} \right\|_{2}
  \end{equation}
  which completes the proof.
\end{proof}
\begin{remark}
  Since $ D $ is a diagonal matrix we can construct~$ \tilde{B} $
  directly via the formulas
  \begin{subequations}
    \label{eqn:defbtildeijkell}
    \begin{align}
        (\tilde{B}_{00})_{k\ell} &= \alpha_{k\ell} \; , \\
        (\tilde{B}_{i0})_{k\ell} &= -\frac{(b_{i\ell}, 
          \phi_{k})_{L^2(\Omega)}}{\kappa_k} \; , \\
        (\tilde{B}_{0j})_{k\ell} &= \frac{(a_{kj},
          \phi_{\ell})_{\suitableFunctionSpace}}{\kappa_\ell}
          \; , \\
        (\tilde{B}_{ij})_{k\ell} &= -\delta_{ij}\delta_{k\ell} \beta_{i} -
          \frac{1}{\kappa_k^2} \delta_{k\ell} \gamma_{ij} +
          \frac{1}{\kappa_\ell} \left(c_{ij}\phi_{\ell},
          \phi_{k}\right)_{L^2(\Omega)}
          \label{eqn:functional-b-tilde-k-ell} \; .
    \end{align}
  \end{subequations}
\end{remark}

\subsection{Construction of an approximate inverse}
\label{sec:approxinverse}
The crucial part in the derivation of our norm bound for the
inverse of~$L$ is the application of Proposition~\ref{prop:neumann}.
For this, we need to construct an approximate inverse of this
operator. Since this construction must be explicit, we will 
approach it in two steps. The first has already been accomplished
in the last section, where we considered a finite-dimensional
projection of~$L$, which can be inverted numerically. 
In this section, we complement this finite-dimensional part with
a consideration of the infinite-dimensional complementary space.
For this, we refer the reader again to the definition of the
matrix representation~$B$ in \eqref{eqn:defbkell} and \eqref{eqn:defbijkell}. 
Since the finite-dimensional approximation is constructed using
the projections~$P_N$ which make use of the low-wavenumber basis
functions, one would expect that as $N \to \infty$ this representation
leads to increasingly better approximations of the operator~$L$. 
Note in particular that every entry~$\left(B_{ij}\right)_{k\ell}$
in~\eqref{eqn:functional-b-k-ell} is the sum of three terms, where
the first one and the last one depend on the Laplacian eigenvalues
from~\eqref{def:kappak}. One can easily see that among these three
terms the first one dominates as $ \ell \to \infty $, and thus 
also as $ N \to \infty $. Based on this observation, we now describe how 
to use the
inverse of the first term on the infinite tail in order to complement
the inverse of~$L_N$. 

To describe this procedure in more detail, consider an arbitrary 
element $ \LrangeElement \in \Lrange $. We decompose this element
into its finite-dimensional part and infinite tail in the form
\begin{displaymath}
  \LrangeElement = \sum_{b \in \basisSymbol_{I}}
    \yCoeff_b b
  = \LrangeElement_N + \LrangeElement_{\infty} \in
  \Lrange_N \oplus \Lrange_\infty \; ,
\end{displaymath}
where we define
\begin{displaymath}
  \Lrange_N = P_N \Lrange
  \qquad\mbox{ and }\qquad
  \Lrange_\infty = \left( I - P_N \right) \Lrange \; .
\end{displaymath}
Using this representation we also have
\begin{displaymath}
  \LrangeElement_\infty \; = \;
  \sum_{j=1}^\nfunctions \sum_{k \in \indexsetSymbol_j
    \setminus (\indexsetSymbol_j)_N} \yCoeff_{jk} \Phi_{jk}
\end{displaymath}
which enables the following definition.
\begin{definition}[Approximate Inverse Operator]
  \label{def:ST}
  Let $ \LrangeElement_N, \LrangeElement_\infty $ be as above.
  We define the operator $ T : \Lrange_\infty \to \Ldomain_\infty $ as
  \begin{displaymath}
    T \LrangeElement_\infty \; = \;
    T \sum_{j=1}^\nfunctions \sum_{k \in \indexsetSymbol_j
      \setminus (\indexsetSymbol_j)_N} \yCoeff_{jk} \Phi_{jk} \; = \;
    -\sum_{j=1}^\nfunctions \sum_{k \in \indexsetSymbol_j \setminus
      (\indexsetSymbol_j)_N} \frac{\yCoeff_{jk}}{\beta_{j}
      \kappa_{k}^2} \Phi_{jk} \; ,
  \end{displaymath}
  and the operator $ S : \Lrange \to \Ldomain $ as
  \begin{equation}
    \label{def:approx-inv-S}
    S \LrangeElement \; = \;
    L_N^{-1} \LrangeElement_N + T \LrangeElement_\infty \; .
  \end{equation}
  One can readily see that the operator $T = S|_{\Lrange_\infty}$
  is one-to-one and onto, and the operator~$ S $ is the candidate
  approximate inverse of $ L \in \cL(\Ldomain, \Lrange) $.
\end{definition}
To close this section, we now derive a bound on the 
operator norm of~$S$, since this will be needed in the
application of Proposition~\ref{prop:neumann}.
\begin{lemma}[Computable~$\rho_2$] \label{lem:rho2}
  Consider the two operators $ S, T $ as defined in
  Definition~\ref{def:ST}, assume that $\beta_{j} > 0$
  for $j=1,\ldots,n$, and define the constant $C_T =
  (\min_{j=1,\ldots,n} \beta_{j})^{-1} > 0$ as in
  Theorem~\ref{thm:k}. Then we have the two inequalities
  \begin{displaymath}
    \| T \|_{\cL(\Lrange_\infty, \Ldomain_\infty)} \le C_T
    \qquad\mbox{ and }\qquad
    \|S\|_{\cL(\Lrange,\Ldomain)} \le \max(K_N,C_T) \; ,
  \end{displaymath}
  where~$K_N$ was introduced in~\eqref{eqn:defkn}. Moreover,
  this implies that we may take $\rho_2 = \max(K_N, C_T)$
  in Proposition \ref{prop:neumann}.
\end{lemma}
\begin{proof}
  To begin with, we let $\LrangeElement_\infty \in \Lrange_\infty$
  be arbitrary and show that $\| T \LrangeElement_\infty \|_\Ldomain
  \le C_T \| \LrangeElement_\infty \|_\Lrange$.
  This follows readily from $\beta_{j} > 0$, Lemma~\ref{lem:orthoghell}, and~\ref{def:hbarellnorm},
  as well as the identities
  \begin{eqnarray*}
    \left\| T \LrangeElement_\infty \right\|_{\Ldomain}^2 & = &
      \left\| \sum_{j=1}^\nfunctions \sum_{k \in \indexsetSymbol_j
      \setminus (\indexsetSymbol_j)_N} \frac{\yCoeff_{jk}}{\beta_{j}
      \kappa_{k}^2} \Phi_{jk} \right\|_{\overline{\cH}^2}^2 \\[2ex]
    & = &
      \sum_{j=1}^\nfunctions \sum_{k \in \indexsetSymbol_j
      \setminus (\indexsetSymbol_j)_N}
      \frac{\yCoeff_{jk}^2\kappa_{k}^2}{(\beta_{j}\kappa_{k}^2)^2}
    \; \le \;
      C_T^2\sum_{j=1}^\nfunctions \sum_{k \in \indexsetSymbol_j
      \setminus (\indexsetSymbol_j)_N}
      \kappa_{k}^{-2} \yCoeff_{jk}^2 \\[2ex]
    & = &
      C_T^2\left\|
      \sum_{j=1}^\nfunctions \sum_{k \in \indexsetSymbol_j
      \setminus (\indexsetSymbol_j)_N}
      \yCoeff_{jk} \Phi_{jk} \right\|_{\overline{\cH}^{-2}}^2
    \; = \;
      C_T^2\left\| \LrangeElement_\infty \right\|_\Lrange^2
      \; .
  \end{eqnarray*}
  This estimate in turn implies for all $ \LrangeElement =
  \LrangeElement_N + \LrangeElement_\infty \in \Lrange_N \oplus
  \Lrange_\infty$ the inequality
  \begin{eqnarray*}
    \| S \LrangeElement \|_\Ldomain^2 & = &
      \| L_N^{-1} \LrangeElement_N \|_\Ldomain^2 +
      \| T \LrangeElement_\infty \|_\Ldomain^2 \\[2ex]
    & \le &
      \underbrace{\| L_N^{-1} \|_{\cL(\Lrange_N,\Ldomain_N)}^2}_{\le K_N^2}
      \| \LrangeElement_N \|_\Lrange^2 + C_T^2\| \LrangeElement_\infty \|_\Lrange^2
      \; \le \; \max(K_N,C_T)^2 \| \LrangeElement \|_\Lrange^2 \; ,
  \end{eqnarray*}
  where we used the definition of~$K_N$ from~(\ref{eqn:defkn}).
  This completes the proof of the lemma.
\end{proof}

In other words, the operator norm of the approximate inverse~$S$
can be bounded in terms of the inverse bound for the finite-dimensional
projection given in Lemma \ref{lem:kn}. Furthermore, it follows directly
from the definition of~$S$ that this operator is one-to-one, as long as~$L_N$
is --- and the latter can be established using interval arithmetic. We
conclude by remarking that, in many cases, the constant~$C_T$ can be taken
as~$1$ through proper scaling of the diffusion coefficients in the model
formulation.
\subsection{Assembling the final inverse estimate}
\label{sec:comprof}
In the last section we addressed two crucial aspects of
Proposition~\ref{prop:neumann}. On the one hand, we provided an
explicit construction for the approximate inverse~$S \in \cL(\Lrange,\Ldomain)$
of~$L$ defined in~(\ref{eqn:ldef}). On the
other hand, we derived an upper bound on the operator norm of~$S$,
which can be computed using the finite-dimensional projection~$L_N$
of~$L$. This in turn provides the constant~$\rho_2$ in
Proposition~\ref{prop:neumann}. In this final subsection, we focus
on the constant~$\rho_1$, i.e., we derive an upper bound on the
norm~$\|I - S L \|_{\cL(\Ldomain,\Ldomain)}$, and show how this bound can be made
smaller than one. Altogether, this will complete the proof of the
estimate for the constant~$K$ which bounds the operator norm
of~$ L^{-1} $. As a first step, we present in the following
lemma a decomposition of~$ L $ in terms of~$ L_N $ and~$ T $
that will help handle the infinite tail estimates.
\begin{lemma}
  \label{lem:LTN-decomp}
  Let $ L \in \cL(\Ldomain, \Lrange) $ be as in~\eqref{eqn:explicit-form-L},
  and let $ S \in \cL(\Lrange,\Ldomain) $,~$ T \in \cL(\Lrange_\infty,
  \Ldomain_\infty) $ be as in Definition~\ref{def:ST}.
  Further, let $ P_N $ be as defined in~\eqref{eqn:defpn}
  and~$ L_N \in \cL(\Ldomain_N, \Lrange_N) $ be as in~\eqref{eqn:defln}.
  Then, using the additive representation
  $ \LdomainElement 
    = \LdomainElement_N + \LdomainElement_\infty 
    = (\LdomainElementScalar, P_N \LdomainElementFun) 
      + (0, (I-P_N) \LdomainElementFun)
      \in \Ldomain_N \oplus \Ldomain_\infty$,
  we have the identity
  \begin{equation}\label{eqn:lx}
    L \LdomainElement = (L_N \LdomainElement_N + \cM \LdomainElement_\infty) +
        \left( T^{-1} \LdomainElement_\infty - \cN \LdomainElement \right) \; ,
  \end{equation}
  where we define $ \cM, \cN $ by
  \begin{align}
    \label{def:calM}
    \cM \LdomainElement_\infty \; &= \; \left( \left[ \sum_{j=1}^\nfunctions
      (a_{kj}, (I-P_N)\LdomainElementFun_j)_{\htwobar} \right]_{k=1}^\nscalars,
      \left[ -P_N \Delta \sum_{j=1}^\nfunctions c_{kj} (I-P_N) v_j
      \right]_{k=1}^\nfunctions \right) \; , \\[2ex]
    \label{def:calN}
    \cN \LdomainElement \; &= \; \left( 0, \left[ (I-P_N)
      \left(\sum_{i=1}^\nscalars \LdomainElementScalar_i b_{ki} +
      \Delta\sum_{j=1}^\nfunctions c_{kj} \LdomainElementFun_j +
      \sum_{j=1}^{\nfunctions} \gamma_{kj} \LdomainElementFun_j \right)
      \right]_{k=1}^\nfunctions \right) \; .
  \end{align}
\end{lemma}
\begin{proof}
  Notice that $ L_N \LdomainElement_N $ and $ \cM \LdomainElement_\infty $
  are in the finite-dimensional space~$\Lrange_N$, while
  $ T^{-1}\LdomainElement_\infty $ and $ \cN \LdomainElement $ are in~$\Lrange_\infty$.
  With this in mind, we detail the derivation of \eqref{eqn:lx} as follows.
  We first note that $ \Delta $ and $ P_N $ commute.
  Then the explicit form of $ L_N \LdomainElement_N = P_N L \LdomainElement_N $ is:
  \begin{displaymath}
    L_N \LdomainElement_N = \begin{pmatrix}
            \left[ \sum_{i=1}^\nscalars \alpha_{ki}\LdomainElementScalar_i +
            \sum_{j=1}^\nfunctions (a_{kj}, P_N
            \LdomainElementFun_j)_{\suitableFunctionSpace} \right]_{k=1}^\nscalars \\[2ex]
            \left[ - \beta_{k}\Delta^2P_N \LdomainElementFun_k 
            - \sum_{i=1}^\nscalars \LdomainElementScalar_i P_N b_{ki} 
            - P_N \Delta\sum_{j=1}^\nfunctions c_{kj} P_N \LdomainElementFun_j 
            - \sum_{j=1}^{\nfunctions} \gamma_{kj} P_N \LdomainElementFun_j
            \right]_{k=1}^\nfunctions
          \end{pmatrix} \; .
  \end{displaymath}
  Next, we consider the difference $ L \LdomainElement - L_N \LdomainElement_N $.
  The first $ \nscalars $ components are given by the scalars
  \begin{displaymath}
    \sum_{j=1}^\nfunctions (a_{kj}, (I-P_N)\LdomainElementFun_j)_{\htwobar}
    \qquad\mbox{ for }\qquad
    k = 1,\ldots,m \; ,
  \end{displaymath}
  and we compute the next $ \nfunctions $ components term by term for clarity.
  The terms involving $ b_{ki} $ and $ \gamma_{kj} $ are straightforward since $ P_N $ is linear.
  The term involving $ \Delta^2 $ is also straightforward since $ \Delta^\ell $
  commutes with $ P_N $. This leaves the term involving $ c_{kj} $, which we can
  decompose using $ P_N $ as follows:
  \begin{align*}
    -\Delta \sum_{j=1}^\nfunctions c_{kj} v_j 
    =& 
    -P_N \Delta \sum_{j=1}^\nfunctions c_{kj} P_N v_j
    -P_N \Delta \sum_{j=1}^\nfunctions c_{kj} (I-P_N) v_j
    \\& -(I-P_N) \Delta \sum_{j=1}^\nfunctions c_{kj} P_N v_j
    -(I-P_N) \Delta \sum_{j=1}^\nfunctions c_{kj} (I-P_N) v_j
    \; .
  \end{align*}
  This immediately implies that 
  \begin{displaymath}
    -\Delta \sum_{j=1}^\nfunctions c_{kj} v_j 
    +P_N \Delta \sum_{j=1}^\nfunctions c_{kj} P_N v_j
    =
    -P_N \Delta \sum_{j=1}^\nfunctions c_{kj} (I-P_N) v_j
    -(I-P_N) \Delta \sum_{j=1}^\nfunctions c_{kj} v_j \; ,
  \end{displaymath}
  and therefore that last $ \nfunctions $ components of the difference $ L \LdomainElement - L_N \LdomainElement_N $
  are given explicitly by
  \begin{eqnarray*}
    & & \underbrace{-P_N \Delta \sum_{j=1}^\nfunctions c_{kj}
      (I-P_N)\LdomainElementFun_j}_{(\cM \LdomainElement_\infty)_k}
      \overbrace{-\beta_k \Delta^2(I-P_N)\LdomainElementFun_k}^{(T^{-1}
      \LdomainElement_\infty)_k} \\[1ex]
    & & \qquad\quad
      \underbrace{-\sum_{i=1}^\nscalars \LdomainElementScalar_i (I-P_N) b_{ki}
    -(I-P_N)\Delta \sum_{j=1}^\nfunctions c_{kj} \LdomainElementFun_j
    -\sum_{j=1}^\nfunctions \gamma_{kj} (I-P_N) \LdomainElementFun_j}
    _{-(\cN \LdomainElement)_k}
  \end{eqnarray*}
  for $ k = 1,\ldots,\nfunctions $.
  Thus we have shown $ L \LdomainElement - L_N \LdomainElement_N =
  \cM \LdomainElement_\infty + T^{-1} \LdomainElement_\infty - \cN \LdomainElement $
  and completed the proof.
\end{proof}
We now use the above representation \eqref{eqn:lx} of the operator~$L$ which
is split along the subspaces~$\Lrange_N$ and~$\Lrange_\infty$ to derive an
expression for the infinite tail~$I - SL \in \cL(\Ldomain,\Ldomain)$.
More precisely, we have
\begin{equation}\label{eqn:imsl}
  (I - SL) \LdomainElement = T \cN \LdomainElement - L_N^{-1}
  \cM \LdomainElement_\infty \; ,
\end{equation}
and this will be verified in detail below. Notice that in this
representation, the first term lies in the
the complement~$\Ldomain_\infty$, while the second term is contained
in the finite-dimensional space~$\Ldomain_N$.
The identity in~(\ref{eqn:imsl}) follows from the definition
of $ S $ in Definition \ref{def:ST} and
\begin{eqnarray*}
  SL \LdomainElement & = & L_N^{-1} L_N \LdomainElement_N
    + L_N^{-1} \cM \LdomainElement_\infty
    + T \left( T^{-1} \LdomainElement_\infty 
    - \cN \LdomainElement \right) \\[1ex]
  & = & \LdomainElement_N
    + \LdomainElement_\infty 
    + L_N^{-1} \cM \LdomainElement_\infty
    - T \cN \LdomainElement \\[1ex]
  & = & \LdomainElement
    + L_N^{-1} \cM \LdomainElement_\infty
    - T \cN \LdomainElement \; .
\end{eqnarray*}
After these preparations, we can now show that the operator norm
of~$I - SL$ can be expected to be small for sufficiently large $N$.
This will provide an estimate for the constant~$\rho_1$ in
Proposition~\ref{prop:neumann}, and conclude the proof of
Theorem~\ref{thm:k}. However, we pause here to remind the
reader that, in principle, the Riesz representative~$ a_{kj} $ could
be a general element of~$ \htwobar $. As mentioned in the introduction,
we restrict ourselves to the case where~$ a_{kj} $ is an element of
the finite-dimensional space~$ P_N U_j \subset U_j $, which implies
that the first~$ m $ components of $ \cM \LdomainElement_\infty $
are in fact identically 0.
\begin{lemma}[Computable $ \rho_1 $]
  \label{lem:rho1}
  Let $ T $ be as in Lemma~\ref{def:ST}, $ \cN $ be as in~\eqref{def:calN},
  and $ b_{ki}, c_{kj}, \gamma_{kj} $ be as in Theorem~\ref{thm:k}.
  Suppose further that, just as in Theorem~\ref{thm:k}, the Riesz
  representative~$ a_{kj} $ of~$ \ell_{kj} $
  lies in~$ P_N U_j $.
  Define $ C_T = (\min_{j=1,\ldots,n} \beta_{j})^{-1} > 0$
  as in Theorem~\ref{thm:k}, $ K_N $ as in Lemma~\ref{lem:kn},
  and finally~$A$ and~$B$ by
  \begin{align*}
    A \; &:= \; \frac{K_N \sqrt{n}}{\pi^2 N^2} \left( \sum_{k=1}^\nfunctions 
    \max_{1 \le j \le \nfunctions}
        \|c_{kj}\|_{\infty}^2
    \right)^{1/2} \; ,
    \\
    B \; &:= \; \frac{C_T \sqrt{2 \max\{m,n\}}}{\pi^2 N^2} \left( \sum_{k=1}^\nfunctions 
    \max_{\substack{1 \le i \le \nscalars \\ 1 \le j \le \nfunctions}} 
        \left\{\|b_{ki}\|_{\hzerobar}, \left( C_b C_e \|c_{kj}\|_{\htwo} +
        \frac{|\gamma_{kj}|}{\pi^2} \right) \right\}^2
    \right)^{1/2}
  .\end{align*}
  Then,
  $ \left\| L_N^{-1}\cM \LdomainElement_\infty \right\|_\Ldomain
  \le A \|\LdomainElement_\infty\|_\Ldomain$,
  $ \left\| T \cN \LdomainElement_\infty \right\|_\Ldomain
  \le B \|\LdomainElement\|_\Ldomain $,
  and $ \| I  - S L \|_{\cL(\Ldomain,\Ldomain)} \le \sqrt{A^2+B^2} $.
  Furthermore, as long as there exists a~$ \tau $ such that
  $ \sqrt{A^2 + B^2} \le \tau < 1 $, we can take $ \rho_1 = \tau $
  in Proposition \ref{prop:neumann}.
\end{lemma}
\begin{proof}
  For brevity in the verification, we define
  \begin{align*}
    M_k \; &:= \; \max_{1 \le j \le \nfunctions} \|c_{kj}\|_{\infty}^2 \; , \\
    N_k \; &:= \; 
      \max_{\substack{1 \le i \le \nscalars \\ 1 \le j \le \nfunctions}} 
      \left\{\|b_{ki}\|_{\hzerobar}, \left( C_b C_e \|c_{kj}\|_{\htwo} +
      \frac{|\gamma_{kj}|}{\pi^2} \right) \right\}^2 \; ,
  \end{align*}
  and verify the estimates as follows.
  First, since $\|L_N^{-1}\|_{\cL(\Lrange_N, \Ldomain_N)} \le K_N$, we must find a bound of the form
  \begin{displaymath}
    \|\cM \LdomainElement_\infty\|_{\Lrange_N} \le C_\cM \|\LdomainElement_\infty\|_{\Ldomain}
    \; .
  \end{displaymath}
  By the definition of $\cM$, and since $ a_{kj} \in P_N U_j $, we have
  \begin{displaymath}
    \|\cM\LdomainElement_\infty\|_{\Lrange_N}^2 = \sum_{k=1}^\nfunctions \left\| -P_N\Delta\sum_{j=1}^\nfunctions c_{kj} (I-P_N) \LdomainElementFun_j \right\|_{\hmtwobar}^2
  .\end{displaymath}
  Since $ P_N $ is an orthogonal projection and $ \Delta $ is an isometry,
  see Lemma~\ref{lem:sobolevscale}(a), we have the upper bound
  \begin{displaymath}
    \|\cM\LdomainElement_\infty\|_{\Lrange_N}^2 \le 
    \sum_{k=1}^\nfunctions \left\| \sum_{j=1}^\nfunctions c_{kj}
    (I-P_N) \LdomainElementFun_j \right\|_{\hzerobar}^2 \le
    \sum_{k=1}^\nfunctions \left( \sum_{j=1}^\nfunctions \left\| c_{kj}
    (I-P_N) \LdomainElementFun_j \right\|_{\hzerobar} \right)^2
  .\end{displaymath}
  Lemma~\ref{lem:projtailest} then yields
  \begin{displaymath}
    \|\cM\LdomainElement_\infty\|_{\Lrange_N}^2 \le
    \sum_{k=1}^\nfunctions 
      \left( \sum_{j=1}^\nfunctions \frac{1}{\pi^2 N^2}
      \left\|c_{kj}\right\|_{\infty}\left\|(I-P_N)\LdomainElementFun_j
      \right\|_{\htwobar}\right)^2 \; ,
  \end{displaymath}
  and factoring out the maximum coefficients gives
  \begin{displaymath}
    \|\cM\LdomainElement_\infty\|_{\Lrange_N}^2 \le
    \left(
      \frac{1}
      {\pi^4 N^4}
      \sum_{k=1}^\nfunctions 
      \max_{1 \le j \le \nfunctions} \left\|c_{kj}\right\|_{\infty}^2
    \right)
    \left(
      \sum_{j=1}^\nfunctions \left\|(I-P_N)\LdomainElementFun_j\right\|_{\htwobar}
    \right)^2
  .\end{displaymath}
  Finally, the Cauchy-Schwarz inequality yields
  \begin{displaymath}
    \|\cM\LdomainElement_\infty\|_{\Lrange_N}^2 \le
    \left(
      \frac{n}
      {\pi^4 N^4}
      \sum_{k=1}^\nfunctions 
      \max_{1 \le j \le \nfunctions} \left\|c_{kj}\right\|_{\infty}^2
    \right)
    \sum_{j=1}^\nfunctions \left\|(I-P_N)\LdomainElementFun_j\right\|_{\htwobar}^2
    \; ,
  \end{displaymath}
  which is precisely
  \begin{displaymath}
    \|\cM\LdomainElement_\infty\|_{\Lrange_N}^2 \le
    \frac{n}
      {\pi^4 N^4}
      \left(\sum_{k=1}^\nfunctions M_k\right)
      \left\|\LdomainElement_\infty\right\|_{\Ldomain}^2 = C_\cM^2 \left\|\LdomainElement_\infty\right\|_{\Ldomain}^2
  .\end{displaymath}
  Therefore, we can take $ A := K_N C_\cM $.
  Second, since $\| T \|_{\cL(\Lrange_\infty,\Ldomain_\infty)} \le C_T$, we must find a bound of the form
  \begin{displaymath}
    \|\cN \LdomainElement\|_\Lrange \le C_\cN \|\LdomainElement\|_\Ldomain.
  \end{displaymath}
  By the definition of $ \cN $, we have
  \begin{displaymath}
    \left\| \cN \LdomainElement \right\|_\Lrange^2 
      \; = \; 
      \sum_{k=1}^\nfunctions 
        \left\| 
          (I-P_N) \left( \sum_{i=1}^\nscalars \LdomainElementScalar_i b_{ki} 
        + \Delta \sum_{j=1}^\nfunctions c_{kj} \LdomainElementFun_j 
        + \sum_{j=1}^{\nfunctions} \gamma_{kj}  \LdomainElementFun_j 
        \right)
        \right\|_{\hmtwobar}^2
  .\end{displaymath}
  We can then split the $\hmtwobar$-norm term with the triangle inequality and use
  Lemma~\ref{lem:sobolev}(c), Lemma~\ref{lem:sobolevscale}(a), and
  Lemma~\ref{lem:projtailest} to obtain the upper bound
  \begin{displaymath}
    \left\| \cN \LdomainElement \right\|_\Lrange^2 \le 
    \sum_{k=1}^\nfunctions 
    \left( 
      \sum_{i=1}^\nscalars \frac{|\LdomainElementScalar_i| \left\| b_{ki} \right\|_{\hzerobar}}{\pi^2 N^2}
      + \sum_{j=1}^\nfunctions \frac{\| c_{kj} \LdomainElementFun_j \|_{\htwo}}{\pi^2 N^2}
      + \sum_{j=1}^\nfunctions \frac{|\gamma_{kj}| \left\| \LdomainElementFun_j \right\|_{\htwobar}}{\pi^4 N^4}
    \right)^2
  .\end{displaymath}
  Now, for the middle term involving $c_{kj}$, we use Lemma~\ref{lem:sobolev}(b),(c)
  to obtain the upper bound
  \begin{displaymath}
    \left\| \cN \LdomainElement \right\|_\Lrange^2 \le 
    \sum_{k=1}^\nfunctions 
    \left( 
      \sum_{i=1}^\nscalars \frac{|\LdomainElementScalar_i| \left\| b_{ki} \right\|_{\hzerobar}}{\pi^2 N^2} 
      + \sum_{j=1}^\nfunctions \left[ 
          \frac{C_b C_e \| c_{kj} \|_{\htwo} \| \LdomainElementFun_j \|_{\htwobar}}{\pi^2 N^2}
        + \frac{|\gamma_{kj}| \left\| \LdomainElementFun_j \right\|_{\htwobar}}{\pi^4 N^4}
        \right]
    \right)^2
  .\end{displaymath}
  Noting that $N^{-4} \le N^{-2}$ and applying the Cauchy-Schwarz inequality
  we find the upper bound
  \begin{displaymath}
    \left\| \cN \LdomainElement \right\|_\Lrange^2 \le 
    \frac{2}{\pi^4 N^4}
    \sum_{k=1}^\nfunctions 
    \left(
      \left[ \sum_{i=1}^\nscalars |\LdomainElementScalar_i| \left\| b_{ki} 
      \right\|_{\hzerobar} \right]^2
      + \left[ \sum_{j=1}^\nfunctions
          \left(C_b C_e \| c_{kj} \|_{\htwo} + \frac{|\gamma_{kj}|}{\pi^{2}} \right)
          \left\| \LdomainElementFun_j \right\|_{\htwobar}
        \right]^2
    \right)
  .\end{displaymath}
  Since we are aiming for a bound in terms of $ \| \LdomainElement \|_\Ldomain $,
  we factor out the maximum coefficients of $|\LdomainElementScalar_i|$ and 
  $\| \LdomainElementFun_j \|_{\htwobar}$, respectively, and can replace the
  right-hand side in the above bound by
  \begin{displaymath}
    \frac{2}{\pi^4 N^4}
    \sum_{k=1}^\nfunctions 
    \left(
      \left[ \max_{1 \le i \le \nscalars} \left\| b_{ki} \right\|_{\hzerobar} 
      \sum_{i=1}^\nscalars |\LdomainElementScalar_i| \right]^2
      + \left[ \max_{1 \le j \le \nfunctions} \left(C_b C_e \| c_{kj} \|_{\htwo} 
      + \frac{|\gamma_{kj}|}{\pi^2}\right) \sum_{j=1}^\nfunctions
          \left\| \LdomainElementFun_j \right\|_{\htwobar}
        \right]^2
    \right)
  .\end{displaymath}
  Note that the innermost sums are now independent of $ k $, and we can repeat the previous step
  and apply the Cauchy-Schwarz inequality again to obtain the new bound
  \begin{displaymath}
    \left\| \cN \LdomainElement \right\|_\Lrange^2 \le
    \frac{2}{\pi^4 N^4}
    \left( \sum_{k=1}^\nfunctions N_k \right)
    \left[ 
      m \sum_{i=1}^\nscalars |\LdomainElementScalar_i|^2
      + n \sum_{j=1}^\nfunctions
          \left\| \LdomainElementFun_j \right\|_{\htwobar}^2
    \right] \; ,
  \end{displaymath}
  from which $\|\cN \LdomainElement\|_\Lrange \le C_\cN \|\LdomainElement\|_\Ldomain$
  follows easily with
  \begin{displaymath}
    C_\cN \; := \; \frac{\sqrt{2 \max\{m,n\}}}{\pi^2 N^2}
    \left( \sum_{k=1}^\nfunctions N_k \right)^{1/2} \; .
  \end{displaymath}
  This concludes the proof of the second bound with $ B := C_T C_\cN $.
  The bound for $ \| I-SL \| $ is a direct result of the bound derived here and Equation~\eqref{eqn:imsl}.
\end{proof}
We also know from Lemma \ref{lem:rho2} that $\| S \|_{\cL(\Lrange, \Ldomain)} \le \max(K_N,C_T)$.
Therefore, we can directly apply Proposition~\ref{prop:neumann} with
the constants $\rho_1 = \sqrt{A^2 + B^2} \le \tau < 1$ and
$\rho_2 = \max(K_N,C_T)$, and this immediately implies that the
operator~$L$ is one-to-one, onto, and the norm of its
inverse operator is bounded via
\begin{displaymath}
  \left\| L^{-1} \right\|_{\cL(\Lrange,\Ldomain)} \; \le \;
  \frac{\rho_2}{1-\rho_1} \; = \;
  \frac{\max(K_N,C_T)}{1-\tau} \; .
\end{displaymath}
This completes the proof of Theorem~\ref{thm:k}.
\let\subspaceSymbol\undefined%
\let\subspaceIndex\undefined%
\let\indexIndex\undefined%
\let\nscalars\undefined%
\let\nfunctions\undefined%
\let\basisSymbol\undefined%
\let\htwobar\undefined%
\let\hmtwobar\undefined%
\let\hzerobar\undefined%
\let\indexsetSymbol\undefined%
\let\embeddingConstantToHtwobar\undefined%
\let\xCoeff\undefined%
\let\yCoeff\undefined%
\let\Ldomain\undefined%
\let\LdomainElement\undefined%
\let\Lrange\undefined%
\let\LrangeElement\undefined%
\let\suitableFunctionSpace\undefined%
\section{Conclusions and future applications}
\label{sec:future}
In this paper, we have developed a framework for establishing a 
rigorous bound for the operator norm of the inverse of a general
type of linear fourth-order elliptic operator that occurs in a large
class of systems, such as for example in the context of materials 
science applications. We have then applied this framework to the
triblock copolymer model, a three monomer version of the Ohta-Kawasaki
equation. In particular, we have validated a series of equilibrium
solutions in spatial dimensions one and two.

The strength of the constructions developed here are their flexibility.
For example, with only minor modifications to the parameters, we 
have been able to use the same construction to validate pitchfork 
bifurcation points for the diblock copolymer 
equation~\cite{rizzi:sander:wanner:p22a}.  Additionally, with
relatively little additional effort, we will be able to use the
same method for rigorous pseudo-arclength continuation methods
for phase field materials models such as the Cahn-Hilliard, the
Cahn-Morral, or the classical Ohta-Kawasaki systems. While we
anticipate that there are still issues that will need to be
addressed, such as incorporating preconditioning and sparseness into our
construction, the generality of our approach means that as we
address such considerations, one will not have to revisit them
again for each separate system and dynamical question. 

\section*{Acknowledgments}
%
%
The research of E.S.\ and T.W.\ was partially supported by the
Simons Foundation under Awards~636383 and~581334, respectively.
We are thankful for the careful and useful comments of the anonymous referees.
%
%

\addcontentsline{toc}{section}{References}
\footnotesize
%
%
\bibliography{wanner1a,wanner1b,wanner2a,wanner2b,wanner2c}
\bibliographystyle{abbrv}
\end{document}